\documentclass{article}
\usepackage{graphicx} 
\usepackage{lipsum}
\usepackage{lscape}
\usepackage{amsthm}
\usepackage{amssymb}
\usepackage{MnSymbol}
\newtheorem{theorem}{Theorem}[section]
\usepackage{pst-all}
\usepackage{pstricks}    
\usepackage{graphicx} 
\usepackage{xfp}
\usepackage{tabularx}
\usepackage{multirow}
\usepackage{tikz}
\usepackage{rotating}
\usetikzlibrary{shapes.geometric}
\usepackage[utf8]{inputenc}
\usepackage{ragged2e}
\usepackage{calculator}
 \usepackage{calculus}
 \usepackage{hyperref}
 \usepackage{lscape}
\newtheorem{example}[theorem]{Example}
\newtheorem{prop}[theorem]{Proposition}
\theoremstyle{definition}
\theoremstyle{remark}
\newtheorem{remark}[theorem]{Remark}
\newtheorem{lemma}[theorem]{Lemma}
\newtheorem{definition}[theorem]{Definition}
\newtheorem{corollary}[theorem]{Corollary}
\usepackage{tkz-euclide}
\usepackage{pgfplots}

\DeclareMathOperator{\lcm}{lcm}
\usepgflibrary{shapes.geometric} 
\usepgflibrary[shapes.geometric] 
\usetikzlibrary{shapes.geometric} 
\usetikzlibrary[shapes.geometric] 
\usetikzlibrary{shapes,decorations} 
\usepackage{graphicx} 
\NewDocumentCommand{\showcalculation}{o m}{$
  \IfValueTF{#1}
    {#1}{#2} = \fpeval{#2}$}
  \renewenvironment{abstract}
 {\par\noindent\textbf{\abstractname.}\ \ignorespaces}
 {\par\medskip}
  \begin{document}
\title{Some determinants and relations in Heronian friezes}
\author{Anja \v{S}neperger}
\date{}
\newcommand{\Addresses}{{
  \bigskip
  \footnotesize
   \textsc{School of Mathematics \\ University of Leeds \\ Leeds, LS2 9JT \\ United Kingdom\\}\par\nopagebreak \noindent
  \textit{E-mail address}: \texttt{mmasn@leeds.ac.uk}}}
  \maketitle
\begin{abstract}
\noindent
   In this article, we give algebraic relations and determinant vanishing equalities that hold for the entries of a single Heronian diamond of a Heronian frieze arising from a cyclic $n$-gon. We also give algebraic relations that hold between entries of multiple adjacent diamonds of such a frieze. Furthermore, we define a plane Heronian frieze, and establish some more determinant vanishing equalities for the entries of a plane Heronian frieze arising from a cyclic $n$-gon, where $n$ is a positive integer divisible by $4$.
  \end{abstract}
\tableofcontents
\section*{Introduction} Frieze patterns of integers were introduced by Coxeter in \cite{key7}, and classified by Conway and Coxeter in \cite{key8} in terms of triangulations of polygons. Broline, Crowe and Isaacs explicitly computed the determinant of the matrix associated to a frieze in \cite{key9}. This result was generalized by Baur and Marsh \cite{key10} for friezes whose entries are cluster variables from a cluster algebra of type $A$. Half a century later, Sergey Fomin and Linus Setiabrata, motivated by computational geometry of point configurations in the Euclidean plane and by the theory of cluster algebras of type $A$, introduced Heronian friezes \cite{fs}. The main idea they used as a motivation is the fact that the measurements related to triangulations of a plane quadrilateral satisfy some equations arising from classical geometry results. One of them, which the friezes were named after, is a version of Heron's formula for the area $T$ of the triangle whose side lengths are $a,b,c$ -- namely $T = \sqrt{s(s-a)(s-b)(s-c)}$, where $s$ is the semiperimeter of the triangle. They prove that a generic Heronian frieze, just like the frieze patterns introduced by Coxeter, possesses glide symmetry. Furthermore, they establish a version of the Laurent phenomenon for these friezes. For more information on this phenomenon for cluster algebras, the reader is referred to \cite{key5}, and for similar results for Coxeter-Conway friezes, to \cite{key4}. \\
In this paper, we consider Heronian friezes arising from a polygon. We state and prove some algebraic relations, as well as determinant vanishing equalities for the entries of a Heronian frieze arising from a cyclic polygon. The paper is organized as follows. \\ In Section 1 we give an overview of some definitions and results concerning Heronian friezes from \cite{fs}. On top of that, we give a pattern of diamond labelings for a Heronian frieze arising from a polygon in Figure \ref{fig5}, and introduce a notion of a plane Heronian frieze  \eqref{def111}. \\ In Section 2 we focus on Heronian friezes arising from cyclic polygons. Using a result of Smith \cite[Section 1]{key2} we prove some algebraic relations that hold between the entries of such a frieze. Then, we use them to establish some determinant vanishing equalities for a single diamond of the frieze. For similar results on Coxeter-Conway and higher $SL_{k}$  tame frieze patterns, see \cite{key4} and  \cite{key6}. Next, we use a result by Gregorac \cite[Theorem 10]{key3} to establish further algebraic relations between entries of several adjacent diamonds of the frieze, as shown in Figure  \ref{figg,pattern}. The last result of the Section 2 states that, given a cyclic $n$-gon, where $n$ is a positive integer divisible by $4$, some $\frac{n}{2}$ x $\frac{n}{2}$ determinants of the corresponding plane Heronian frieze vanish as well. \\ For the main result of the Section 3, we use Theorem \ref{thm009}, which states that the alternating sum of the reciprocal products of certain chord lengths of a cyclic polygon equals zero.  Namely, we use the fact that the area $T$ of a triangle with side lengths $a,b,c$ and a radius of the circumscribed circle $R$ can be computed using the formula $T = \frac{abc}{4R}$, in order to transform the equality from Theorem \ref{thm009} into an equality involving entries of the Heronian frieze. We get that, given a cyclic $n$-gon, where $n>4$ is an even number, an alternating sum of the form  $\sum_{m=1}^{n}(-1)^{m+1}x(m)S(m)$ equals $0$, where $x(m)$ is a monomial in the squared lengths of the boundary edges of the polygon, and $S(m)$ is a monomial in $4$ times the signed areas of some of the triangles of the polygon. An example of the formula for $n=6$ is 
\begin{equation*}
\begin{split}
    & x_{12}x_{45}S_{234}S_{256}S_{356}S_{456} - x_{12}x_{45}S_{134}S_{156}S_{356}S_{456} + \\ & \qquad \qquad x_{12}x_{45}S_{124}S_{156}S_{256}S_{456} - x_{12}x_{45}S_{123}S_{156}S_{256}S_{356} + \\ & \qquad \qquad x_{45}x_{56}S_{123}S_{124}S_{126}S_{346} - x_{12}x_{56}S_{123}S_{145}S_{245}S_{345} = 0,
    \end{split}
\end{equation*}
where the $x_{ij}$ are squared distances between the vertices of the polygon, and $S_{ijk}$ are four times the signed areas of the corresponding triangles.
\section{Definition and some properties of Heronian friezes}
In this section, we give an overview of definitions and main results from \cite{fs}. We also give a pattern of diamond labelings for a Heronian frieze arising from a polygon, as well as define a plane polygonal Heronian frieze, which will be considered in one of the results of Section 2.
\begin{definition}\cite[Definition 2.3]{fs}\label{def:def100}
A \textit{Heronian diamond} is an ordered $10$-tuple of complex numbers $(a,b,c,d,e,f,p,q,r,s)$ satisfying the following equations:\begin{equation} p^2=H(b,c,e) \end{equation}
\begin{equation} q^2=H(a,d,e)\end{equation}
\begin{equation} r^2=H(a,f,b)\end{equation}
\begin{equation} s^2=H(c,f,d)\end{equation}
\begin{equation} r+s=p+q \end{equation}
\begin{equation} 4ef=(p+q)^2+(a-b+c-d)^2\end{equation}
\begin{equation} e(r-s)=p(a-d)+q(b-c),\end{equation} where $$
 H(x,y,z):= -x^2-y^2-z^2+2xy+2xz+2yz.$$ Instead of listing the components of a Heronian diamond as a row of $10$ numbers, we will typically arrange them in a diamond pattern as in \cite{fs} and shown in the Figure \ref{fig1}.
\begin{figure}
\begin{center}
\begin{tikzpicture}
 \node[draw=none](n1) at (0,0) {$a$};
 \node[draw=none](n2) at (1,-1) {$r$};
 \node[draw=none](n3) at (2,-2) {$f$};
 \node[draw=none](n4) at (1,-3) {$s$};
 \node[draw=none](n5) at (0,-4) {$c$};
 \node[draw=none](n6) at (-1,-3){$p$};
 \node[draw=none](n7) at (-2,-2){$e$};
 \node[draw=none](n8) at (-1,-1){$q$};
 \node[draw=none](n9) at (-2,-4){$b$};
 \node[draw=none](n10) at (2,-4){$d$};
 \node[draw=none](n11) at (-2,0){$d$};
 \node[draw=none](n12) at (2,0){$b$};
 
 \draw[-] (n1) to (n2);
 \draw[-] (n2) to (n3);
 \draw[-] (n3) to (n4);
 \draw[-] (n4) to (n5);
 \draw[-] (n5) to (n6);
 \draw[-] (n6) to (n7);
 \draw[-] (n7) to (n8);
 \draw[-] (n8) to (n1);
 \draw [ thick , densely dashed ] (n2) -- (n6);
 \draw[thick,densely dashed] (n4) -- (n8);
 \draw[thick,densely dashed] (n2) -- (n12);
 \draw[thick,densely dashed] (n6) -- (n9);
 \draw[thick,densely dashed] (n8) -- (n11);
 \draw[thick,densely dashed] (n4) -- (n10);
\end{tikzpicture}
\end{center}  
\caption {A Heronian diamond. Here, $b$ and $d$ are associated to the dashed lines extending the bimedians of the diamond. The remaining eight numbers are placed at the vertices of the diamond and at the midpoints of its sides. }
\label{fig1}
\end{figure}
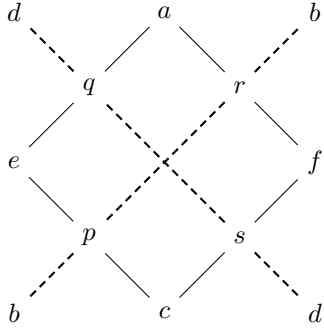
\label{fig,diamond}
\label{def}
\end{definition}
\noindent Motivated by Definition \ref{def}, in \cite{fs} the authors introduced the definition of a \textit{Heronian frieze}: roughly speaking, a Heronian frieze is a collection of complex numbers arranged in the below shown pattern, satisfying the Heronian diamond equations for all the diamonds in the pattern (plus some boundary conditions). 
\begin{center}

\includegraphics[scale=0.2]{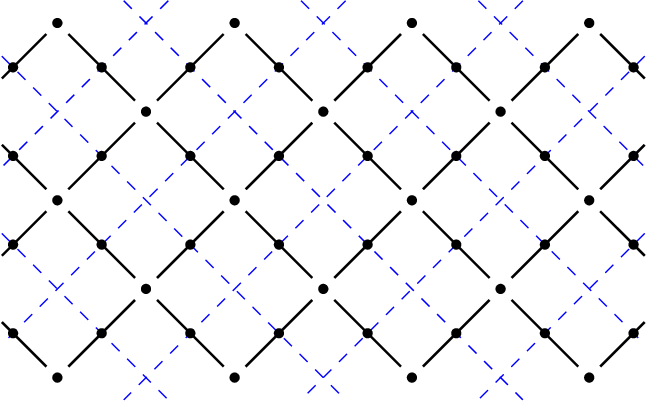}
\end{center}
Before passing on to a formal definition of Heronian frieze, let us first introduce a definition of a  labeled polygon, and some more motivation behind the actual notion of a Heronian diamond.
\begin{definition} Let $n \geq 3$ be an integer. Then a labeled \textit{polygon} (specifically an $n$-gon) in the complex plane $A$ is an ordered $n$-tuple of vertices $P = (A_1, A_2, ..., A_n) \in A^n$.
    \end{definition}
    \noindent
 Let $(A_1, A_2, ..., A_n)$ be a  labeled polygon in the complex plane. \\ Then let $x_{ij}$ denote the squared distance between points $A_i$ and $A_j$, and $S_{ijk}$ denote four times the signed area of the triangle $A_iA_jA_k$, where $i,j, k$ $\in\{1,2,...,n\}$. \\ Namely, if $A_m = (x_m,y_m)$, for $m \in \{i,j,k\}$, then (as defined in \cite{fs})
\begin{equation}
x_{ij}: = (x_j - x_i)^2 + (y_j - y_i)^2, \label{eqx}
\end{equation} 
and 
\begin{equation}S_{ijk} := 2[(x_j - x_i)(y_k - y_i) - (y_j - y_i)(x_k - x_i)]\label{eqs}.
\end{equation}
If  $i \leq 0$, $j \leq 0$, or $k \leq 0$,  we adopt a convention that $S_{ijk} = x_{ij} = 1$.
\begin{remark}
When the vertices $A_i, A_j, A_k$ are ordered anticlockwise, $S_{ijk}$ is positive, i.e. equal to four times the actual area of the triangle. Moreover, the following equalities hold in addition \cite{fs}: 
\begin{equation}
S_{ijk} = - S_{ikj} = - S_{jik} = S_{jki} = S_{kij} = - S_{kji}
\end{equation}
\label{rem13}
\end{remark}
\noindent As explained in \cite{fs}, the motivation for the definition of a Heronian diamond actually came from triangulating a quadrilateral. \\ Namely, a quadrilateral $(A_1, A_2, A_3, A_4)$ has two triangulations, involving diagonals $A_1A_3$ and $A_2A_4$, respectively. 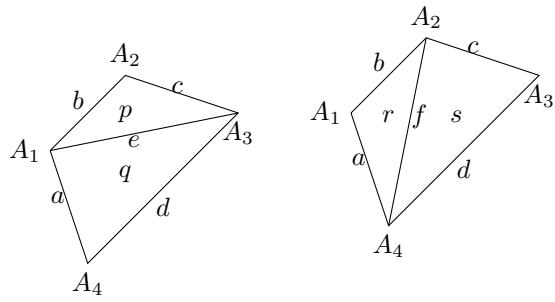
\begin{figure}
\begin{center}
\begin{tikzpicture}
 \coordinate[label=left:$A_1$] (p) at (0,0);
 \coordinate[label=above:$p$](y) at (1,0.25);
 \coordinate[label=above:$q$](h) at (1,-0.55);
 \coordinate[label=above:$A_2$] (q) at (1,1);
 \coordinate[label=below:$A_3$] (r) at (2.5,0.5);
 \coordinate[label=below:$A_4$] (s) at (0.5,-1.5);
 \tkzLabelSegment[above left = -2pt and -2 pt](p,q){$b$}
 \tkzLabelSegment[above left = -3pt and -4 pt](q,r){$c$}
 \tkzLabelSegment(r,s){$d$}
 \tkzLabelSegment[above left = -2pt and -2 pt](s,p){$a$}
 \tkzLabelSegment[below left= -2 pt and -2 pt](r,p){$e$}
 \draw (p)--(q)--(r)--(s)--(p);
 \draw (p)--(r);
\coordinate[label=left:$A_1$] (a) at (4,0.5);
 \coordinate[label=above:$A_2$] (b) at (5,1.5);
 \coordinate[label=below:$A_3$] (c) at (6.5,1);
 \coordinate[label=below:$A_4$] (d) at (4.5,-1);
 \tkzLabelSegment[above left=-2 pt and -2pt](a,b){$b$}
 \tkzLabelSegment[above left=-2 pt and -2 pt](b,c){$c$}
 \tkzLabelSegment(c,d){$d$}
 \tkzLabelSegment[above left=-2 pt and -2 pt](d,a){$a$}
 \tkzLabelSegment[above right=-2 pt and -2 pt](b,d){$f$}
 \coordinate[label=above:$r$](t) at (4.5,0.25);
 \coordinate[label=above:$s$](u) at (5.4,0.25);
 \draw (a)--(b)--(c)--(d)--(a);
 \draw (b)--(d);
\end{tikzpicture}
\end{center}
\caption{Two triangulations of a plane quadrilateral}
\label{fig2}
\end{figure} 
Figure \ref{fig2} shows these two triangulations, along with their respective measurement data:
\begin{equation} a = x_{14}, b = x_{12}, c = x_{23}, d = x_{34}, e = x_{13}, f = x_{24}, \end{equation}
\begin{equation} p = S_{123}, q = S_{134}, r = S_{124}, s = S_{234}.\end{equation}
Fomin and Setiabrata have shown \cite[Prop. 2.8]{fs} that the measurements (11) and (12) satisfy the equations (1) -- (7). Hence, the measurements related to the two triangulations of a $4$-gon give rise to a Heronian diamond, and represent a geometrical motivation for introducing this notion in general. Using this fact and previously introduced notation for squared distances and four times signed areas, a Heronian diamond for a quadruple of vertices $A_i, A_j, A_k, A_l$ can be represented as in Figure \ref{fig3}.
 
 \begin{figure}
\begin{center}
\begin{tikzpicture}
 \node[draw=none](n1) at (0,0) {$x_{il}$};
 \node[draw=none](n2) at (1,-1) {$S_{ijl}$};
 \node[draw=none](n3) at (2,-2) {$x_{jl}$};
 \node[draw=none](n4) at (1,-3) {$S_{jkl}$};
 \node[draw=none](n5) at (0,-4) {$x_{jk}$};
 \coordinate[label=${\mathbf{ijkl}}$] (c1) at (0,-2.4);
 \node[draw=none](n6) at (-1,-3){$S_{ijk}$};
 \node[draw=none](n7) at (-2,-2){$x_{ik}$};
 \node[draw=none](n8) at (-1,-1){$S_{ikl}$};
 \node[draw=none](n9) at (-2,-4){$x_{ij}$};
 \node[draw=none](n10) at (2,-4){$x_{kl}$};
 \node[draw=none](n11) at (-2,0){$x_{kl}$};
 \node[draw=none](n12) at (2,0){$x_{ij}$};
 
 \draw[-] (n1) to (n2);
 \draw[-] (n2) to (n3);
 \draw[-] (n3) to (n4);
 \draw[-] (n4) to (n5);
 \draw[-] (n5) to (n6);
 \draw[-] (n6) to (n7);
 \draw[-] (n7) to (n8);
 \draw[-] (n8) to (n1);
 \draw [densely dashed ] (n2) -- (n6);
 \draw[densely dashed] (n4) -- (n8);
 \draw[densely dashed] (n2) -- (n12);
 \draw[densely dashed] (n6) -- (n9);
 \draw[densely dashed] (n8) -- (n11);
 \draw[densely dashed] (n4) -- (n10);
\end{tikzpicture}
\end{center} 
\caption{A Heronian diamond for a quadruple of vertices $A_i, A_j, A_k, A_l$}
\label{fig3}
\end{figure}
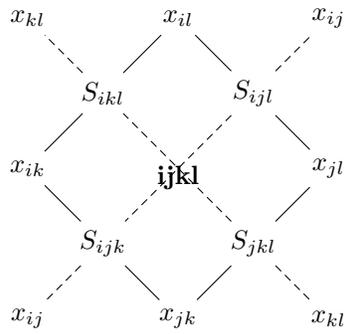 
\noindent Motivated by Definition ~\ref{def:def100}, Fomin and Setiabrata also introduce a notion of a \textit{Heronian frieze}. Here, we give a  slightly simplified version of the definition given in \cite{fs}.
\begin{definition}
A \textit{Heronian frieze of order $n$} is a collection of complex numbers arranged in a pattern shown in Figure \ref{friz}, such that every diamond of the pattern is Heronian. In addition, we impose the boundary conditions:
\begin{equation}
    z_{ii} = \tilde{z}_{iji} = \tilde{z}_{iij} = \tilde{z}_{ijj} = 0,
    \label{bound}
\end{equation}
for $i,j$ distinct elements of $\{1,2,...,n\}$.
\label{frizz}
 \end{definition}
 If all the entries associated to the dashed lines of the Heronian frieze of order $n$ are equal, i.e. $z_{12} = z_{23} = z_{34} = ... = z_{(n-1)n} = z_{n1}$, we say that the frieze is \textit{equilateral} \cite{fs}.
\begin{remark}Note that every entry of a Heronian frieze of order $n$ is in one of the forms $z_{ij}$, $\tilde{z}_{i,i+1,j}$, or $\tilde{z}_{i,j,j+1}$, where $i, j \in \{1,2,...,n\}$, and addition is modulo $n$.
\label{napomena}
\end{remark}
\begin{figure}
\begin{center}
\begin{tikzpicture}[xscale=1,yscale=1.3]
  \node[red,very thick](n1) at (4,4) {$z_{11}$};
   \node[blue,very thick](n2) at (4.5,3.5) {$\tilde{z}_{121}$};
   \node[draw=none] (n123098786787) at (2,3.5) {$\ddots$};
   \node[draw=none] (n51621712) at (5.2,4.5) {$z_{12}$};
    \draw[densely dashed] (n2) -- (n51621712);
   \node[draw=none] (n51621712) at (4.5,4.5) {$z_{12}$};

   \node[draw=none] (n5262722) at (7.5,4.5) {$\cdots$};
    \node[red,very thick](n3) at (5,3) {$z_{21}$};
    \node[blue,very thick](n5) at (5.5,2.5) {$\tilde{z}_{231}$};
    \node[draw=none] (n123787665) at (3,2.5) {$\ddots$};
    \node[red,very thick](n4) at (6,2) {$z_{31}$};
    \node[blue,very thick](n5) at (6.5,1.5) {$\tilde{z}_{341}$};
    \node[red,very thick](n9) at (8,0) {$z_{n1}$};
    \node[draw=none](n12) at (7,1) {$\ddots$};
    \node[draw=none](n314151) at (11,1.5) {$\ddots$};
    \node[red,very thick](n25151) at (9.5,2){$\cdots$};
    \node[red,very thick](n25151) at (11.5,0){$\cdots$};
    \node[blue,very thick](n10) at (7.5,0.5) {$\tilde{z}_{(n-1)n1}$};
    \node[draw=none] (n1230987) at (5,0.5) {$\ddots$};
    \node[blue,very thick](n7) at (8.5,-0.5) {$\tilde{z}_{n11}$};
    \node[draw=none] (n4155116) at (7.5, -1.5) {$z_{n1}$};
     \draw[densely dashed] (n4155116) -- (n7);
    \node[red,very thick](n8) at (9,-1) {$z_{11}$};
    \node[draw=none] (n14151) at (9.7,-1.5) {$z_{12}$};
   
    \node[red,very thick](n13) at (6,4) {$z_{22}$};
    \node[red,very thick](n805439) at (8,4) {$\cdots$};
    \node[blue,very thick](n8043879) at (8.5,3.5) {$\cdots$};
    \node[blue,very thick](n17) at (6.5,3.5) {$\tilde{z}_{232}$};

    \node[blue,very thick](n14) at (10,1) {$\ddots$};

    \node[red,very thick](n14) at (7,3) {$z_{32}$};
    \node[blue,very thick](n87) at (7.5,2.5) {$\tilde{z}_{342}$};
  
    \node[red,very thick](n8059) at (10,4) {$z_{nn}$};
    \node[red,very thick](n1) at (12,4) {$z_{11}$};
   \node[blue,very thick](n2) at (12.5,3.5) {$\tilde{z}_{121}$};
    \node[red,very thick](n3) at (13,3) {$z_{21}$};
    \node[blue,very thick](n5) at (13.5,2.5) {$\tilde{z}_{231}$};
    \node[red,very thick](n4) at (14,2) {$z_{31}$};
    \node[blue,very thick](n5) at (14.5,1.5) {$\tilde{z}_{341}$}; 
    \node[draw=none] (n1230987) at (15,3.5) {$\ddots$};
    \node[draw=none] (n1230988767) at (16,2.5) {$\ddots$};
    \node[draw=none] (n1230987) at (18,0.5) {$\ddots$};
    \node[blue,very thick](n8049) at (10.5,3.5) {$\tilde{z}_{\tiny n1n}$};
    \node[draw=none] (n151661) at (11.4,4.5) {$z_{n1}$};
     \draw[densely dashed] (n151661) -- (n8049);
    \node[draw=none] (n16161717) at (13.6,4.5) {$z_{12}$};
     \draw[densely dashed] (n16161717) -- (n2);
    \node[draw=none] (n151661) at (10.6,4.5) {$z_{n1}$};
    \node[blue,very thick] (n1230987) at (11.5,3.5) {$\tilde{z}_{1n1}$};
       \draw[densely dashed] (n151661) -- (n1230987);
    \node[blue,very thick] (n123098877) at (12.5,2.5) {$\tilde{z}_{2n1}$};
       \draw[densely dashed] (n1230987) -- (n123098877);
    \node[blue,very thick] (n1230976487) at (13.5,1.5) {$\tilde{z}_{3n1}$};
       \draw[densely dashed] (n123098877) -- (n1230976487);
    \node[draw=none] (n123098721) at (15,1) {$\ddots$};
    
    \node[red,very thick](n809) at (11,3) {$z_{1n}$};
    \node[blue,very thick](n80129) at (11.5,2.5) {$\tilde{z}_{12n}$};
     \draw[densely dashed] (n80129) -- (n2);
    \node[red,very thick](n805649) at (12,2) {$z_{2n}$};
    \node[blue,very thick](n555) at (12.5,1.5) {$\tilde{z}_{23n}$};
    \node[draw=none](n4532) at (13,1) {$\ddots$};
    \node[blue,very thick](n809) at (13.5,0.5) {$\tilde{z}_{(n-2)(n-1)n}$};
    \node[blue,very thick] (n1230987) at (15.5,0.5) {$\tilde{z}_{(n-1)n1}$};
    \node[red,very thick] (n1230987) at (16,0) {$z_{n1}$};
    \node[blue,very thick] (n1230987) at (16.5,-0.5) {$\tilde{z}_{n11}$};
    \node[red,very thick] (n1230987647) at (17,-1) {$z_{11}$};
    
    \node[red,very thick](n8077) at (14,0) {$z_{(n-1)n}$};
    \node[blue,very thick](n80769) at (14.5,-0.5) {$\tilde{z}_{(n-1)nn}$};
     \node[blue,very thick] (n1230976287) at (15.5,-0.5) {$\tilde{z}_{nn1}$};
    \node[draw=none] (n51516) at (16.4,-1.5)  {$z_{n1}$};
    \node[draw=none] (n15627171) at (18.5,-1.5) {$\cdots$};
\node[draw=none] (n51516543) at (15.7,-1.5)  {$z_{n1}$};
 \draw[densely dashed] (n1230987) -- (n51516543);
   \draw[densely dashed] (n1230976287) -- (n51516);
    \node[red,very thick](n809) at (15,-1) {$z_{nn}$};
    
    \node[red,very thick](n14) at (8,2) {$z_{42}$};
    \node[blue,very thick](n14) at (8.5,1.5) {$\tilde{z}_{452}$};
    \node[draw=none](n22) at (9,1) {$\ddots$};
    \node[blue,very thick](n24) at (9.5,0.5) {$\tilde{z}_{n12}$};
    \draw[densely dashed] (n7) -- (n24); 
    \node[red,very thick](n65) at (10,0) {$z_{12}$};
    \node[blue,very thick](n32) at (10.5,-0.5) {$\tilde{z}_{122}$};
     \draw[densely dashed] (n14151) -- (n32);
    \node[red,very thick](n90) at (11,-1) {$z_{22}$};
    \node[blue,very thick](n90) at (5.5,3.5) {$\tilde{z}_{212}$};
       \draw[densely dashed] (n51621712) -- (n90);
    \node[blue,very thick](n19) at (6.5,2.5) {$\tilde{z}_{312}$};
    \draw[densely dashed](n90)--(n19);
    \node[blue,very thick](n14) at (7.5,1.5) {$\tilde{z}_{412}$};
    \draw[densely dashed](n14)--(n19);
    \node[blue,very thick](n14) at (8,1) {$\ddots$};
    \node[blue,very thick](n18) at (8.5,0.5) {$\tilde{z}_{n12}$};
    \node[blue,very thick](n99) at (9.5,-0.5) {$\tilde{z}_{112}$};
    \draw[densely dashed](n18)--(n99);
    \node[draw=none] (n151616) at (10.4,-1.5) {$z_{12}$};
    \draw[densely dashed](n99)--(n151616);
    \node[draw=none] (n92592962) at (13,-1.5) {$\cdots$};
    \node[red,very thick](n14) at (7,3) {$z_{32}$};
\end{tikzpicture}
\end{center}
\caption{Heronian frieze of order $n$}
\label{friz}
\end{figure}
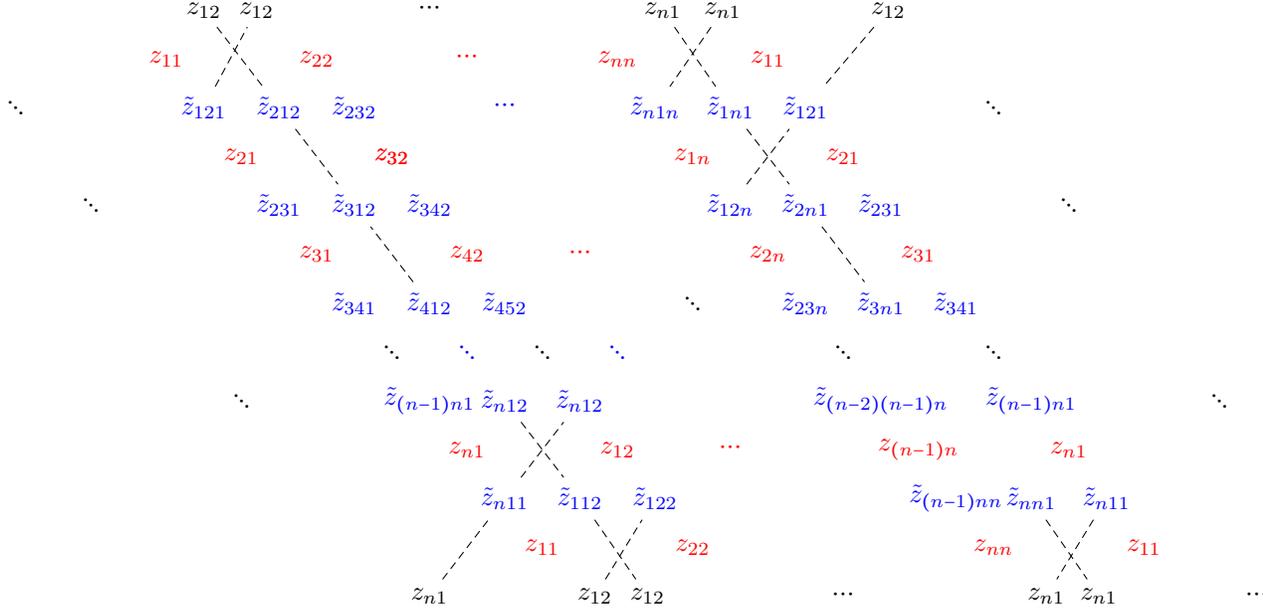

Using the fact that measurements related to triangulations of a quadrilateral give rise to a Heronian diamond, the authors of \cite{fs} also state the following.
\begin{prop}\cite{fs}
 Any $n$-gon $P$ in the complex plane gives rise to a Heronian frieze of order $n$, as given in the Figure \ref{friz}, in the following way:
\begin{equation*}
    z_{ij}:= x_{ij},
    \end{equation*}
    \begin{equation*}
    \tilde{z}_{i,i+1,j}:= S_{i,i+1,j},
\end{equation*}
\begin{equation*}
    \tilde{z}_{i,j,j+1}:= S_{i,j,j+1},
\end{equation*}
where $i, j \in \{1,2,...,n\}$, addition is modulo $n$, and the $x_{ij}$, $S_{i,i+1,j}$ and $S_{i,j,j+1}$ are as in \eqref{eqx} and \eqref{eqs}.\\
(Boundary conditions \eqref{bound} hold since the squared distance between a vertex and itself equals zero. Similarly, the signed area of a triangle with two equal vertices is zero.)
\label{prop15}
\end{prop}
\begin{definition}
 Given an $n$-gon $P = (A_1,A_2,...,A_n)$, we define a \textit{polygonal Heronian frieze of order $n$} to be a Heronian frieze as in Proposition \ref{prop15}. Such a frieze is shown in the Figure \ref{fig6}.
\label{def16}
\end{definition}
\begin{remark}
Every diamond of a polygonal Heronian frieze corresponds to a choice of a quadruple of, not necessarily distinct, vertices of a corresponding $n$-gon. The diamonds of the first and the last row of the frieze correspond to quadruples where two of the four vertices coincide. In accordance with Figure \ref{fig3}, we give the pattern of diamond labels of a polygonal Heronian frieze of order $n$ in  Figure \ref{fig5}. 
\end{remark}
\begin{sidewaysfigure}
    \centering
\begin{tikzpicture}
    \node[diamond,draw,minimum width = 3.5cm,
    minimum height = 4cm] (d1) at (1.25,-1.5) {\tiny $12n1$};
    \node[diamond,draw,minimum width = 3.5cm,
    minimum height = 4cm] (d2) at (3,-3.5) {\tiny $23n1$};
     \node[draw=none] (n12378732665) at (0.5,-4) {$\cdots$};
    \node[draw=none] (n123787665) at (5,-5.5) {$\ddots$};
    \node[draw=none] (n123787665) at (4.3,-4.7) {$\ddots$};
     \node[diamond,draw,minimum width = 3.5cm,
    minimum height = 4cm] (d3) at (6.5,-7.5) {\tiny $(n-2)(n-1)n1$};
     \node[draw=none] (n123787665) at (3,-8) {$\cdots$};
   \node[diamond,draw,minimum width = 3.5cm,
    minimum height = 4cm] (d4) at (8.25,-9.5) {\tiny $(n-1)nn1$}; 
   \node[diamond,draw,minimum width = 3.5cm,
    minimum height = 4cm] (d5) at (4.75,-1.5) {\tiny $2312$};  
     \node[draw=none] (n12378766576) at (7.75,-2) {$\cdots$};
    \node[diamond,draw,minimum width = 3.5cm,
    minimum height = 4cm] (d6) at (6.5,-3.5) {\tiny $3412$};
     \node[draw=none] (n123787665) at (9.3,-4) {$\cdots$};
     \node[draw=none] (n1237874665) at (8.5,-5.5) {$\ddots$};   
     
    \node[diamond,draw,minimum width = 3.5cm,
    minimum height = 4cm] (d8) at (10,-7.5) {\tiny $(n-1)n12$}; 
     \node[diamond,draw,minimum width = 3.5cm,
    minimum height = 4cm] (d9) at (11.75,-9.5) {\tiny $n112$}; 
  \node[draw=none] (n12378766578) at (14.75,-10) {$\cdots$};   
  \node[diamond,draw,minimum width = 3.5cm,
    minimum height = 4cm] (d10) at (10.25,-1) {\tiny $n1(n-1)n$};
    \node[diamond,draw,minimum width = 3.5cm,
    minimum height = 4cm] (d11) at (12,-3) {\tiny $12(n-1)n$};
    \node[draw=none] (n625627) at (14,-5.5) {$\ddots$};
    \node[diamond,draw,minimum width = 3.5cm,
    minimum height = 4cm] (d16) at (15.5,-7) {\tiny $(n-3)(n-2)(n-1)n$};
     \node[diamond,draw,minimum width = 3.5cm,
    minimum height = 4cm] (d19) at (17.25,-9) {\tiny $(n-2)(n-1)(n-1)n$};
    \node[diamond,draw,minimum width = 3.5cm,
    minimum height = 4cm] (d106) at (13.75,-1) {\tiny $12n1$};
    \node[diamond,draw,minimum width = 3.5cm,
    minimum height = 4cm] (d107) at (15.5,-3) {\tiny $23n1$};
    \node[draw=none] (n123787665788) at (18.3,-2.5) {$\cdots$};  
    \node[draw=none] (n123787466576) at (17.5,-5) {$\ddots$};   
   \node[diamond,draw,minimum width = 3.5cm,
    minimum height = 4cm] (d107) at (19,-7) {\tiny $(n-2)(n-1)n1$}; 
    \node[draw=none] (n1237876566578) at (22.3,-8) {$\cdots$};  
    \node[diamond,draw,minimum width = 3.5cm,
    minimum height = 4cm] (d1077) at (20.75,-9) {\tiny $(n-1)nn1$};
\end{tikzpicture}
\caption{Diamond labeling pattern for a polygonal Heronian frieze of order $n$}
\label{fig5}
\end{sidewaysfigure}
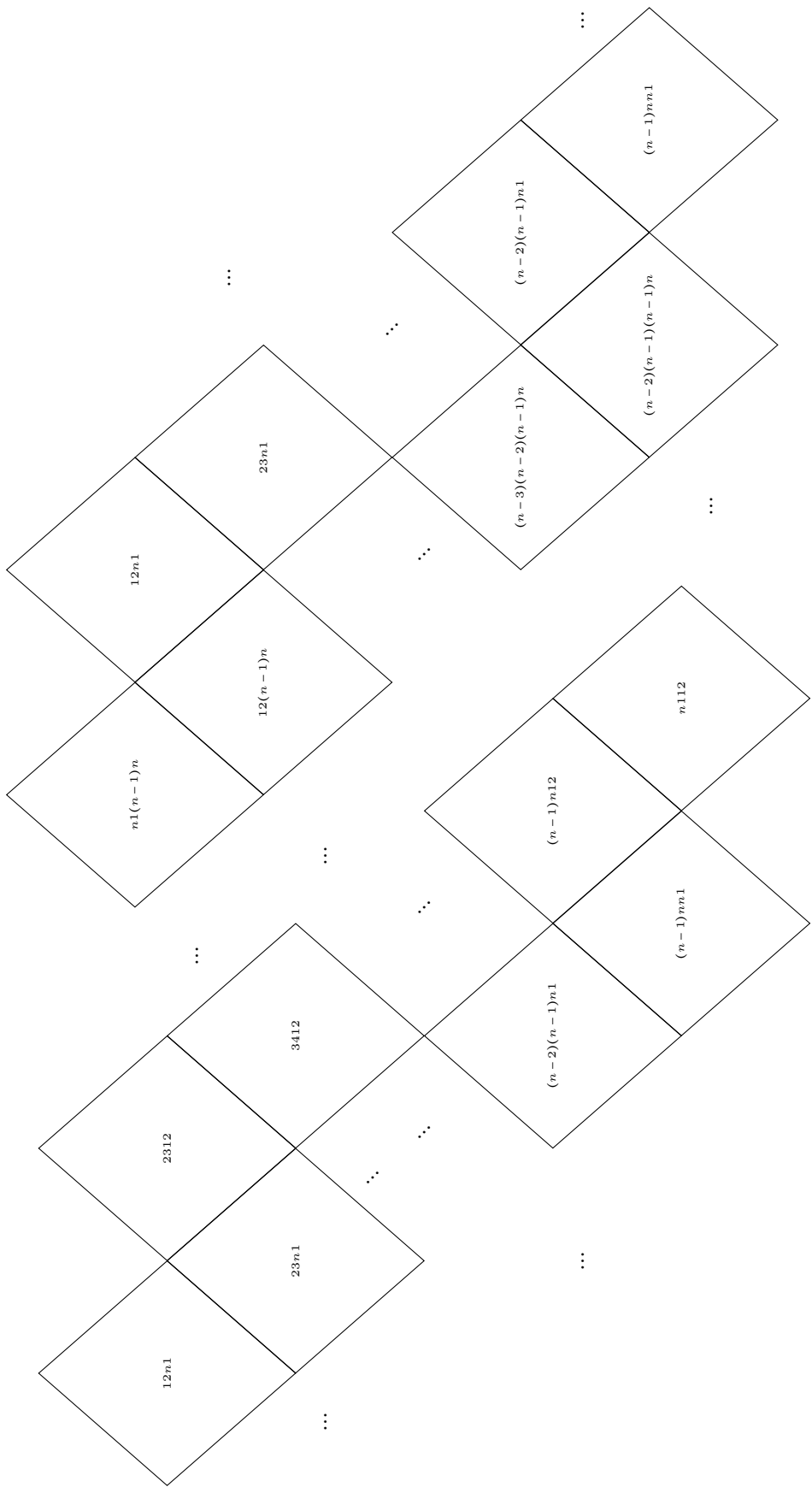
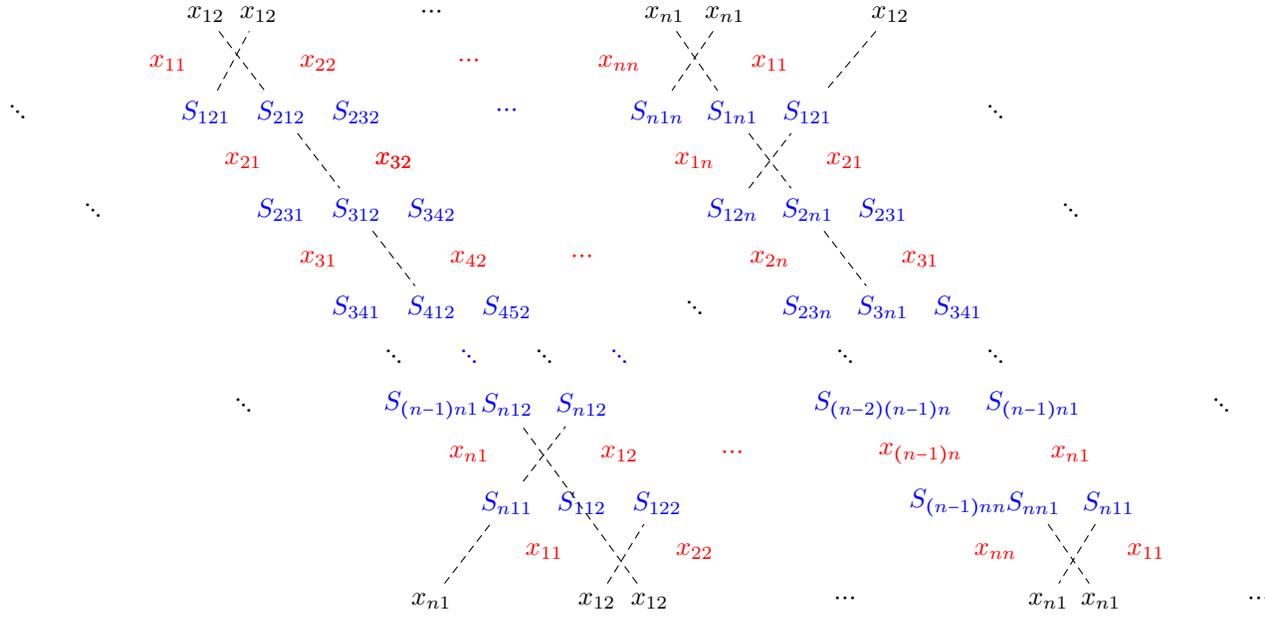
\begin{figure}
\begin{center}
\begin{tikzpicture}[xscale=1,yscale=1.3]
  \node[red,very thick](n1) at (4,4) {$x_{11}$};
   \node[blue,very thick](n2) at (4.5,3.5) {$S_{121}$};
   \node[draw=none] (n123098786787) at (2,3.5) {$\ddots$};
   \node[draw=none] (n51621712) at (5.2,4.5) {$x_{12}$};
    \draw[densely dashed] (n2) -- (n51621712);
   \node[draw=none] (n51621712) at (4.5,4.5) {$x_{12}$};

   \node[draw=none] (n5262722) at (7.5,4.5) {$\cdots$};
    \node[red,very thick](n3) at (5,3) {$x_{21}$};
    \node[blue,very thick](n5) at (5.5,2.5) {$S_{231}$};
    \node[draw=none] (n123787665) at (3,2.5) {$\ddots$};
    \node[red,very thick](n4) at (6,2) {$x_{31}$};
    \node[blue,very thick](n5) at (6.5,1.5) {$S_{341}$};
    \node[red,very thick](n9) at (8,0) {$x_{n1}$};
    \node[draw=none](n12) at (7,1) {$\ddots$};
    \node[draw=none](n314151) at (11,1.5) {$\ddots$};
    \node[red,very thick](n25151) at (9.5,2){$\cdots$};
    \node[red,very thick](n25151) at (11.5,0){$\cdots$};
    \node[blue,very thick](n10) at (7.5,0.5) {$S_{(n-1)n1}$};
    \node[draw=none] (n1230987) at (5,0.5) {$\ddots$};
    \node[blue,very thick](n7) at (8.5,-0.5) {$S_{n11}$};
    \node[draw=none] (n4155116) at (7.5, -1.5) {$x_{n1}$};
     \draw[densely dashed] (n4155116) -- (n7);
    \node[red,very thick](n8) at (9,-1) {$x_{11}$};

    \node[draw=none] (n14151) at (9.7,-1.5) {$x_{12}$};
   
    \node[red,very thick](n13) at (6,4) {$x_{22}$};
    \node[red,very thick](n805439) at (8,4) {$\cdots$};
    \node[blue,very thick](n8043879) at (8.5,3.5) {$\cdots$};
    \node[blue,very thick](n17) at (6.5,3.5) {$S_{232}$};

    \node[blue,very thick](n14) at (10,1) {$\ddots$};

    \node[red,very thick](n14) at (7,3) {$x_{32}$};
    \node[blue,very thick](n87) at (7.5,2.5) {$S_{342}$};
  
    \node[red,very thick](n8059) at (10,4) {$x_{nn}$};
    \node[red,very thick](n1) at (12,4) {$x_{11}$};
   \node[blue,very thick](n2) at (12.5,3.5) {$S_{121}$};
    \node[red,very thick](n3) at (13,3) {$x_{21}$};
    \node[blue,very thick](n5) at (13.5,2.5) {$S_{231}$};
    \node[red,very thick](n4) at (14,2) {$x_{31}$};
    \node[blue,very thick](n5) at (14.5,1.5) {$S_{341}$}; 
    \node[draw=none] (n1230987) at (15,3.5) {$\ddots$};
    \node[draw=none] (n1230988767) at (16,2.5) {$\ddots$};
    \node[draw=none] (n1230987) at (18,0.5) {$\ddots$};
    \node[blue,very thick](n8049) at (10.5,3.5) {$S_{\tiny n1n}$};
    \node[draw=none] (n151661) at (11.4,4.5) {$x_{n1}$};
     \draw[densely dashed] (n151661) -- (n8049);
    \node[draw=none] (n16161717) at (13.6,4.5) {$x_{12}$};
     \draw[densely dashed] (n16161717) -- (n2);
    \node[draw=none] (n151661) at (10.6,4.5) {$x_{n1}$};
    \node[blue,very thick] (n1230987) at (11.5,3.5) {$S_{1n1}$};
       \draw[densely dashed] (n151661) -- (n1230987);
    \node[blue,very thick] (n123098877) at (12.5,2.5) {$S_{2n1}$};
       \draw[densely dashed] (n1230987) -- (n123098877);
    \node[blue,very thick] (n1230976487) at (13.5,1.5) {$S_{3n1}$};
       \draw[densely dashed] (n123098877) -- (n1230976487);
    \node[draw=none] (n123098721) at (15,1) {$\ddots$};
    
    \node[red,very thick](n809) at (11,3) {$x_{1n}$};
    \node[blue,very thick](n80129) at (11.5,2.5) {$S_{12n}$};
     \draw[densely dashed] (n80129) -- (n2);
    \node[red,very thick](n805649) at (12,2) {$x_{2n}$};
    \node[blue,very thick](n555) at (12.5,1.5) {$S_{23n}$};
    \node[draw=none](n4532) at (13,1) {$\ddots$};
    \node[blue,very thick](n809) at (13.5,0.5) {$S_{(n-2)(n-1)n}$};
    \node[blue,very thick] (n1230987) at (15.5,0.5) {$S_{(n-1)n1}$};
    \node[red,very thick] (n1230987) at (16,0) {$x_{n1}$};
    \node[blue,very thick] (n1230987) at (16.5,-0.5) {$S_{n11}$};
    \node[red,very thick] (n1230987647) at (17,-1) {$x_{11}$};
    
    \node[red,very thick](n8077) at (14,0) {$x_{(n-1)n}$};
    \node[blue,very thick](n80769) at (14.5,-0.5) {$S_{(n-1)nn}$};
     \node[blue,very thick] (n1230976287) at (15.5,-0.5) {$S_{nn1}$};
    \node[draw=none] (n51516) at (16.4,-1.5)  {$x_{n1}$};
    \node[draw=none] (n15627171) at (18.5,-1.5) {$\cdots$};
\node[draw=none] (n51516543) at (15.7,-1.5)  {$x_{n1}$};
 \draw[densely dashed] (n1230987) -- (n51516543);
   \draw[densely dashed] (n1230976287) -- (n51516);
    \node[red,very thick](n809) at (15,-1) {$x_{nn}$};
    
    \node[red,very thick](n14) at (8,2) {$x_{42}$};
    \node[blue,very thick](n14) at (8.5,1.5) {$S_{452}$};
    \node[draw=none](n22) at (9,1) {$\ddots$};
    \node[blue,very thick](n24) at (9.5,0.5) {$S_{n12}$};
    \draw[densely dashed] (n7) -- (n24); 
    \node[red,very thick](n65) at (10,0) {$x_{12}$};
    \node[blue,very thick](n32) at (10.5,-0.5) {$S_{122}$};
     \draw[densely dashed] (n14151) -- (n32);
    \node[red,very thick](n90) at (11,-1) {$x_{22}$};
    \node[blue,very thick](n90) at (5.5,3.5) {$S_{212}$};
       \draw[densely dashed] (n51621712) -- (n90);
    \node[blue,very thick](n19) at (6.5,2.5) {$S_{312}$};
    \draw[densely dashed](n90)--(n19);
    \node[blue,very thick](n14) at (7.5,1.5) {$S_{412}$};
    \draw[densely dashed](n19)--(n14);
    \node[blue,very thick](n14) at (8,1) {$\ddots$};
    \node[blue,very thick](n18) at (8.5,0.5) {$S_{n12}$};
    \node[blue,very thick](n99) at (9.5,-0.5) {$S_{112}$};
    \node[draw=none] (n151616) at (10.4,-1.5) {$x_{12}$};
       \draw[densely dashed] (n151616) -- (n18);
    \node[draw=none] (n92592962) at (13,-1.5) {$\cdots$};
    \node[red,very thick](n14) at (7,3) {$x_{32}$};
\end{tikzpicture}
\end{center}
\caption{Polygonal Heronian frieze of order $n$}
\label{fig6}
\end{figure}
\begin{example}
In Figure \ref{fig7} we give a fragment of a polygonal Heronian frieze of order 5.
\end{example}
\begin{figure}
\begin{center}
\begin{tikzpicture}[xscale=1,yscale=1.3]
 \node[draw=none](n1) at (0,0) {$x_{12}$};
 \node[draw=none](n2) at (0.5,-0.5) {$S_{122}$};
 \node[draw=none](n24252) at (-0.5,-1.5) {$x_{12}$};   
 \draw[densely dashed] (n2) -- (n24252);
 \node[draw=none](n3) at (1,-1) {$x_{22}$};
 \node[draw=none](n1111) at (1,0) {\textbf{1223}};

 \node[draw=none](n7) at (-1,-1){$x_{11}$};
 \node[draw=none](n8) at (-0.5,-0.5){$S_{112}$};
 \node[draw=none](n14151) at (0.5, -1.5){$x_{12}$};
 \draw[densely dashed] (n14151) -- (n8);
 \draw[-] (n7) to (n8);
  \draw[-] (n8) to (n1);
   \draw[-] (n1) to (n2);
    \draw[-] (n2) to (n3);
 \node[draw=none](n9) at (0.5,0.5) {$S_{123}$};
 \node[draw=none](n10) at (1,1) {$x_{13}$};

 \node[draw=none](n11) at (1.5,0.5) {$S_{123}$};
 \node[draw=none](n12) at (2,0) {$x_{23}$};
 \node[draw=none](n13) at (1.5,-0.5){$S_{223}$};
 \node[draw=none](n415151) at (2.5,-1.5){$x_{23}$};
 \draw[densely dashed] (n13) -- (n415151);
 \node[draw=none](n14) at (1,-1){$x_{22}$};
  \node[draw=none](n565875) at (0,2) {$x_{53}$};
  
 \draw[densely dashed] (n9) -- (n13);

  \node[draw=none](n1908733) at (0.5,1.5) {$S_{513}$};
  \draw[-] (n565875) -- (n1908733);
  \node[draw=none](n67689) at (-1,3) {$x_{43}$};
  \node[draw=none](n6768911) at (0,4) {$x_{44}$};
  \node[draw=none](n1387445) at (10,4) {$x_{44}$};
  
  \node[draw=none](n1387445000) at (9.5,3.5) {$S_{434}$};
  \node[draw=none](n68484123) at (11,3) {$x_{54}$};
  \node[draw=none](n141616171) at (10,3) {\textbf{4534}};
  \node[draw=none](n51616) at (8,1) {\textbf{4512}};
  \node[draw=none](n68484) at (10.5,2.5) {$S_{534}$};
 \node[draw=none] (n15116) at (12,2) {$x_{14}$};
 \node[draw=none] (n19196199) at (12.5,1.5) {$S_{124}$};
 
 \draw[-] (n15116)--(n19196199);
 \node[draw=none] (n2682827) at (12.5,0.5) {$S_{234}$};
 \node[draw=none](n1519691681) at (13,1) {$x_{24}$};
 \draw[-] (n1519691681)--(n19196199);
 \draw[-] (n2682827)--(n1519691681);
 
 \node[draw=none](n251616) at (11,2) {\textbf{5134}};
 \node[draw=none](n5161616) at (11.5,2.5){$S_{514}$};
 \node[draw=none](n52622727) at (11.5,1.5){$S_{134}$};
 \draw[densely dashed] (n52622727)--(n2682827);
 \draw[densely dashed] (n68484)--(n52622727);
 \draw[-] (n68484123)--(n5161616);
 \draw[-](n5161616)--(n15116);
 \draw[-](n15116)--(n52622727);
  \node[draw=none](n6848444) at (10.5,3.5) {$S_{454}$};
  \node[draw=none](n5101961) at (10.5,4.5){$x_{45}$};
  \node[draw=none](n62062602) at (11.5,4.5){$x_{45}$};
  \draw[densely dashed] (n6848444)--(n62062602);
  \draw[-] (n6848444)--(n68484123);
  \draw[-](n1387445)--(n6848444);
   \node[draw=none](n676891123) at (-0.5,3.5) {$S_{434}$}; 
   \draw[-] (n68484123)--(n68484);
   
   \draw[-] (n6848444)--(n68484123);
   \node[draw=none] (n22526272) at (8,3) {\textbf{3423}};
   \node[draw=none] (n581671671) at (7,2) {\textbf{3412}};
   \node[draw=none] (n26829727) at (6,1) {\textbf{3451}};
   \draw[-] (n6768911)--(n676891123);
    \draw[-] (n67689)--(n676891123);
    \node[draw=none](n676891199) at (0.5,3.5) {$S_{454}$};
    
    \node[draw=none](n51519691) at (0.5,4.5){$x_{45}$};
    \node[draw=none] (n156196119) at (1.5,4.5){$x_{45}$};
    \draw[densely dashed] (n676891199) -- (n156196119);
     \draw[-] (n6768911)--(n676891199);
  \node[draw=none](n19087337) at (1,3) {$x_{54}$};
  \draw[-] (n19087337)--(n676891199);
  \node[draw=none](n1908739) at (-0.5,2.5) {$S_{453}$};
  \draw[densely dashed] (n676891199)--(n1908739);

  \draw[-] (n67689)--(n1908739);
  \draw[-] (n1908739)--(n565875);

  \node[draw=none](n190873376) at (0.5,2.5) {$S_{534}$};
  \draw[densely dashed] (n68484)--(n1387445000);
  \draw[-] (n1387445000)--(n1387445);
 
  \draw[densely dashed] (n676891123)--(n190873376);
   \draw[-] (n19087337) to (n190873376);
   \node[draw=none](n1908733996) at (1.5,2.5) {$S_{514}$};
    \draw[-] (n565875) -- (n190873376);
    \draw[densely dashed] (n1908733) -- (n1908733996);
  \draw[-] (n1) to (n9);
   \draw[-] (n9) to (n10);
   \draw[-] (n19087337)--(n1908733996);
   \draw[-] (n1908733)--(n10);
    \draw[-] (n1) to (n9);
    \draw[-] (n10) to (n11);
     \draw[-] (n11) to (n12);
      \draw[-] (n12) to (n13);
       \draw[-] (n13) to (n14);

\draw[densely dashed] (n2) -- (n11);
\draw[densely dashed] (n9) -- (n13);

\node[draw=none](n15) at (1.5,1.5) {$S_{134}$};
 \draw[densely dashed] (n190873376) -- (n15);
 \node[draw=none](n16) at (2,2) {$x_{14}$};
 \draw[-] (n16) to (n1908733996);
 \node[draw=none](n17) at (2.5,1.5) {$S_{124}$};
 \node[draw=none](n18) at (3,1) {$x_{24}$};
 \node[draw=none](n19) at (2.5,0.5){$S_{234}$};
 \node[draw=none](n131) at (2,1) {\textbf{1234}};
  \draw[-] (n10) to (n15);
   \draw[-] (n15) to (n16);
    \draw[-] (n16) to (n17);
     \draw[-] (n17) to (n18);
      \draw[-] (n18) to (n19);
\node[draw=none](n131) at (2,4) {$x_{55}$};    
\node[draw=none](n131890) at (1.5,3.5) {$S_{545}$};
\draw[densely dashed] (n131890) -- (n51519691);
\draw[-] (n131) to (n131890);
\draw[-] (n19087337) to (n131890);
\node[draw=none](n138761) at (2.5,3.5) {$S_{515}$};
\draw[-] (n131)--(n138761);
\draw[densely dashed] (n1908733996)--(n138761);
\node[draw=none](n13876100) at (2,3) {\textbf{5145}};
\node[draw=none](n13876101) at (4,1) {\textbf{2345}};
\node[draw=none](n13876102) at (1,2) {\textbf{5134}};
\node[draw=none](n13876109) at (5,2) {\textbf{2351}};
\draw[-] (n12) to (n19);
      \draw[densely dashed] (n11) -- (n17);
\draw[densely dashed] (n15) -- (n19);
 \node[draw=none](n20) at (2.5,2.5) {$S_{145}$};
 \draw[densely dashed] (n131890)--(n20);
 \node[draw=none](n21) at (3,3) {$x_{15}$};
 \draw[-] (n138761)--(n21);
 \node[draw=none](n22) at (3.5,2.5) {$S_{125}$};
 \node[draw=none](n23) at (4,2) {$x_{25}$};
 \node[draw=none](n24) at (3.5,1.5){$S_{245}$};
 \node[draw=none] (n242) at (3,2){\textbf{1245}};
 \node[draw=none](n138761908) at (0,3) {\textbf{4534}};

 \draw[densely dashed] (n20) -- (n24);

 \draw[densely dashed] (n17) -- (n22);
  \draw[-] (n16) to (n20);
   \draw[-] (n20) to (n21);
    \draw[-] (n21) to (n22);
    \draw[-] (n22) to (n23); 
     \draw[-] (n23) to (n24);
      \draw[-] (n24) to (n18);

 \node[draw=none](n25) at (3.5,3.5) {$S_{151}$};
 \node[draw=none](n26) at (4,4) {$x_{11}$};
 \node[draw=none](n27) at (4.5,3.5) {$S_{121}$};
 \node[draw=none](n1261979290) at (4.5,4.5){$x_{12}$};
 \node[draw=none](n12619792) at (5.5,4.5){$x_{12}$};
 \draw[densely dashed] (n27) -- (n12619792);
 \node[draw=none](n151661) at (3.5,4.5){$x_{51}$};
 \node[draw=none](n15161981) at (2.5,4.5){$x_{51}$};
  \draw[densely dashed] (n25) -- (n15161981);
  \draw[densely dashed] (n151661) -- (n138761);
 \node[draw=none](n28) at (5,3) {$x_{21}$};
 \node[draw=none](n29) at (4.5,2.5){$S_{251}$};
 \node[draw=none] (n141) at (4,3) {\textbf{1251}};
  \draw[-] (n21) to (n25);
   \draw[-] (n25) to (n26);
    \draw[-] (n26) to (n27);
     \draw[-] (n27) to (n28);
      \draw[-] (n28) to (n29);
       \draw[-] (n29) to (n23);

\draw[densely dashed] (n27) -- (n22);
\draw[densely dashed] (n25) -- (n29);

 \node[draw=none](n30) at (3.5,0.5) {$S_{234}$};
 \node[draw=none](n31) at (4,0) {$x_{34}$};
 \node[draw=none](n32) at (3.5,-0.5) {$S_{334}$};
 \node[draw=none](n411616) at (4.5,-1.5){$x_{34}$};
 \draw[densely dashed] (n32) -- (n411616);
\node[draw=none](n33) at (3,-1) {$x_{33}$};

  \node[draw=none](n34) at (2.5,-0.5) {$S_{233}$};
  \node[draw=none] (n61611) at (1.5,-1.5){$x_{23}$};
   \draw[densely dashed] (n34) -- (n61611);
  \node[draw=none] (n242) at (3,0) {\textbf{2334}};
  \draw[-] (n12) to (n34);
    \draw[-] (n34) to (n33);
     \draw[-] (n18) to (n19);
      \draw[-] (n18) to (n30);
       \draw[-] (n30) to (n31);
        \draw[-] (n31) to (n32);
         \draw[-] (n32) to (n33);
         \draw[densely dashed] (n34) -- (n30);
\draw[densely dashed] (n19) -- (n32);

 \node[draw=none](n37) at (4.5,1.5) {$S_{235}$};
 \draw[densely dashed] (n34) -- (n37);

 \node[draw=none](n38) at (5.0,1) {$x_{35}$};
 \node[draw=none](n39) at (4.5,0.5) {$S_{345}$};
  \draw[-] (n23) to (n37);
   \draw[-] (n37) to (n38);
    \draw[-] (n38) to (n39);
   \draw[-] (n31) to (n39); 
\draw[densely dashed] (n24) -- (n39);
  \node[draw=none](n40) at (4.5,-0.5) {$S_{344}$};
  \node[draw=none](n56106191) at (3.5,-1.5){$x_{34}$};
  \draw[densely dashed] (n40) -- (n56106191);
  \node[draw=none](n41) at (5,-1) {$x_{44}$};
  \node[draw=none] (n14141) at (5,0) {\textbf{3445}};
 \draw[-] (n41) to (n40);

  \node[draw=none](n44) at (5.5,2.5) {$S_{231}$};
  \node[draw=none](n45) at (6,2) {$x_{31}$};
 \node[draw=none](n46) at (5.5,1.5) {$S_{351}$};
  \draw[-] (n28) to (n44);
   \draw[-] (n44) to (n45);
    \draw[-] (n45) to (n46);
     \draw[-] (n46) to (n38);
\draw[densely dashed] (n44) -- (n37);
\draw[densely dashed] (n29) -- (n46);

\node[draw=none](n47) at (5.5,0.5) {$S_{345}$};
\node[draw=none](n48) at (6,0) {$x_{45}$};
\node[draw=none](n49) at (5.5,-0.5) {$S_{445}$};
\node[draw=none](n4101951) at (6.5,-1.5){$x_{45}$};
\draw[densely dashed] (n49) -- (n4101951);
\node[draw=none](n05191061) at (5.5,-1.5){$x_{45}$};
 \draw[-] (n38) to (n47);
 \draw[-] (n47) to (n48);
  \draw[-] (n48) to (n49);
   \draw[-] (n49) to (n41);
   \draw[-] (n31) to (n40);
\draw[densely dashed] (n40) -- (n47);
\draw[densely dashed] (n39) -- (n49);

\node[draw=none](n50) at (5.5,3.5) {$S_{212}$};
\draw[densely dashed] (n50) -- (n1261979290);
\node[draw=none](n51) at (6,4) {$x_{22}$};
\node[draw=none](n52) at (6.5,3.5) {$S_{232}$};
\node[draw=none](n51616109) at (6.5,4.5){$x_{23}$};
\node[draw=none](n0596020620) at (7.5,4.5){$x_{23}$};
\node[draw=none](n552929062) at (8.5,4.5){$x_{34}$};
\draw[densely dashed] (n1387445000)--(n552929062);
\node[draw=none](n529296206) at (9.5,4.5){$x_{34}$};
\draw[densely dashed] (n52)-- (n0596020620);
\node[draw=none](n53) at (7,3) {$x_{32}$};
\node[draw=none](n100) at (8,4) {$x_{33}$};
\node[draw=none](n101) at (7.5,3.5) {$S_{323}$};
\draw[densely dashed] (n101) -- (n51616109);
\node[draw=none](n141) at (6,3) {\textbf{2312}};
\draw[-] (n53) to (n101);
\draw[-] (n100) to (n101);
\draw[densely dashed] (n44) -- (n52);

\node[draw=none](n102) at (8.5,3.5) {$S_{343}$};
\draw[densely dashed] (n102)--(n529296206);
\draw[-] (n100) to (n102);
\node[draw=none](n103) at (9,3) {$x_{43}$};
\draw[-] (n102) to (n103);

\node[draw=none](n104) at (8.5,2.5) {$S_{423}$};
\draw[-] (n103) to (n104);
\node[draw=none](n105) at (9.5,2.5) {$S_{453}$};
\node[draw=none](n106) at (10,2) {$x_{53}$};
\node[draw=none](n106) at (10,2) {$x_{53}$};
\node[draw=none](n108) at (10.5,1.5) {$S_{513}$};

\node[draw=none](n109) at (11,1) {$x_{13}$};
\draw[-](n52622727)--(n109);
\node[draw=none](n110) at (10.5,0.5) {$S_{123}$};
\node[draw=none](n14151) at (9,2) {\textbf{4523}};
\node[draw=none](n111) at (10,0) {$x_{12}$};
\node[draw=none](n112) at (9.5,0.5) {$S_{512}$};
\draw[densely dashed] (n101) -- (n104);
\draw[densely dashed] (n105)--(n6848444);
\node[draw=none](n107) at (9.5,1.5) {$S_{523}$};
\draw[-] (n105) to (n103);
\draw[-] (n105) to (n106);
\draw[-] (n106) to (n107);
\node[draw=none](n54) at (6.5,2.5) {$S_{312}$};
\draw[densely dashed] (n104) -- (n107);
 \draw[densely dashed] (n50) -- (n54);
\draw[-] (n106) to (n108);
\draw[densely dashed] (n107) -- (n110);
\draw[-] (n108) to (n109);
\draw[-] (n110) to (n109);
\draw[-] (n110) to (n111);
\draw[-] (n111) to (n112);

 \node[draw=none](n55) at (6.5,1.5) {$S_{341}$};
 \draw[densely dashed] (n47) -- (n55);
 \node[draw=none](n56) at (7,1) {$x_{41}$};
  \node[draw=none](n57) at (6.5,0.5) {$S_{451}$};
  \draw[-] (n28) to (n50);
  \draw[-] (n50) to (n51);
  \draw[-] (n51) to (n52);
  \draw[-] (n52) to (n53);
  \draw[-] (n53) to (n54);
  \draw[-] (n54) to (n45);
  \draw[-] (n45) to (n55);
  \draw[-] (n55) to (n56);
   \draw[-] (n56) to (n57);
 \draw[densely dashed] (n46) -- (n57);  

\node[draw=none](n58) at (7.5,2.5) {$S_{342}$};
 \node[draw=none](n59) at (8,2) {$x_{42}$};
 \draw[-] (n59) to (n104);
  \node[draw=none](n60) at (7.5,1.5) {$S_{412}$};
  \draw[densely dashed] (n54) -- (n60);
   
\draw[densely dashed] (n58) -- (n55);
\draw[densely dashed] (n58) -- (n102);
  
\node[draw=none](n61) at (7.5,0.5) {$S_{451}$};
 \node[draw=none](n62) at (8,0) {$x_{51}$};
  \node[draw=none](n63) at (7.5,-0.5) {$S_{551}$};
  \node[draw=none](n50626921) at (8.5,-1.5){$x_{51}$};
  \node[draw=none](n5619196591) at (7.5,-1.5){$x_{51}$};
  \draw[densely dashed] (n63) -- (n50626921);
\node[draw=none](n64) at (7,-1) {$x_{55}$};
\node[draw=none](n141) at (7,0) {\textbf{4551}};
 \node[draw=none](n65) at (6.5,-0.5) {$S_{455}$};
 \draw[densely dashed] (n65) -- (n05191061);
  \draw[-] (n62) to (n61);
   \draw[-] (n56) to (n61);
   \draw[-] (n48) to (n65);
   \draw[-] (n64) to (n65);
  \draw[-] (n48) to (n57); 
  \draw[-] (n63) to (n64);
  \draw[-] (n62) to (n63);
\draw[densely dashed] (n61) -- (n65);
\draw[densely dashed] (n57) -- (n63);

 \draw[-] (n53) to (n58);
 \draw[-] (n58) to (n59);
 \draw[-] (n59) to (n60);
 \draw[-] (n60) to (n56);

\node[draw=none](n69) at (8.5,-0.5) {$S_{511}$};
\draw[densely dashed] (n5619196591)--(n69);
\node[draw=none](n70) at (9,-1) {$x_{11}$};
\node[draw=none](n141) at (9,0) {\textbf{5112}};

\node[draw=none](n73) at (8.5,1.5) {$S_{452}$};
\node[draw=none](n74) at (9,1) {$x_{52}$};
\node[draw=none](n14178) at (10,1) {\textbf{5123}};
\node[draw=none](n75) at (8.5,0.5) {$S_{512}$};
\draw[densely dashed] (n61) -- (n73);
\draw[densely dashed] (n73) -- (n105);
 \draw[-] (n59) to (n73);
  \draw[-] (n73) to (n74);
   \draw[-] (n74) to (n112);
    \draw[-] (n75) to (n62);
    \draw[-] (n107) to (n74);
\draw[-] (n74) to (n75);
\draw[densely dashed] (n60) -- (n75);

\node[draw=none](n200) at (9.5,-0.5) {$S_{112}$};
\node[draw=none](n69196919619) at (10.5,-1.5){$x_{12}$};
\draw[densely dashed] (n200)--(n69196919619);
\node[draw=none](n1595159961) at (9.5,-1.5){$x_{12}$};
\node[draw=none](n202) at (12,0) {$x_{23}$};
\node[draw=none](n201) at (11.5,0.5) {$S_{123}$};
\node[draw=none](n203) at (11.5,-0.5) {$S_{223}$};
\node[draw=none](n9519519) at (12.5,-1.5){$x_{23}$};
\draw[densely dashed] (n203)--(n9519519);
\node[draw=none](n519515195119) at (11.5,-1.5){$x_{23}$};
\node[draw=none](n204) at (11,-1) {$x_{22}$};
\node[draw=none](n205) at (10.5,-0.5) {$S_{122}$};
\draw[densely dashed] (n1595159961)--(n205);
\node[draw=none](n210) at (13,-1) {$x_{33}$};
\node[draw=none](n211) at (12.5,-0.5) {$S_{233}$};
\draw[densely dashed] (n211)--(n519515195119);
\draw[densely dashed] (n75) -- (n200);
\draw[-] (n202) to (n211);
\draw[-] (n2682827)--(n202);
\draw[densely dashed] (n203) -- (n110);
\draw[-] (n111) to (n200);
\draw[-] (n200) to (n70);
\draw[-] (n69) to (n70);
\draw[-] (n62) to (n69);
\draw[-] (n109) to (n201);
\draw[-] (n201) to (n202);
\draw[-] (n202) to (n203);
\draw[-] (n203) to (n204);
\draw[-] (n204) to (n205);
\draw[-] (n205) to (n111);
\draw[-] (n210) to (n211);
 \draw[-] (n106)--(n68484);
\draw[densely dashed] (n69) -- (n112);
\draw[densely dashed] (n108) -- (n112);
\draw[densely dashed] (n201) -- (n205);
 \draw[-] (n1387445000)--(n103);
 \draw[densely dashed] (n19196199)--(n201);
 \node[draw=none](n25161616) at (11,0) {\textbf{1223}};
\node[draw=none](n25161616) at (12,1) {\textbf{1234}};
\end{tikzpicture}
\end{center}
\caption{Fragment of a polygonal Heronian frieze of order 5}
\label{fig7}
\end{figure}
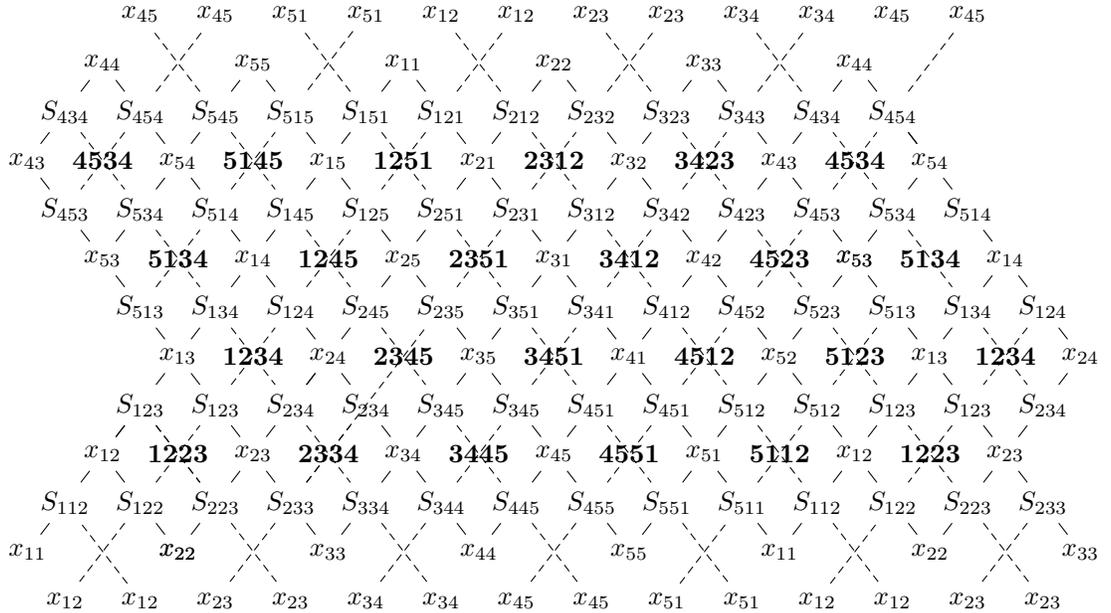
\begin{remark}
From Figure \ref{fig5}, it follows that all the diamonds of the polygonal Heronian frieze of order $n$ correspond to the quadruples of vertices of the form $(A_a,A_{a+1}, A_{b}, A_{b+1})$, as presented in Figure \ref{fig8},  where $a, b$ are distinct elements of $\{1,2...,n\}$, and addition is modulo $n$. 
\\ 
Note as well that the diamonds in the top and the bottom row of the polygonal frieze of order $n$ correspond to the quadruples $(A_a,A_{a+1},A_b,A_{b+1})$, where $b = a-1$ and $b = a+1$, respectively, and subtraction and addition are modulo $n$. In other words, the diamonds in the top and the bottom row of the polygonal Heronian frieze of order $n$ correspond to the quadruples of vertices of the form $(A_a,A_{a+1},A_{a-1},A_{a})$ and $(A_{a},A_{a+1},A_{a+1},A_{a+2})$, respectively, where $a \in \{1,2,...,n\}$.
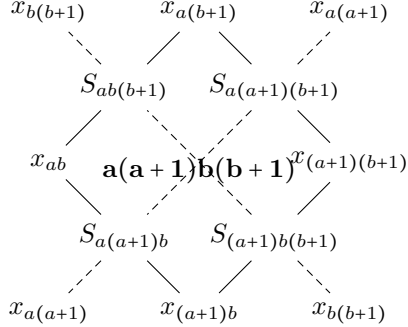
\begin{figure}
\begin{center}
\begin{tikzpicture}
 \node[draw=none](n1) at (0,0) {$x_{a(b+1)}$};
 \node[draw=none](n2) at (1,-1) {$S_{a(a+1)(b+1)}$};
 \node[draw=none](n3) at (2,-2) {$x_{(a+1)(b+1)}$};
 \node[draw=none](n4) at (1,-3) {$S_{(a+1)b(b+1)}$};
 \node[draw=none](n5) at (0,-4) {$x_{(a+1)b}$};
 \coordinate[label=${\mathbf{a(a+1)b(b+1)}}$] (c1) at (0,-2.4);
 \node[draw=none](n6) at (-1,-3){$S_{a(a+1)b}$};
 \node[draw=none](n7) at (-2,-2){$x_{ab}$};
 \node[draw=none](n8) at (-1,-1){$S_{ab(b+1)}$};
 \node[draw=none](n9) at (-2,-4){$x_{a(a+1)}$};
 \node[draw=none](n10) at (2,-4){$x_{b(b+1)}$};
 \node[draw=none](n11) at (-2,0){$x_{b(b+1)}$};
 \node[draw=none](n12) at (2,0){$x_{a(a+1)}$};
 
 \draw[-] (n1) to (n2);
 \draw[-] (n2) to (n3);
 \draw[-] (n3) to (n4);
 \draw[-] (n4) to (n5);
 \draw[-] (n5) to (n6);
 \draw[-] (n6) to (n7);
 \draw[-] (n7) to (n8);
 \draw[-] (n8) to (n1);
 \draw [densely dashed ] (n2) -- (n6);
 \draw[densely dashed] (n4) -- (n8);
 \draw[densely dashed] (n2) -- (n12);
 \draw[densely dashed] (n6) -- (n9);
 \draw[densely dashed] (n8) -- (n11);
 \draw[densely dashed] (n4) -- (n10);
\end{tikzpicture}
\end{center} 
\caption{Heronian diamond for a quadruple of vertices $A_a, A_{a+1}, A_b, A_{b+1}$}
\label{fig8}
\end{figure}
\label{rem110}
\end{remark}
\newpage
\noindent Let us now introduce the concept of a plane polygonal Heronian frieze, that will be needed for stating some results in Section 2.
\\ \\
 In order to define a plane polygonal Heronian frieze of order $n$, we are using a new type of diamonds for the purpose of  "gluing" together the copies of the original frieze. Namely, those are the diamonds corresponding to the quadruples of vertices of the form $(A_a,A_{a+1},A_a,A_{a+1})$, where $a \in \{1,2,...,n\}$, and addition is modulo $n$, i.e. the diamonds as in Figure \ref{fig8}, where $a = b$. It is easy to check that these are Heronian diaomonds as well. Hence the following definition.
\begin{definition}
 A \textit{plane polygonal Heronian frieze of order $n$} is a collection of numbers arranged as in Figure \ref{fig9}, arising from a polygon $(A_1,A_2,...,A_n)$ in the complex plane with $n \geq 3$, and satisfying a condition that all the diamonds constituting the frieze are Heronian. \\ 
The diamonds used for gluing the copies of the original frieze, as explained above, are drawn in red in Figure \ref{fig9}, and we will call them the \textit{gluing diamonds}.
\begin{remark}
The frieze from Definition \ref{def111} covers the whole plane, hence the name.    
\end{remark}
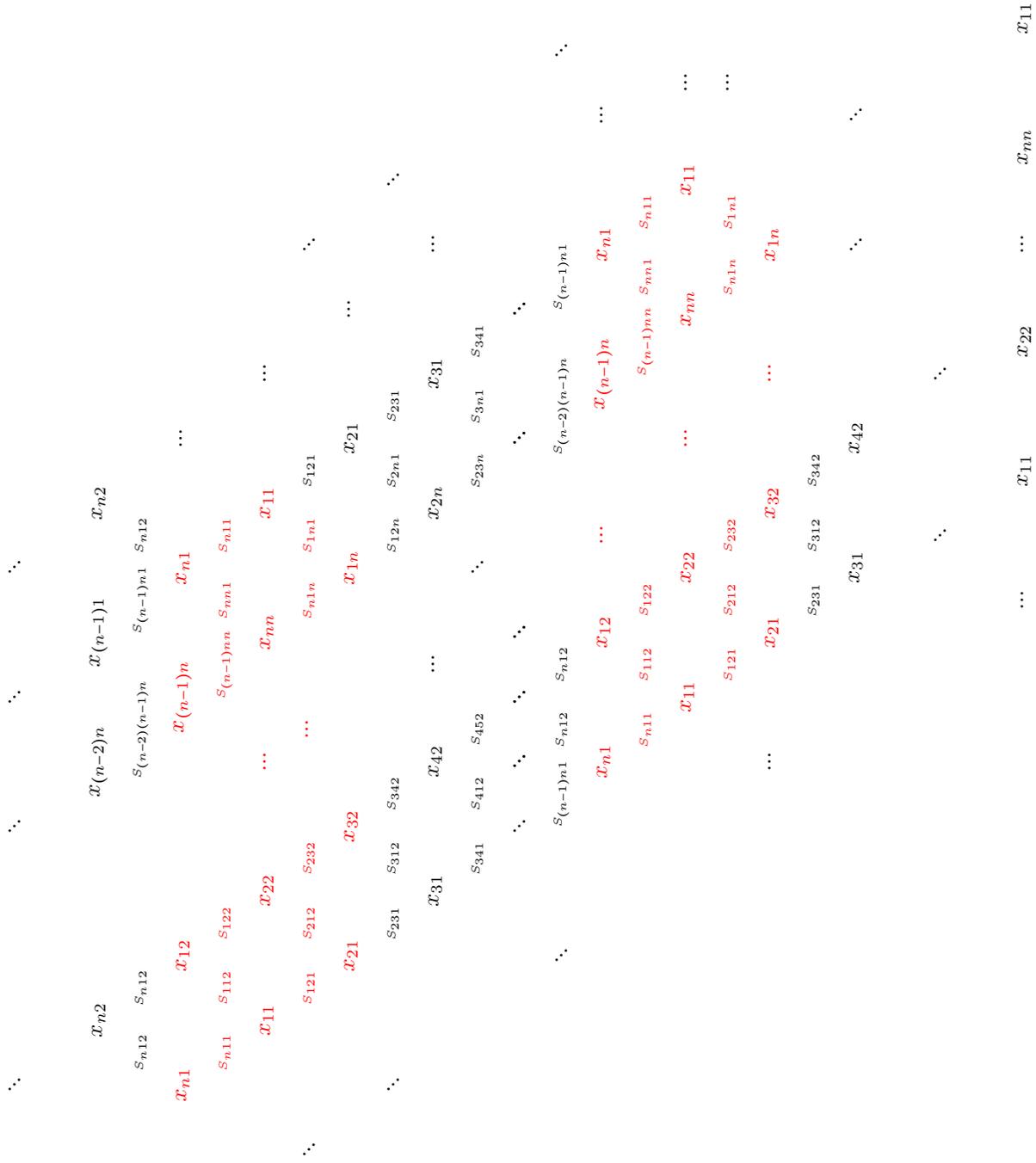
\begin{sidewaysfigure}
\begin{center}
\begin{tikzpicture}[xscale=1,yscale=1.3]
  \node[red](n1) at (4,4) {$x_{11}$};
   \node[red](n627728) at (4.5,4.5) {\tiny $S_{112}$};
   \node[red](n2) at (4.5,3.5) {\tiny $S_{121}$};
   \node[draw=none] (n123098786787) at (2,3.5) {$\ddots$};
   
    \node[red](n3) at (5,3) {$x_{21}$};
    \node[very thick](n5) at (5.5,2.5) {\tiny $S_{231}$};
    \node[draw=none] (n123787665) at (3,2.5) {$\ddots$};
    \node[draw=none](n4) at (6,2) {$x_{31}$};
    \node[very thick](n5) at (6.5,1.5) {\tiny $S_{341}$};
    \node[red](n9) at (8,0) {$x_{n1}$};
    \node[draw=none](n12) at (7,1) {$\ddots$};
    \node[draw=none](n314151) at (11,1.5) {$\ddots$};
    \node[very thick](n25151) at (9.5,2){$\cdots$};
    \node[red](n25151) at (11.5,0){$\cdots$};
    \node[very thick](n10) at (7.5,0.5) { \tiny $S_{(n-1)n1}$};
    \node[draw=none] (n1230987) at (5,0.5) {$\ddots$};
    \node[red](n7) at (8.5,-0.5) {\tiny $S_{n11}$};
    
    \node[draw=none](n494098) at (12.5,-5) {$x_{11}$};
    \node[red](n8) at (9,-1) {$x_{11}$};
    \node[red](n8) at (9.5,-1.5) {\tiny $S_{121}$};
 \node[red](n89879) at (10,-2) {$x_{21}$};
  \node[very thick](n62772218) at (10.5,-2.5) {\tiny $S_{231}$};
 \node[red](n89976) at (10.5,-1.5) {\tiny $S_{212}$};
     \node[very thick](n900997654) at (11,-3) {$x_{31}$};
      \node[very thick](n62772887) at (11.5,-2.5) {\tiny $S_{312}$};
    \node[very thick](n8978666868) at (11.5,-4) {$\ddots$};
     
      \node[very thick](n1897647687) at (10.5,-5) {$\cdots$};
      \node[very thick](n231567547) at (14.5,-5) {$x_{22}$};
       \node[very thick](n1897647687989) at (16,-5) {$\cdots$};
       \node[very thick](n1897647687) at (17.5,-5) {$x_{nn}$};
       \node[very thick](n1897647687) at (19.5,-5) {$x_{11}$}; 
   \node[very thick] (n1591691) at (8, -2) {$\cdots$};
     \node[red](n137866) at (5,5) {$x_{12}$};
      \node[red](n627728) at (5.5,4.5) {\tiny $S_{122}$};
      \node[very thick](n136373) at (4,6) {$x_{n2}$};
      \node[red](n41991500) at (3,5){$x_{n1}$};
      \node[red](n5951519) at (3.5,4.5){\tiny $S_{n11}$};
      \node[draw=none](n5195101) at (3.5,5.5){\tiny $S_{n12}$};
       \node[very thick](n6277997528) at (4.5,5.5) {\tiny $S_{n12}$};
       \node[very thick](n1685393) at (3,7) {$\ddots$};

    \node[very thick](n17) at (6.5,3.5) {$ $};

    \node[very thick](n14) at (10,1) {$\ddots$};

    \node[very thick](n87) at (7.5,2.5) {$ $};

     \node[red](n628687728) at (10.5,4.5) {\tiny $S_{nn1}$};
    
     \node[red](n892478868) at (9,5) {$x_{(n-1)n}$};
      
      \node[red](n652527728) at (9.5,4.5) {\tiny $S_{(n-1)nn}$};
      \node[very thick](n8978123868) at (8,6) {$x_{(n-2)n}$};
       \node[very thick](n627989728) at (8.5,5.5) {\tiny $S_{(n-2)(n-1)n}$};
       \node[very thick](n8978868) at (7,7) {$\ddots$};
    \node[red](n900987) at (11,5) {$x_{n1}$};
     \node[red](n627687728) at (11.5,4.5) {\tiny $S_{n11}$};
     \node[very thick](n627767728) at (11.5,5.5) {\tiny $S_{n12}$};
     \node[very thick](n189764877687) at (13,5) {$\cdots$};
     \node[very thick](n8970898868) at (10,6) {$x_{(n-1)1}$};
      \node[very thick](n62427728) at (10.5,5.5) {\tiny $S_{(n-1)n1}$};
      \node[very thick](n897880065568) at (9,7) {$\ddots$};
    \node[very thick](n98900660) at (12,6) {$x_{n2}$};
     \node[very thick](n89778658868) at (11,7) {$\ddots$};
   
     \node[very thick](n1897678567) at (14,4) {$\cdots$};
   \node[very thick](n2) at (12.5,3.5) {$ $};
     \node[very thick](n1897654655) at (15,3) {$\cdots$};
    \node[very thick](n5) at (13.5,2.5) {$ $};
  
     \node[very thick](n189989075) at (16,2) {$\cdots$};
    \node[very thick](n5) at (14.5,1.5) {$ $}; 
    \node[very thick](n8049) at (10.5,3.5) {$ $};
    \node[very thick] (n1230987) at (11.5,3.5) {$ $};
    \node[very thick] (n123098877) at (12.5,2.5) {$ $};
    \node[very thick] (n1230976487) at (13.5,1.5) {$ $};
    \node[draw=none] (n123098721) at (15,1) {$\ddots$};
   
    \node[very thick](n80129) at (11.5,2.5) {$ $};
  
    \node[very thick](n555) at (12.5,1.5) {$ $};
    \node[draw=none](n4532) at (13,1) {$\ddots$};
    \node[very thick](n809) at (13.5,0.5) {$ $};
    \node[very thick] (n1230987) at (15.5,0.5) {$ $};
     \node[very thick](n189765478) at (18,0) {$\cdots$};
    \node[very thick] (n1230987) at (16.5,-0.5) {$ $};
    \node[very thick](n80769) at (14.5,-0.5) {$ $};
     \node[very thick] (n1230976287) at (15.5,-0.5) {$ $};
    \node[draw=none] (n15627171) at (18.5,-1) {$\cdots$};
     \node[red](n65427728) at (15.5,-1.5) {\tiny $S_{n1n}$};
   
    \node[very thick](n14) at (8.5,1.5) {$ $};
    \node[draw=none](n22) at (9,1) {$\ddots$};
    \node[very thick](n24) at (9.5,0.5) {$ $};
   
    \node[very thick](n32) at (10.5,-0.5) {$ $};
   \node[red](n627728) at (11.5,-1.5) {\tiny $S_{232}$};
    \node[red](n189745211) at (13,-1) {$\cdots$};
    \node[red](n9789670) at (12,-2) {$x_{32}$};
     \node[very thick](n627878728) at (12.5,-2.5) {\tiny $S_{342}$};
     \node[red](n18947383) at (14,-2) {$\cdots$};
    \node[very thick](n9090776) at (13,-3) {$x_{42}$};
     \node[very thick](n1897647687) at (16,-3) {$\ddots$};
      \node[very thick](n1897647687) at (18,-3) {$\ddots$};
    \node[very thick](n99908770) at (14,-4) {$\ddots$};
    \node[very thick](n90) at (5.5,3.5) {$ $};
    \node[very thick](n19) at (6.5,2.5) {$ $};
    \node[very thick](n14) at (7.5,1.5) {$ $};
    \node[very thick](n14) at (8,1) {$\ddots$};
    \node[red](n18) at (8.5,0.5) {$ $};
    \node[very thick](n99) at (9.5,-0.5) {$ $};
      
    \node[red](n13) at (6,4) {$x_{22}$};
    \node[red](n805439) at (8,4) {$\cdots$};
    \node[red](n8043879) at (8.5,3.5) {$\cdots$};
    \node[red](n17) at (6.5,3.5) {\tiny $S_{232}$};
    \node[very thick](n14) at (10,1) {$\ddots$};
    
    \node[very thick](n87) at (7.5,2.5) {\tiny $S_{342}$};
  
    \node[red](n8059) at (10,4) {$x_{nn}$};
    \node[red](n1) at (12,4) {$x_{11}$};
   \node[very thick](n2) at (12.5,3.5) {\tiny $S_{121}$};
    \node[draw=none](n3) at (13,3) {$x_{21}$};
    \node[very thick](n5) at (13.5,2.5) {\tiny $S_{231}$};
    \node[very thick](n4) at (14,2) {$x_{31}$};
    \node[very thick](n5) at (14.5,1.5) {\tiny $S_{341}$}; 
    \node[draw=none] (n1230987) at (16,3.5) {$\ddots$};
    \node[draw=none] (n1230988767) at (17,2.5) {$\ddots$};
    \node[draw=none] (n1230987) at (19,0.5) {$\ddots$};
    \node[red](n8049) at (10.5,3.5) {\tiny $  S_{ n1n}$};
    
    \node[red] (n1230987) at (11.5,3.5) {\tiny $S_{1n1}$};
    
    \node[very thick] (n123098877) at (12.5,2.5) {\tiny $S_{2n1}$};

    \node[very thick] (n1230976487) at (13.5,1.5) {\tiny $S_{3n1}$};
       
    \node[draw=none] (n123098721) at (15,1) {$\ddots$};
    
    \node[red](n809) at (11,3) {$x_{1n}$};
    \node[very thick](n80129) at (11.5,2.5) {\tiny $S_{12n}$};
     
    \node[draw=none](n805649) at (12,2) {$x_{2n}$};
    \node[very thick](n555) at (12.5,1.5) {\tiny $S_{23n}$};
    \node[draw=none](n4532) at (13,1) {$\ddots$};
    \node[very thick](n809) at (13.5,0.5) {\tiny $S_{(n-2)(n-1)n}$};
    \node[draw=none] (n1230987) at (15.5,0.5) {\tiny $S_{(n-1)n1}$};
    \node[red] (n1230987) at (16,0) {$x_{n1}$};
    \node[red] (n1230987) at (16.5,-0.5) {\tiny $S_{n11}$};
    \node[red] (n1230987647) at (17,-1) {$x_{11}$};
    
    \node[red](n8077) at (14,0) {$x_{(n-1)n}$};
    \node[red](n80769) at (14.5,-0.5) {\tiny $S_{(n-1)nn}$};
     \node[red] (n1230976287) at (15.5,-0.5) {\tiny $S_{nn1}$};
    
    \node[draw=none] (n15627171) at (18.5,-1.5) {$\cdots$};
\node[red] (n51516543) at (16,-2)  {$x_{1n}$};
  \node[red](n629897728) at (16.5,-1.5) {\tiny $S_{1n1}$};
    \node[red](n809) at (15,-1) {$x_{nn}$};
    
    \node[draw=none](n14) at (8,2) {$x_{42}$};
    \node[very thick](n14) at (8.5,1.5) {\tiny $S_{452}$};
    \node[draw=none](n22) at (9,1) {$\ddots$};
    \node[very thick](n24) at (9.5,0.5) {\tiny $S_{n12}$};

    \node[red](n65) at (10,0) {$x_{12}$};
    \node[red](n32) at (10.5,-0.5) {\tiny $S_{122}$};
    
    \node[red](n90) at (11,-1) {$x_{22}$};
    \node[red](n90) at (5.5,3.5) {\tiny $S_{212}$};
       
    \node[very thick](n19) at (6.5,2.5) {\tiny $S_{312}$};
    
    \node[very thick](n14) at (7.5,1.5) {\tiny $S_{412}$};
    
    \node[very thick](n14) at (8,1) {$\ddots$};
    \node[very thick](n18) at (8.5,0.5) {\tiny $S_{n12}$};
    \node[red](n99) at (9.5,-0.5) {\tiny $S_{112}$};

    \node[red](n14) at (7,3) {$x_{32}$};
\end{tikzpicture}
\end{center}
\caption{Plane polygonal Heronian frieze of order $n$ (the dashed lines are extended)}
\label{fig9}
\end{sidewaysfigure}
\label{def111}
\end{definition}
\section{On Determinants in Heronian friezes}
Before stating the results, we shall first introduce some geometrical prerequisites.
\begin{definition}
    A polygon $P = (A_1, A_2, ..., A_n)$ is cyclic if all of its vertices are distinct and it can be inscribed in a circle.
\end{definition}
\begin{definition}\cite{key2}
Let $P = (A_1, A_2, ..., A_n)$ be a cyclic polygon, and $|A_iA_j|$, where $i, j \in \{1, 2, ..., n\}$, the length of the chord $A_iA_j$. Then we define $c_r$ to be the reciprocal of the product of the chord lengths that start at $A_r$, for $r \in \{1, 2, ..., n\}$, i.e. \begin{equation}
 c_{r} := \frac{1}{\prod_{s\neq r} |A_rA_s|}.
\end{equation}
\end{definition}
\begin{theorem} \cite[Section 1]{key2}
 Let $P = (A_1, A_2, ..., A_n)$ be a cyclic polygon and $n$ an even integer. Then $c_1 - c_2 + c_3 - ... - c_n = 0$.
 \label{thm009}
\end{theorem}
\begin{theorem} \cite[Theorem 10]{key3}
Let $P = (A_0, A_1, ..., A_n)$ be a cyclic  $(n+1)$-gon, where $n \geq 2$. If $\delta_{ij}$ is the length of the segment $A_iA_j$, then
\begin{equation}
    \frac{\delta_{0(n-1)}}{\delta_{0n}\delta_{n(n-1)} }=  \sum_{i=1}^{n - 1} \frac{\delta_{(i-1)i}}{\delta_{(i-1)n}\delta_{ni}}.
\end{equation}
\label{thm24}
\end{theorem}
\begin{theorem} \cite[Theorem 12]{key3}
Let $n \geq 2$. Suppose $P = (A_0, A_1,..., A_{2n-1})$ is a cyclic $2n$-gon. Let $Q$ be the $n$-gon with vertices $(A_0, A_2, ..., A_{2(n-1)})$. If $n$ is even, then the relation 
\begin{equation}
    det(|A_{2i-1}A_{2(j-1)}|^d)_{i,j=1} ^{n}= 0
    \label{detid}
\end{equation}
holds for $d =
0,2,...,n- 2$.
\label{thm25}
\end{theorem}
Note that the identity \eqref{detid} involves distances between odd-numbered and even-numbered vertices only.
\begin{theorem}For four distinct points $A_i, A_j, A_k, A_l$, lying on a circle, in an anticlockwise order, it holds that 
\[
\begin{vmatrix}
     x_{kl} & S_{ikl} & S_{jkl}\\ 
     -x_{ij} & S_{ijk} & S_{ijl}\\
     0 & x_{ik} & x_{jl} 
\end{vmatrix} = 
\begin{vmatrix}
     x_{kl} & S_{ikl} & S_{jkl}\\ 
     x_{ij} & S_{ijl} & S_{ijk}\\
     0 & x_{il} & x_{jk} 
\end{vmatrix} = \begin{vmatrix}
     x_{jk} & S_{ijk} & S_{jkl}\\ 
     -x_{il} & S_{ikl} & S_{ijl}\\
     0 & x_{ik} & x_{jl} 
\end{vmatrix} = 0,\]
where the entries are as defined in~\eqref{eqx} and~\eqref{eqs}.
\label{thm26}
\end{theorem}
\begin{proof}
Let $\delta_{pq}:= |A_pA_q|$, for $p, q \in \{i, j, k, l\}$. Then, from Theorem ~\ref{thm009}, it follows that \begin{equation*}
    \frac{1}{\delta_{ij}\delta_{ik}\delta_{il}} - \frac{1}{\delta_{ji}\delta_{jk}\delta_{jl}} + \frac{1}{\delta_{ki}\delta_{kj}\delta_{kl}} - \frac{1}{\delta_{li}\delta_{lj}\delta_{lk}} = 0,
\end{equation*}
which is equivalent to 
\begin{equation}
\delta_{jk}\delta_{jl}\delta_{kl} - \delta_{ik}\delta_{il}\delta_{kl} + \delta_{ij}\delta_{il}\delta_{jl} - \delta_{ij}\delta_{ik}\delta_{jk} = 0.
\label{eq17}
\end{equation}
\noindent
Multiplying equation ~\eqref{eq17} by $\frac{\delta_{ij}\delta_{ik}\delta_{jl}\delta_{kl}}{R}$, where $R$ is the radius of the circle, and using  the fact that four times signed area of a triangle is positive if its vertices are oriented anticlockwise, we get the following:
\begin{equation}
    S_{ijk}x_{jl}x_{kl} - S_{ijl}x_{ik}x_{kl} + S_{ikl}x_{ij}x_{jl} - S_{jkl}x_{ij}x_{ik} = 0.
    \label{eq18}
\end{equation}
Note that, here we also used the fact that the area $T$ of a triangle whose side lengths are $a, b, c$ and the radius of its circumscribed circle is $R$,  can be computed using the formula $T = \frac{abc}{4R
}$.\\ \\
Now, from ~\eqref{eq18}, it follows that 
\begin{equation}
    x_{kl}(S_{ijk}x_{jl} - S_{ijl}x_{ik}) + x_{ij}(S_{ikl}x_{jl} - S_{jkl}x_{ik}) = 0,
\end{equation}
which is equivalent to 
\[
\begin{vmatrix}
     x_{kl} & S_{ikl} & S_{jkl}\\ 
     -x_{ij} & S_{ijk} & S_{ijl}\\
     0 & x_{ik} & x_{jl} 
\end{vmatrix} = 0.\]
Similarly, by multiplying equation \eqref{eq17} by $\frac{\delta_{ij}\delta_{il}\delta_{jk}\delta_{kl}}{R}$ and $\frac{\delta_{ik}\delta_{il}\delta_{jk}\delta_{jl}}{R}$, respectively, we get that 
\[\begin{vmatrix}
     x_{kl} & S_{ikl} & S_{jkl}\\ 
     x_{ij} & S_{ijl} & S_{ijk}\\
     0 & x_{il} & x_{jk} 
\end{vmatrix} = 0,\] and
\[\begin{vmatrix}
     x_{jk} & S_{ijk} & S_{jkl}\\ 
     -x_{il} & S_{ikl} & S_{ijl}\\
     0 & x_{ik} & x_{jl} 
\end{vmatrix} = 0.\]
\end{proof}
\begin{corollary}
Let $P = (A_1, ..., A_n)$ be a cyclic $n$-gon, with anticlockwise ordering of vertices. Then some $3 \times 3$ determinants related to the diamonds of the corresponding polygonal Heronian frieze vanish - namely, 
for every diamond of the frieze, as shown in Figure \ref{fig8}, the following holds:
\[
\begin{vmatrix}
     x_{b(b+1)} & S_{ab(b+1)} & S_{(a+1)b(b+1)}\\ 
     -x_{a(a+1)} & S_{a(a+1)b} & S_{a(a+1)(b+1)}\\
     0 & x_{ab} & x_{(a+1)(b+1)} 
\end{vmatrix} = 0; \]
\[
\begin{vmatrix}
     x_{b(b+1)} & S_{ab(b+1)} & S_{(a+1)b(b+1)}\\ 
     x_{a(a+1)} & S_{a(a+1)(b+1)} & S_{a(a+1)b}\\
     0 & x_{a(b+1)} & x_{(a+1)b} 
\end{vmatrix} = 0;\]
\[\begin{vmatrix}
     x_{(a+1)b} & S_{a(a+1)b} & S_{(a+1)b(b+1)}\\ 
     -x_{a(b+1)} & S_{ab(b+1)} & S_{a(a+1)(b+1)}\\
     0 & x_{ab} & x_{(a+1)(b+1)} 
\end{vmatrix} = 0.\]
Furthermore, the corresponding determinants vanish for all the diamonds of the plane polygonal Heronian frieze as well.
\label{cor261}
\end{corollary}
\begin{proof}
For all the diamonds of the frieze, except for those in the top row and the bottom row, the statement follows from Theorem \ref{thm26}. \\
The diamonds in the bottom row of the frieze, as stated in Remark \ref{rem110}, correspond to the quadruples of the form $(A_a,A_{a+1},A_b,A_{b+1})$, where $b = a+1$, so we have that 
\[
\begin{vmatrix}
     x_{b(b+1)} & S_{ab(b+1)} & S_{(a+1)b(b+1)}\\ 
     -x_{a(a+1)} & S_{a(a+1)b} & S_{a(a+1)(b+1)}\\
     0 & x_{ab} & x_{(a+1)(b+1)} 
\end{vmatrix} =
\begin{vmatrix} x_{a+1,a+2} & S_{a,a+1,a+2} & 0\\ 
     -x_{a,a+1} & 0 & S_{a,a+1,a+2}\\
     0 & x_{a,a+1} & x_{a+1,a+2} 
\end{vmatrix},
\]
and one easily checks (by expanding along the first row) that the latter determinant equals $0$. \\
 We also have that \[
\begin{vmatrix}
     x_{b(b+1)} & S_{ab(b+1)} & S_{(a+1)b(b+1)}\\ 
     x_{a(a+1)} & S_{a(a+1)(b+1)} & S_{a(a+1)b}\\
     0 & x_{a(b+1)} & x_{(a+1)b} 
\end{vmatrix} = \begin{vmatrix}
     x_{a+1,a+2} & S_{a,a+1,a+2} & 0\\ 
     x_{a,a+1} & S_{a,a+1,a+2} & 0\\
     0 & x_{a,a+2} & 0
\end{vmatrix} = 0,\]
 and
 \[\begin{vmatrix}
     x_{(a+1)b} & S_{a(a+1)b} & S_{(a+1)b(b+1)}\\ 
     -x_{a(b+1)} & S_{ab(b+1)} & S_{a(a+1)(b+1)}\\
     0 & x_{ab} & x_{(a+1)(b+1)} 
\end{vmatrix} = \begin{vmatrix}
     0 & 0 & 0\\ 
     -x_{a,a+2} & S_{a,a+1,a+2} & S_{a,a+1,a+2}\\
     0 & x_{a,a+1} & x_{a+1,a+2} 
\end{vmatrix} = 0. \]
As for the diamonds of the top row of the frieze, it has already been stated in Remark \ref{rem110}, that they correspond to the quadruples of the form $(A_a,A_{a+1},A_b,A_{b+1})$, where $b = a-1$. Checking that the determinants vanish in this case as well is analogous to checking that they vanish for the diamonds of the bottom row of the frieze. \\ In order to prove the statement for the plane polygonal Heronian frieze, we only  need to show that the determinants vanish for gluing diamonds as well. As already stated, they correspond to the quadruples of vertices of the form $(A_a,A_{a+1},A_b,A_{b+1})$, where $a = b$, so we have that
\[
\begin{split}
\begin{vmatrix}
     x_{b(b+1)} & S_{ab(b+1)} & S_{(a+1)b(b+1)}\\ 
     -x_{a(a+1)} & S_{a(a+1)b} & S_{a(a+1)(b+1)}\\
     0 & x_{ab} & x_{(a+1)(b+1)} 
\end{vmatrix} &= 
\begin{vmatrix}
     x_{a(a+1)} & S_{aa(a+1)} & S_{(a+1)a(a+1)}\\ 
     -x_{a(a+1)} & S_{a(a+1)a} & S_{a(a+1)(a+1)}\\
     0 & x_{aa} & x_{(a+1)(a+1)} 
\end{vmatrix}\\ & = 
\begin{vmatrix}
     x_{a(a+1)} & 0 & 0\\ 
     -x_{a(a+1)} & 0 & 0\\
     0 & 0 & 0 
\end{vmatrix} = 0.
\end{split}\]
It is an easy check that the second and the third determinant of the statement vanish in the case of gluing diamonds as well.
\end{proof}
\begin{remark}
Note that vanishing of the first determinant from Corollary \ref{cor261} can also be interpreted in the following way:
\[
\begin{vmatrix}
     x_{b(b+1)} & S_{ab(b+1)} & S_{(a+1)b(b+1)}\\ 
     -x_{a(a+1)} & S_{a(a+1)b} & S_{a(a+1)(b+1)}\\
     x_{(a+1)b} & x_{ab} & x_{(a+1)(b+1)} 
\end{vmatrix} = 0,\]
for the diamonds of the bottom row of the frieze, or 
\[
\begin{vmatrix}
     x_{b(b+1)} & S_{ab(b+1)} & S_{(a+1)b(b+1)}\\ 
     -x_{a(a+1)} & S_{a(a+1)b} & S_{a(a+1)(b+1)}\\
     x_{a(b+1)} & x_{ab} & x_{(a+1)(b+1)} 
\end{vmatrix} = 0, \]
for the diamonds of the top row of the frieze. \\ \\
On the other hand, vanishing of the second determinant from Corollary \ref{cor261} can be interpreted as:
\[
\begin{vmatrix}
     x_{b(b+1)} & S_{ab(b+1)} & S_{(a+1)b(b+1)}\\ 
     x_{a(a+1)} & S_{a(a+1)(b+1)} & S_{a(a+1)b}\\
     x_{ab} & x_{a(b+1)} & x_{(a+1)b} 
\end{vmatrix} = 0,\]
or
\[
\begin{vmatrix}
     x_{b(b+1)} & S_{ab(b+1)} & S_{(a+1)b(b+1)}\\ 
     x_{a(a+1)} & S_{a(a+1)(b+1)} & S_{a(a+1)b}\\
     x_{(a+1)(b+1)} & x_{a(b+1)} & x_{(a+1)b} 
\end{vmatrix} = 0,\]
for the gluing diamonds of the corresponding plane polygonal Heronian frieze. 
\end{remark}
\begin{theorem}
Let $P = (A_1, A_2, ..., A_{m+1})$ be a cyclic $(m+1) - gon$, with anticlockwise ordering of the vertices, where $m \geq 2$ is an integer. Then 
\begin {equation}
S_{1,m,m+1}\prod_{i=2}^{m-1} x_{i,m+1} = \sum_{j=1}^{m-1} S_{j,j+1,m+1} \frac{\prod_{k=1}^{m} x_{k,m+1}}{x_{j,m+1}x_{j+1,m+1}},
\end{equation}
where the entries $S_{***}$ and $x_{**}$ are the entries of the corresponding polygonal Heronian frieze.
\label{thm28}
\end{theorem}
\begin{proof}
Using Theorem \ref{thm24}, we have that 
\begin{equation*}
    \frac{\delta_{1,m}}{\delta_{1,m+1}\delta_{m,m+1}} = \frac{\delta_{12}}{\delta_{1,m+1}\delta_{2,m+1}} + \frac{\delta_{23}}{\delta_{2,m+1}\delta_{3,m+1}} + \cdots + \frac{\delta_{m-1,m}}{\delta_{m-1,m+1}\delta_{m,m+1}},
\end{equation*}
which is equivalent to 
\begin{equation}
    \delta_{1m}\prod_{i=2}^{m-1} \delta_{i,m+1} = \sum_{j=1}^{m-1} \delta_{j,j+1}\frac{\prod_{k=1}^{m} \delta_{k,m+1}}{\delta_{j,m+1}\delta_{j+1,m+1}}.
    \label{eq21}
\end{equation}
Multiplying equation \eqref{eq21} by $\frac{\prod_{l=1}^{m} \delta_{l,m+1}}{R}$, where $R$ is the radius of the circumscirbed circle of the $(m+1)$-gon, we get
\begin{equation}
\frac{\delta_{1m}\prod_{i=2}^{m-1} \delta_{i,m+1} \prod_{l=1}^{m} \delta_{l,m+1}}{R} = \sum_{j=1}^{m-1} \frac{\delta_{j,j+1}\prod_{k=1}^{m} \delta_{k,m+1}\prod_{l=1}^{m} \delta_{l,m+1}}{\delta_{j,m+1}\delta_{j+1,m+1}R}.
\label{eq22}
\end{equation}
On the right hand side of \eqref{eq22}, note that the $j$-th summand, for $j \in \{1,2,..,m-1\}$, is equal to 
\begin{equation*}
\begin{split}
    & \frac{\delta_{j,j+1}\left(\delta_{1,m+1}\delta_{2,m+1}\cdots\delta_{j,m+1}\delta_{j+1,m+1}\cdots \delta_{m,m+1}\right)\left(\delta_{1,m+1}\delta_{2,m+1} \cdots \delta_{j,m+1}\delta_{j+1,m+1} \cdots \delta_{m,m+1}\right)}{\delta_{j,m+1}\delta_{j+1,m+1}R} \\
   & \qquad \qquad \qquad= \frac{\delta_{j,j+1}\delta_{j,m+1}\delta_{j+1,m+1}}{R} \prod_{\substack{k=1 \\ k \neq j,j+1}}^m \delta_{k,m+1}^2 \\ & \qquad \qquad \qquad= S_{j,j+1,m+1}\frac{\prod_{k=1}^{m}x_{k,m+1}}{x_{j,m+1}x_{j+1,m+1}}.
    \end{split}
    \end{equation*}
From here, it follows that the right hand side of the equality \eqref{eq22} is equal to $\sum_{j=1}^{m-1} S_{j,j+1,m+1} \frac{\prod_{k=1}^{m} x_{k,m+1}}{x_{j,m+1}x_{j+1,m+1}}$. \\ \\
As for the left hand side, note that 
\begin{equation*}
\begin{split}
    \frac{\delta_{1m}\prod_{i=2}^{m-1} \delta_{i,m+1} \prod_{l=1}^{m} \delta_{l,m+1}}{R} &= \frac{\delta_{1m}\delta_{1,m+1}\delta_{m,m+1}}{R}  (\delta_{2,m+1}\delta_{3,m+1}...\delta_{m-1,m+1})^2 \\ & = S_{1,m,m+1} \prod_{i=2}^{m-1} x_{i,m+1},
    \end{split}
\end{equation*}
and the statement follows.
\end{proof}

\begin{corollary}
 Let $P = (A_1, A_2, ..., A_n)$ be a cyclic $n$-gon, with anticlockwise ordering of the vertices, and let $F$ be the corresponding polygonal Heronian frieze. Let $m$ be an integer such that $2 \leq m \leq n-1$, and $q$ and $r$ integers such that $1 \leq q,r \leq n$, and $q, q+1, ..., q+m-2, r, r+1$ are distinct. Then 
 \begin{equation}
 \begin{split}
     S_{q,r,r+1}\prod_{k=q+1}^{q+m-2}x_{k,r+1}   &= S_{q,q+1,r+1}x_{q+2,r+1}x_{q+3,r+1}\cdots x_{q+m-2,r+1}x_{r,r+1}  \\ & \qquad \qquad + S_{q+1,q+2,r+1}x_{q,r+1}x_{q+3,r+1}\cdots x_{q+m-2,r+1}x_{r,r+1} \\  & \qquad \qquad \qquad \qquad \qquad \qquad \vdots  \\ & \qquad \qquad +  S_{q+m-3,q+m-2,r+1}x_{q,r+1}x_{q+1,r+1}\cdots x_{q+m-4,r+1}x_{r,r+1} \\  & \qquad \qquad + S_{q+m-2,r,r+1}x_{q,r+1}x_{q+1,r+1}\cdots x_{q+m-4,r+1}x_{q+m-3,r+1},
     \end{split}
     \label{eq23}
 \end{equation}
 where index addition is modulo $n$.
\begin{figure}

    \begin{tikzpicture}
 \node[red,very thick](n1) at (0,0) {$x_{q(r+1)}$};
 \node[red,very thick](n2) at (1,-1) {$S_{q(q+1)(r+1)}$};
 \node[red,very thick](n3) at (2,-2) {$x_{(q+1)(r+1)}$};
 \node[red,very thick] (n448) at (3,-3) {$S_{(q+1)(q+2)(r+1)}$};
 \draw[-] (n3) to (n448);
 \node[draw=none](n4) at (1,-3) {$S_{(q+1)r(r+1)}$};
 \node[draw=none] (n444) at (2,-6) {$x_{q+2,r}$};
\node[draw=none] (n4443) at (3,-7) {$S_{(q+2)(q+3)r}$}; 
\draw[-] (n444) to (n4443);
\node[draw=none] (n44439) at (4,-8) {$x_{q+3,r}$};
\draw[-] (n4443) to (n44439);
 \node[draw=none] (n445) at (3,-5) {$S_{(q+2)r(r+1)}$};
 \draw[densely dashed](n4)--(n445);
 \node[red,very thick] (n446) at (4,-4) {$x_{(q+2)(r+1)}$};
 
\node[draw=none] (n446677) at (5,-7) {$S_{(q+3)r(r+1)}$}; 
\draw[densely dashed](n445)--(n446677);
 \node[draw=none](n446688) at (6,-8) {$\ddots$};
  \node[draw=none](n14123) at (7,-9) {$\ddots$};
  \node[draw=none] (n10555) at (6.5,-10.5) {$x_{q+m-3,r}$};
  \node[draw=none] (n10556) at (7.5,-11.5) {$S_{q+m-3,q+m-2,r}$};
  \draw[-] (n10555) to (n10556);
 \node[draw=none](n1055778) at (8.5,-12.5) {$x_{q+m-2,r}$}; 
 \node[red,very thick] (n11224455) at (9.5,-11.5){$S_{q+m-2,r,r+1}$};
 \node[draw=none] (n14515) at (10.5,-12.5){$\ddots$};
 \node[red,very thick](n415151) at (11.5,-13.5){$x_{r,r+1}$};
 \draw[densely dashed](n14515)--(n415151);
 \draw[-] (n1055778) to (n11224455);
 \node[red,very thick](n1055799) at (10.5,-10.5) {$x_{q+m-2,r+1}$};
 \draw[-] (n11224455) to (n1055799);
 \draw[-] (n10556) to (n1055778);
   \node[draw=none](n10557) at (7.5,-9.5) {$S_{q+m-3,r,r+1}$};
   \draw[densely dashed](n11224455)--(n10557);
   \draw[-] (n10555) to (n10557);
  \node[red,very thick](n141516) at (8.5,-8.5) {$x_{q+m-3,r+1}$};  
  \node[red,very thick](n1055667) at (9.5,-9.5) {$S_{q+m-3,q+m-2,r+1}$};
  \draw[-] (n1055799) to (n1055667);
  \draw[-] (n141516) to (n1055667);
 \coordinate[label=\tiny ${\mathbf{q+m-3,q+m-2,r,r+1}}$] (c4) at (8.5,-10.9); 
  \draw[-] (n10557) to (n141516);
\draw[-] (n44439) to (n446677);
\coordinate[label=\tiny ${\mathbf{q+2,q+3,r,r+1}}$] (c3) at (4,-6.4);
 \node[red,very thick] (n888) at (5,-5) {$S_{(q+2)(q+3)(r+1)}$};
 \draw[-] (n446) to (n888);
 \node[red,very thick] (n889) at (6,-6) {$x_{(q+3)(r+1)}$};
 \draw[-] (n889) to (n446677);
 \draw[-] (n888) to (n889);
 \draw[-] (n444) to (n445);
 \node[draw=none](n5) at (0,-4) {$x_{(q+1)r}$};
 \draw[-] (n445) to (n446);
\node[draw=none] (n447) at (1,-5) {$S_{(q+1)(q+2)r}$};
\draw[-] (n444) to (n447);
\draw[-] (n5) to (n447);
\draw[-] (n446) to (n448);
\coordinate[label=\tiny $ {\mathbf{q(q+1)r(r+1)}}$] (c1) at (0,-2.4);
\coordinate[label=\tiny${\mathbf{q+1,q+2,r,r+1}}$] (c2) at (2,-4.4);
 \node[draw=none](n6) at (-1,-3){$S_{q(q+1)r}$};
 \node[draw=none](n7) at (-2,-2){$x_{qr}$};
 \node[red,very thick](n8) at (-1,-1){$S_{qr(r+1)}$};

 \node[draw=none](n11) at (-2,-0){$\ddots$};
 \node[red,very thick] (n14141) at (-3,1){$x_{r,r+1}$};
 \draw[densely dashed] (n11)--(n14141);
 
 \draw[-] (n1) to (n2);
 \draw[-] (n2) to (n3);
 \draw[-] (n3) to (n4);
 \draw[-] (n4) to (n5);
 \draw[-] (n5) to (n6);
 \draw[-] (n6) to (n7);
 \draw[-] (n7) to (n8);
 \draw[-] (n8) to (n1);
 \draw[densely dashed] (n4) -- (n8);
 \end{tikzpicture}
 \caption{}
\label{figg,pattern}
 \end{figure}
 \label{cor281}
 \end{corollary}
 \begin{proof}
 Applying Theorem \ref{thm28} to the distinct vertices $A_q, A_{q+1}, ..., A_{q+m-2}, A_r, A_{r+1}$ gives \eqref{eq23}.
 \end{proof}
 \begin{remark}
 Relation \eqref{eq23} involves entries of $m-2$ adjacent diamonds of $F$. As shown in Figure \ref{figg,pattern}, those are the diamonds corresponding to the quadruples of vertices $(A_q,A_{q+1},A_r,A_{r+1})$, $(A_{q+1},A_{q+2},A_r,A_{r+1})$, ... , $(A_{q+m-3},A_{q+m-2},A_r,A_{r+1})$, and the entries in red are specifically those that appear in the relation.
 \end{remark}
\begin{example}
Let $P = (A_1,A_2,...,A_{10})$ be a cyclic $10$-gon, with anticlockwise ordering of vertices. Let $m=5$, $q=3$, and $r=9$. Then, from Corollary \ref{cor281}, we have that
\begin{equation}
\begin{split}
    S_{3,9,10}x_{4,10}x_{5,10}x_{6,10} &= S_{3,4,10}x_{5,10}x_{6,10}x_{9,10} 
    + S_{4,5,10}x_{3,10}x_{6,10}x_{9,10}\\
    &+ S_{5,6,10}x_{3,10}x_{4,10}x_{9,10}
    + S_{6,9,10}x_{3,10}x_{4,10}x_{5,10}.
    \end{split}
    \label{eq25}
\end{equation}
Figure \ref{fig11} shows three adjacent diamonds whose entries are involved in \eqref{eq25}.
\begin{figure}
 \begin{tikzpicture}
 \node[red,very thick](n1) at (0,0) {$x_{3,10}$};
 \node[draw=none](n141) at (-2,-2) {$x_{39}$};
 \coordinate[label=${\mathbf{3,4,9,10}}$] (c1) at (0,-2.4);
 \node[draw=none](n114141) at (-1,-3) {$S_{349}$};
 \draw[-] (n141) to (n114141);
 \node[red,very thick](n3141) at (-1,-1) {$S_{3,9,10}$};
 \node[draw=none](n15156) at (-2,0){$\ddots$};
     \node[red,very thick](n595885) at (-3,1) {$x_{9,10}$};
     \draw[densely dashed](n15156)--(n595885);
 \draw[-](n1) to (n3141);
 \draw[-] (n141) to (n3141);
 \node[red,very thick](n2) at (1,-1) {$S_{3,4,10}$};
 \node[red,very thick](n3) at (2,-2) {$x_{4,10}$};
 \draw[-] (n2) to (n3);
 \draw[-] (n1) to (n2);
 \node[red,very thick] (n448) at (3,-3) {$S_{4,5,10}$};
 \draw[-] (n3) to (n448);
 \node[draw=none](n4) at (1,-3) {$S_{4,9,10}$};
 \draw[densely dashed] (n3141) -- (n4);
 \draw[-] (n4) to (n3);
 \node[draw=none] (n5261) at (0,-4) {$x_{49}$};
\coordinate[label=${\mathbf{4,5,9,10}}$] (c1) at (2,-4.4); 
 \draw[-] (n114141) to (n5261);
 \draw[-] (n4) to (n5261);
 \node[draw=none](n52141) at (1,-5) {$S_{459}$};
 \draw[-] (n5261) to (n52141);
 \node[draw=none] (n444) at (2,-6) {$x_{59}$};
 \coordinate[label=${\mathbf{5,6,9,10}}$] (c1) at (4,-6.4); 
 \draw[-](n52141) to (n444);
\node[draw=none] (n4443) at (3,-7) {$S_{569}$}; 
\draw[-] (n444) to (n4443);
\node[draw=none] (n44439) at (4,-8) {$x_{69}$};
\draw[-] (n4443) to (n44439);
 \node[draw=none] (n445) at (3,-5) {$S_{5,9,10}$};
 \draw[densely dashed] (n4) -- (n445);
 \draw[-] (n444) to (n445);
 \node[red,very thick] (n446) at (4,-4) {$x_{5,10}$};
 \draw[-] (n445) to (n446);
 \draw[-] (n448) to (n446);
 \node[red,very thick] (n1561) at (5,-5) {$S_{5,6,10}$};
 \draw[-] (n446) to (n1561);
 \node[red,very thick] (n50261) at (6,-6){$x_{6,10}$};
 \draw[-] (n1561) to (n50261);
\node[red,very thick] (n446677) at (5,-7) {$S_{6,9,10}$}; 
\node[draw=none](n18185195) at (6,-8){$\ddots$};
\node[red,very thick](n1481941) at (7,-9){$x_{9,10}$};
\draw[densely dashed](n18185195)--(n1481941);
\draw[densely dashed] (n445) -- (n446677);
\draw[-] (n44439) to (n446677);
\draw[-] (n446677) to (n50261);
\end{tikzpicture}
\caption{}
\label{fig11}
\end{figure}

\end{example}
\begin{theorem}
Let $n$ be a positive integer divisible by $4$, and let $P = (A_1, A_2, ..., A_n)$ be a cyclic $n$-gon in the complex plane. Then the following $\frac{n}{2} \times \frac{n}{2}$ determinants of the corresponding plane Heronian frieze vanish:
\[
\begin{vmatrix}
     x_{21} & x_{23} & x_{25} & \cdots & x_{2,n-1}\\ 
     x_{41} & x_{43} & x_{45} & \cdots & x_{4,n-1}\\
     \vdots & \vdots & \vdots & \vdots & \vdots\\
     
     x_{n1} & x_{n3} & x_{n5} & \cdots & x_{n,n-1}\\
\end{vmatrix} = 0.\]
\end{theorem}
\begin{proof}
The proof follows from Theorem \ref{thm25} by taking $d = 2$.
\end{proof}
\begin{example} Let $n=8$, and $P = (A_1, A_2, ..., A_8)$ be a cyclic $8$-gon in the complex plane. Then the following holds for the entries of the corresponding plane Heronian frieze:
\[
\begin{vmatrix}
     x_{21} & x_{23} & x_{25} &  x_{2,7}\\ 
     x_{41} & x_{43} & x_{45} & x_{4,7}\\
     
     x_{61} & x_{63} & x_{65} & x_{6,7}\\
     x_{81} & x_{83} & x_{85}  & x_{8,7}\\
\end{vmatrix} = 0.\]
The corresponding picture is given in Figure \ref{fig12}, where the entries highlighted in red are the ones appearing  in the determinant.
\begin{figure}
\begin{center}
\begin{tikzpicture}[xscale=1,yscale=1.3]
 \node[draw=none](n1) at (0,0) {$x_{12}$};
 \node[draw=none](n2) at (0.5,-0.5) {$S_{122}$};
    
 \draw[densely dashed] (n2) -- (n24252);
 \node[draw=none](n3) at (1,-1) {$x_{22}$};
 \node[draw=none](n2491) at (2,-2){$x_{32}$};
 \node[draw=none](n148181) at (2.5,-2.5){$S_{342}$};
 \draw[-](n2491) to (n148181);
 \node[red,very thick](n511511) at (0,-2) {$x_{21}$};
 \node[draw=none](n96191) at (-3,-2) {$\cdots$};
  \node[draw=none](n1131444) at (0.5,-2.5) {$S_{231}$};
  \draw[-](n511511) to (n1131444);
  \node[draw=none](n1431414) at (1,-3) {$x_{31}$};
   \node[draw=none](n14414141) at (1.5,-3.5) {$S_{341}$};
   \draw[-](n1431414) to (n14414141);
   \draw[-](n1431414) to (n1131444);
   \draw[-](n1431414) to (n14414141);
   \draw[densely dashed](n14414141) -- (n148181);
   \node[draw=none](n41151144) at (1.5,-2.5) {$S_{312}$};
    \draw[-](n41151144) to (n1431414); 
   \draw[-](n2491) to (n41151144);
 \node[draw=none](n144) at (-2,0) {$x_{81}$};
 \node[draw=none](n141) at (-1.5,-0.5) {$S_{811}$};
 \node[draw=none](n7) at (-1,-1){$x_{11}$};
 \node[draw=none](n401591510) at (-0.5,-1.5){$S_{121}$};
 \draw[-](n511511) to (n401591510) ;
 \draw[-](n7) to (n401591510);
 \node[draw=none](n8) at (-0.5,-0.5){$S_{112}$};
 \node[draw=none](n14151) at (0.5, -1.5){$S_{212}$};
 \draw[densely dashed] (n41151144) -- (n14151);
 \draw[-] (n511511) to (n14151);
 \draw[-](n3) to (n14151);
 \draw[densely dashed] (n14151) -- (n8);
 \draw[-] (n7) to (n8);
  \draw[-] (n8) to (n1);
   \draw[-] (n1) to (n2);
    \draw[-] (n2) to (n3);

 \node[draw=none](n9) at (0.5,0.5) {$S_{123}$};
 \node[draw=none](n10) at (1,1) {$x_{13}$};
 \node[draw=none](n1489) at (-1,1) {$x_{82}$};
  
 \node[draw=none](n1456) at (-0.5,0.5) {$S_{812}$};
 \node[draw=none](n14131) at (-1.5,0.5) {$S_{812}$};
\draw[-](n144) to (n14131);
 \draw[-] (n144) -- (n141);
 \draw[-] (n141) -- (n7);
 \draw[-] (n1489) -- (n14131);
 \draw[densely dashed] (n141) -- (n1456);
 \draw[densely dashed] (n14131) -- (n8);

 \draw[-] (n1489) -- (n1456);
 \draw[-] (n1456)--(n1);
 \node[draw=none](n11) at (1.5,0.5) {$S_{123}$};
 \node[red,very thick](n12) at (2,0) {$x_{23}$};
 \node[draw=none](n13) at (1.5,-0.5){$S_{223}$};
 \node[draw=none](n415151) at (2.5,-1.5){$S_{323}$};
 \draw[-](n2491) to (n415151);
 \draw[-](n33) to (n415151);
 \draw[densely dashed] (n13) -- (n415151);
 \node[draw=none](n14) at (1,-1){$x_{22}$};
  \node[draw=none](n565875) at (0,2) {$x_{83}$};
  \node[draw=none](n9990) at (-0.5,1.5) {$S_{823}$};
  
   \draw[densely dashed] (n9) -- (n9990);
   
  \draw[-] (n1489)--(n9990);
  \draw[-] (n9990)--(n565875);
 \draw[densely dashed] (n9) -- (n13);
 \node[draw=none](n1908789) at (-1.5,1.5) {$S_{782}$};
 \node[draw=none](n19087) at (-2,2) {$x_{72}$};
 
 \draw[-] (n1908789) to (n19087);
 \draw[-](n1908789)--(n1489);
  \node[draw=none](n1908733) at (0.5,1.5) {$S_{813}$};
  \draw[-] (n565875) -- (n1908733);
  \node[draw=none](n67689) at (-1,3) {$x_{73}$};
  \node[draw=none](n161891) at (-3,3) {$\cdots$};
  \node[draw=none](n6768911) at (0,4) {$x_{74}$};
  \node[draw=none](n519196) at (0.5,4.5){$S_{745}$};
  \draw[-](n6768911) to (n519196);
  \node[draw=none](n1951591) at (1,5){$x_{75}$};
  \draw[-](n519196) to (n1951591);
  \node[draw=none](n2171) at (1.5,4.5){$S_{785}$};
  \draw[-](n131) to (n2171);
  \node[draw=none](n1387445) at (10,4) {$x_{41}$};
  \draw[-](n1951591) to (n2171);
  \node[draw=none](n1387445000) at (9.5,3.5) {$S_{481}$};
  \node[draw=none](n68484123) at (11,3) {$x_{51}$};
  \node[draw=none](n11969) at (13,3){$\cdots$};
  \node[draw=none](n68484) at (10.5,2.5) {$S_{581}$};
 \node[draw=none] (n15116) at (12,2) {$x_{61}$};
 \node[draw=none] (n19196199) at (12.5,1.5) {$S_{671}$};
 
 \draw[-] (n15116)--(n19196199);
 \node[draw=none] (n2682827) at (12.5,0.5) {$S_{781}$};
 \node[draw=none](n1519691681) at (13,1) {$x_{71}$};
 \draw[-] (n1519691681)--(n19196199);
 \draw[-] (n2682827)--(n1519691681);

 \node[draw=none](n5161616) at (11.5,2.5){$S_{561}$};
 \draw[densely dashed] (n108) -- (n5161616);
 \node[draw=none](n52622727) at (11.5,1.5){$S_{681}$};
 \draw[densely dashed] (n52622727)--(n2682827);
 \draw[densely dashed] (n68484)--(n52622727);
 \draw[-] (n68484123)--(n5161616);
 \draw[-](n5161616)--(n15116);
 \draw[-](n15116)--(n52622727);
  \node[draw=none](n6848444) at (10.5,3.5) {$S_{451}$};
  \draw[densely dashed] (n6848444)--(n62062602);
  \draw[-] (n6848444)--(n68484123);
  \draw[-](n1387445)--(n6848444);
   \node[draw=none](n676891123) at (-0.5,3.5) {$S_{734}$}; 
   \draw[-] (n68484123)--(n68484);
   
   \draw[-] (n6848444)--(n68484123);
   
   \draw[-] (n6768911)--(n676891123);
    \draw[-] (n67689)--(n676891123);
    \node[draw=none](n676891199) at (0.5,3.5) {$S_{784}$};
    \draw[densely dashed] (n676891199) -- (n156196119);
     \draw[-] (n6768911)--(n676891199);
  \node[draw=none](n19087337) at (1,3) {$x_{84}$};
  \draw[-] (n19087337)--(n676891199);
  \node[draw=none](n1908739) at (-0.5,2.5) {$S_{783}$};
  \draw[densely dashed] (n676891199)--(n1908739);
  \draw[densely dashed] (n1908789)--(n1908739);
  \draw[-] (n67689)--(n1908739);
  \draw[-] (n1908739)--(n565875);
  \node[draw=none](n190865) at (-1.5,2.5) {$S_{723}$};
   \draw[-] (n19087) -- (n190865);
   \draw[-] (n67689) -- (n190865);
  \node[draw=none](n190873376) at (0.5,2.5) {$S_{834}$};
  \draw[densely dashed] (n68484)--(n1387445000);
  \draw[-] (n1387445000)--(n1387445);
 
  \draw[densely dashed] (n676891123)--(n190873376);
   \draw[-] (n19087337) to (n190873376);
   \node[draw=none](n1908733996) at (1.5,2.5) {$S_{814}$};
    \draw[-] (n565875) -- (n190873376);
    \draw[densely dashed] (n1908733) -- (n1908733996);
  \draw[-] (n1) to (n9);
   \draw[-] (n9) to (n10);
   \draw[-] (n19087337)--(n1908733996);
   \draw[-] (n1908733)--(n10);
    \draw[-] (n1) to (n9);
    \draw[-] (n10) to (n11);
     \draw[-] (n11) to (n12);
      \draw[-] (n12) to (n13);
       \draw[-] (n13) to (n14);
      \draw[densely dashed] (n9990)--(n190865); 
 
\draw[densely dashed] (n2) -- (n11);
\draw[densely dashed] (n9) -- (n13);

\node[draw=none](n15) at (1.5,1.5) {$S_{134}$};
 \draw[densely dashed] (n190873376) -- (n15);
 \node[draw=none](n16) at (2,2) {$x_{14}$};
 \draw[-] (n16) to (n1908733996);
 \node[draw=none](n17) at (2.5,1.5) {$S_{124}$};
 \node[draw=none](n18) at (3,1) {$x_{24}$};
 \node[draw=none](n19) at (2.5,0.5){$S_{234}$};
 
  \draw[-] (n10) to (n15);
   \draw[-] (n15) to (n16);
    \draw[-] (n16) to (n17);
     \draw[-] (n17) to (n18);
      \draw[-] (n18) to (n19);
\node[draw=none](n131) at (2,4) {$x_{85}$};    
\node[draw=none](n105161) at (2.5,4.5){$S_{856}$};
\draw[-](n131) to (n105161);
\node[draw=none](n195196) at (3,5){$x_{86}$};
\draw[-](n105161) to (n195196);
\node[draw=none](n195116) at (3.5,4.5){$S_{816}$};
\draw[-](n195196) to (n195116);
\node[draw=none](n131890) at (1.5,3.5) {$S_{845}$};
\draw[densely dashed] (n131890) -- (n51519691);
\draw[-] (n131) to (n131890);
\draw[-] (n19087337) to (n131890);
\node[draw=none](n138761) at (2.5,3.5) {$S_{815}$};
\draw[-] (n131)--(n138761);
\draw[densely dashed] (n1908733996)--(n138761);
\draw[-] (n12) to (n19);
      \draw[densely dashed] (n11) -- (n17);
\draw[densely dashed] (n15) -- (n19);
\draw[densely dashed] (n1456) -- (n1908733);

 \node[draw=none](n20) at (2.5,2.5) {$S_{145}$};
 \draw[densely dashed] (n131890)--(n20);
 \node[draw=none](n21) at (3,3) {$x_{15}$};
 \draw[-] (n138761)--(n21);
 \node[draw=none](n22) at (3.5,2.5) {$S_{125}$};
 \node[red,very thick](n23) at (4,2) {$x_{25}$};
 \node[draw=none](n24) at (3.5,1.5){$S_{245}$};
 \draw[densely dashed] (n20) -- (n24);

 \draw[densely dashed] (n17) -- (n22);
  \draw[-] (n16) to (n20);
   \draw[-] (n20) to (n21);
    \draw[-] (n21) to (n22);
    \draw[-] (n22) to (n23); 
     \draw[-] (n23) to (n24);
      \draw[-] (n24) to (n18);

 \node[draw=none](n25) at (3.5,3.5) {$S_{156}$};
 \node[draw=none](n26) at (4,4) {$x_{16}$};
 \draw[-](n195116) to (n26);
 \node[draw=none](n115191) at (4.5,4.5){$S_{167}$};
 \draw[-](n26) to (n115191);
 \node[draw=none](n51061010) at (5,5){$x_{17}$};
 \draw[-](n115191) to (n51061010);
 \node[draw=none] (n15961016) at (5,6){$\cdots$};
 \node[draw=none](n195196) at (3,6){$\cdots$};
 \node[draw=none](n19915916) at (5.5,4.5){$S_{127}$};
 \draw[-](n51) to (n19915916);
 \draw[-](n51061010) to (n19915916);
 \node[draw=none](n27) at (4.5,3.5) {$S_{126}$};
 \draw[densely dashed] (n27) -- (n12619792);
  \draw[densely dashed] (n25) -- (n15161981);
  \draw[densely dashed] (n151661) -- (n138761);
 \node[draw=none](n28) at (5,3) {$x_{26}$};
 \node[draw=none](n29) at (4.5,2.5){$S_{256}$};
 
  \draw[-] (n21) to (n25);
   \draw[-] (n25) to (n26);
    \draw[-] (n26) to (n27);
     \draw[-] (n27) to (n28);
      \draw[-] (n28) to (n29);
       \draw[-] (n29) to (n23);

\draw[densely dashed] (n27) -- (n22);
\draw[densely dashed] (n25) -- (n29);

 \node[draw=none](n30) at (3.5,0.5) {$S_{234}$};
 \node[draw=none](n31) at (4,0) {$x_{34}$};
 \node[draw=none](n32) at (3.5,-0.5) {$S_{334}$};
 \node[draw=none](n411616) at (4.5,-1.5){$S_{434}$};
 \draw[-](n41) to (n411616);
 \draw[densely dashed] (n32) -- (n411616);
\node[draw=none](n33) at (3,-1) {$x_{33}$};

  \node[draw=none](n34) at (2.5,-0.5) {$S_{233}$};
  \node[draw=none] (n61611) at (1.5,-1.5){$S_{232}$};
 \draw[densely dashed] (n1131444) -- (n61611);
  \draw[-](n3) to (n61611);
   \draw[densely dashed] (n34) -- (n61611);
\draw[-](n2491) to (n61611);
  \draw[-] (n12) to (n34);
    \draw[-] (n34) to (n33);
     \draw[-] (n18) to (n19);
      \draw[-] (n18) to (n30);
       \draw[-] (n30) to (n31);
        \draw[-] (n31) to (n32);
         \draw[-] (n32) to (n33);
         \draw[densely dashed] (n34) -- (n30);
\draw[densely dashed] (n19) -- (n32);

 \node[draw=none](n37) at (4.5,1.5) {$S_{235}$};
 \draw[densely dashed] (n34) -- (n37);

 \node[draw=none](n38) at (5.0,1) {$x_{35}$};
 \node[draw=none](n39) at (4.5,0.5) {$S_{345}$};
  \draw[-] (n23) to (n37);
   \draw[-] (n37) to (n38);
    \draw[-] (n38) to (n39);
   \draw[-] (n31) to (n39); 
\draw[densely dashed] (n24) -- (n39);
  \node[draw=none](n40) at (4.5,-0.5) {$S_{344}$};
  \node[draw=none](n56106191) at (3.5,-1.5){$S_{343}$};
  \draw[densely dashed](n148181)--(n56106191);
  \draw[-](n33) to (n56106191);
  \draw[densely dashed] (n40) -- (n56106191);
  \node[draw=none](n41) at (5,-1) {$x_{44}$};
  \node[draw=none](n14451919) at (6,-2){$x_{54}$};
  \draw[-](n4101951) to (n14451919);
  \draw[-](n05191061) to (n14451919);
  \node[draw=none](n14947218199) at (6.5,-2.5){$S_{564}$};
  \draw[-](n14451919) to (n14947218199);
  \draw[densely dashed](n5619196591) --(n14947218199) ;
  \node[red,very thick](n14919) at (4,-2){$x_{43}$};
  \node[draw=none](n11384919) at (4.5,-2.5){$S_{453}$};
  \draw[-](n14919) to (n11384919);
  \node[draw=none](n14472829919) at (5,-3){$x_{53}$};
  \draw[-](n11384919) to (n14472829919);
  \node[draw=none](n1149472819) at (5.5,-3.5){$S_{563}$};
  \draw[-](n14472829919) to (n1149472819);
  \draw[densely dashed](n14947218199) --(n1149472819) ;
  \node[draw=none](n11414919) at (5.5,-2.5){$S_{534}$};
    \draw[-](n11414919) to (n14472829919);
  \draw[-](n14451919) to (n11414919);
  \draw[densely dashed] (n411616)--(n11414919);
  \draw[-](n411616) to (n14919);
  \draw[-](n56106191) to (n14919);
 \draw[-] (n41) to (n40);
\node[draw=none](n141591) at (3,-3) {$x_{42}$};
\node[draw=none](n14947281129) at (3.5,-3.5){$S_{452}$};
\draw[-](n141591) to (n14947281129);
\draw[densely dashed](n11384919)--(n14947281129);
\node[draw=none](n14947342819) at (4,-4){$x_{52}$};
\draw[-](n14947281129) to (n14947342819) ;
\node[draw=none](n14947299) at (4.5,-4.5){$S_{562}$};
\draw[densely dashed](n1149472819) -- (n14947299);
\draw[-] (n14947342819) to (n14947299);
\node[draw=none](n098472819) at (4.5,-3.5){$S_{523}$};
\draw[-](n14947342819) to (n098472819) ;
draw[-](n098472819) to (n14947342819) ;
\draw[-](n14472829919) to (n098472819);
\draw[-](n148181) to (n141591);
\node[draw=none](n94191) at (3.5,-2.5){$S_{423}$};
\draw[densely dashed](n94191) -- (n098472819);
\draw[densely dashed](n415151)--(n94191);
\draw[-] (n141591) to (n94191);
\draw[-] (n14919) to (n94191);
\node[red,very thick](n251951591) at (2,-4){$x_{41}$};

\draw[-](n14414141) to (n251951591);
\node[draw=none](n1494728359) at (2.5,-4.5){$S_{451}$};
\node[draw=none](n19816169) at (-0.5,-4.5) {$\cdots$};
\draw[-](n251951591) to (n1494728359) ;
\draw[densely dashed] (n14947281129) -- (n1494728359);
\node[draw=none](n311149472819) at (3,-5){$x_{51}$};
\draw[-](n1494728359) to (n311149472819) ;
\node[draw=none](n779472819) at (3.5,-5.5){$S_{561}$};
\draw[-](n311149472819) to (n779472819) ;
\draw[densely dashed](n14947299) -- (n779472819);
\node[draw=none](n149472819) at (3.5,-4.5){$S_{512}$};
\draw[-](n14947342819) to (n149472819) ;
\draw[-](n311149472819) to (n149472819);
\node[draw=none](n19951951) at (2.5,-3.5){$S_{412}$};
\draw[-](n141591) to (n19951951);
\draw[-](n19951951) to (n251951591);
\draw[densely dashed] (n149472819) -- (n19951951);
\draw[densely dashed](n41151144)--(n19951951);
  \node[draw=none](n44) at (5.5,2.5) {$S_{236}$};
  \node[draw=none](n45) at (6,2) {$x_{36}$};
 \node[draw=none](n46) at (5.5,1.5) {$S_{356}$};
  \draw[-] (n28) to (n44);
   \draw[-] (n44) to (n45);
    \draw[-] (n45) to (n46);
     \draw[-] (n46) to (n38);
\draw[densely dashed] (n44) -- (n37);
\draw[densely dashed] (n29) -- (n46);

\node[draw=none](n47) at (5.5,0.5) {$S_{345}$};
\node[red,very thick](n48) at (6,0) {$x_{45}$};
\node[draw=none](n49) at (5.5,-0.5) {$S_{445}$};
\node[draw=none](n4101951) at (6.5,-1.5){$S_{545}$};
\draw[densely dashed] (n49) -- (n4101951);
\node[draw=none](n05191061) at (5.5,-1.5){$S_{454}$};
\draw[-](n41) to (n05191061);
\draw[densely dashed](n11384919)--(n05191061);
 \draw[-] (n38) to (n47);
 \draw[-] (n47) to (n48);
  \draw[-] (n48) to (n49);
   \draw[-] (n49) to (n41);
   \draw[-] (n31) to (n40);
\draw[densely dashed] (n40) -- (n47);
\draw[densely dashed] (n39) -- (n49);

\node[draw=none](n50) at (5.5,3.5) {$S_{267}$};
\draw[densely dashed] (n50) -- (n1261979290);
\node[red,very thick](n51) at (6,4) {$x_{27}$};
\node[draw=none](n511718) at (7,5) {$x_{28}$};
\node[draw=none](n511166) at (7.5,4.5) {$S_{238}$};
\draw[-](n511718) to (n511166);
\node[draw=none](n52) at (6.5,3.5) {$S_{237}$};
\node[draw=none](n51616109) at (6.5,4.5){$S_{278}$};
\draw[-](n511718) to (n51616109);
\draw[-](n51) to (n51616109);
\node[draw=none](n552929062) at (8.5,4.5){$S_{381}$};
\draw[densely dashed] (n1387445000)--(n552929062);
\draw[densely dashed] (n52)-- (n0596020620);
\node[draw=none](n53) at (7,3) {$x_{37}$};
\node[draw=none](n100) at (8,4) {$x_{38}$};
\draw[-] (n511166) to (n100);
\draw[-](n552929062) to (n100);
\node[draw=none](n587787671) at (9,5) {$x_{31}$};
\draw[-](n552929062) to (n587787671);
\node[draw=none](n5969619) at (9.5,4.5){$S_{341}$};
\draw[-] (n1387445) to (n5969619);
\draw[-](n587787671) to (n5969619);

\node[draw=none](n101) at (7.5,3.5) {$S_{378}$};
\draw[densely dashed] (n101) -- (n51616109);

\draw[-] (n53) to (n101);
\draw[-] (n100) to (n101);
\draw[densely dashed] (n44) -- (n52);

\node[draw=none](n102) at (8.5,3.5) {$S_{348}$};
\draw[densely dashed] (n102)--(n529296206);
\draw[-] (n100) to (n102);
\node[draw=none](n103) at (9,3) {$x_{48}$};
\draw[-] (n102) to (n103);

\node[draw=none](n104) at (8.5,2.5) {$S_{478}$};
\draw[-] (n103) to (n104);
\node[draw=none](n105) at (9.5,2.5) {$S_{458}$};
\node[draw=none](n106) at (10,2) {$x_{58}$};
\node[draw=none](n106) at (10,2) {$x_{53}$};
\node[draw=none](n108) at (10.5,1.5) {$S_{568}$};

\node[draw=none](n109) at (11,1) {$x_{68}$};
\draw[-](n52622727)--(n109);
\node[draw=none](n110) at (10.5,0.5) {$S_{678}$};

\node[red,very thick](n111) at (10,0) {$x_{67}$};
\node[draw=none](n112) at (9.5,0.5) {$S_{567}$};
\draw[densely dashed] (n101) -- (n104);
\draw[densely dashed] (n105)--(n6848444);
\node[draw=none](n107) at (9.5,1.5) {$S_{578}$};
\draw[-] (n105) to (n103);
\draw[-] (n105) to (n106);
\draw[-] (n106) to (n107);
\node[draw=none](n54) at (6.5,2.5) {$S_{367}$};
\draw[densely dashed] (n104) -- (n107);
 \draw[densely dashed] (n50) -- (n54);
\draw[-] (n106) to (n108);
\draw[densely dashed] (n107) -- (n110);
\draw[-] (n108) to (n109);
\draw[-] (n110) to (n109);
\draw[-] (n110) to (n111);
\draw[-] (n111) to (n112);

 \node[draw=none](n55) at (6.5,1.5) {$S_{346}$};
 \draw[densely dashed] (n47) -- (n55);
 \node[draw=none](n56) at (7,1) {$x_{46}$};
  \node[draw=none](n57) at (6.5,0.5) {$S_{456}$};
  \draw[-] (n28) to (n50);
  \draw[-] (n50) to (n51);
  \draw[-] (n51) to (n52);
  \draw[-] (n52) to (n53);
  \draw[-] (n53) to (n54);
  \draw[-] (n54) to (n45);
  \draw[-] (n45) to (n55);
  \draw[-] (n55) to (n56);
   \draw[-] (n56) to (n57);
 \draw[densely dashed] (n46) -- (n57);  

\node[draw=none](n58) at (7.5,2.5) {$S_{347}$};
 \node[red,very thick](n59) at (8,2) {$x_{47}$};
 \draw[-] (n59) to (n104);
  \node[draw=none](n60) at (7.5,1.5) {$S_{467}$};
  \draw[densely dashed] (n54) -- (n60);
   
\draw[densely dashed] (n58) -- (n55);
\draw[densely dashed] (n58) -- (n102);
  
\node[draw=none](n61) at (7.5,0.5) {$S_{456}$};
 \node[draw=none](n62) at (8,0) {$x_{56}$};
  \node[draw=none](n63) at (7.5,-0.5) {$S_{556}$};
  \node[draw=none](n50626921) at (8.5,-1.5){$S_{656}$};
  \draw[-](n70) to (n50626921);
  \node[draw=none](n5619196591) at (7.5,-1.5){$S_{565}$};
  \draw[densely dashed] (n63) -- (n50626921);
\node[draw=none](n64) at (7,-1) {$x_{55}$};
\draw[-](n5619196591) to (n64);
\draw[-](n4101951) to (n64);
 \node[draw=none](n65) at (6.5,-0.5) {$S_{455}$};
 \draw[densely dashed] (n65) -- (n05191061);
  \draw[-] (n62) to (n61);
   \draw[-] (n56) to (n61);
   \draw[-] (n48) to (n65);
   \draw[-] (n64) to (n65);
  \draw[-] (n48) to (n57); 
  \draw[-] (n63) to (n64);
  \draw[-] (n62) to (n63);
\draw[densely dashed] (n61) -- (n65);
\draw[densely dashed] (n57) -- (n63);

 \draw[-] (n53) to (n58);
 \draw[-] (n58) to (n59);
 \draw[-] (n59) to (n60);
 \draw[-] (n60) to (n56);

\node[draw=none](n69) at (8.5,-0.5) {$S_{566}$};
\draw[densely dashed] (n5619196591)--(n69);
\node[draw=none](n70) at (9,-1) {$x_{66}$};
\node[red,very thick](n419910) at (8,-2){$x_{65}$};
\draw[-](n5619196591) to (n419910);
\draw[-](n50626921) to (n419910);
\node[draw=none](n149478472819) at (8.5,-2.5){$S_{675}$};
\draw[-](n419910) to (n149478472819) ;
\node[draw=none](n529529) at (7,-3) {$x_{64}$};
\draw[-](n14947218199) to (n529529);
 \node[draw=none](n1451614) at (7.5,-3.5) {$S_{674}$};
 \draw[-](n529529) to (n1451614);
 \draw[densely dashed](n149478472819) -- (n1451614);
\node[draw=none](n14947281932) at (7.5,-2.5){$S_{645}$};
\draw[-](n419910) to (n14947281932) ;
\draw[-](n529529) to (n14947281932) ;
\draw[densely dashed](n4101951) -- (n14947281932); 
\node[red,very thick](n52952249) at (6,-4) {$x_{63}$};
\draw[-](n1149472819) to (n52952249);
 \node[draw=none](n515144) at (6.5,-4.5) {$S_{673}$};
\draw[-] (n52952249) to (n515144);
 \draw[densely dashed](n1451614) -- (n515144);
\node[draw=none](n149131472819) at (6.5,-3.5){$S_{634}$};
\draw[-](n52952249) to (n149131472819) ;
\draw[-](n529529) to (n149131472819);
\draw[densely dashed](n11414919) --(n149131472819) ;
\node[draw=none](n529521319) at (5,-5) {$x_{62}$};
\draw[-](n14947299) to (n529521319);
\node[draw=none](n56117) at (5.5,-5.5){$S_{672}$};
\draw[-](n529521319) to (n56117);
\draw[densely dashed](n515144) -- (n56117);
\node[draw=none](n149472810989) at (5.5,-4.5){$S_{623}$};
\draw[-](n529521319) to (n149472810989) ;
\draw[-](n52952249) to (n149472810989) ;
\draw[densely dashed](n098472819) -- (n149472810989) ;
\node[red,very thick](n529524229) at (4,-6) {$x_{61}$};
\node[draw=none](n19519619) at (4.5,-6.5){$S_{671}$};
\draw[-](n529524229) to (n19519619);
\draw[densely dashed](n56117)--(n19519619);
\draw[-](n779472819) to (n529524229);
 \node[draw=none](n141314) at (4.5,-5.5) {$S_{612}$};
 \draw[-](n529524229) to (n141314);
 \draw[-](n529521319) to (n141314);
 \draw[densely dashed](n149472819) -- (n141314);
\node[draw=none](n73) at (8.5,1.5) {$S_{457}$};
\node[draw=none](n74) at (9,1) {$x_{57}$};

\node[draw=none](n75) at (8.5,0.5) {$S_{567}$};
\draw[densely dashed] (n61) -- (n73);
\draw[densely dashed] (n73) -- (n105);
 \draw[-] (n59) to (n73);
  \draw[-] (n73) to (n74);
   \draw[-] (n74) to (n112);
    \draw[-] (n75) to (n62);
    \draw[-] (n107) to (n74);
\draw[-] (n74) to (n75);
\draw[densely dashed] (n60) -- (n75);

\node[draw=none](n200) at (9.5,-0.5) {$S_{667}$};
\node[draw=none](n69196919619) at (10.5,-1.5){$S_{767}$};
\draw[-](n204) to (n69196919619) ;
\node[draw=none](n14929518919) at (10,-2){$x_{76}$};
\draw[-](n69196919619) to (n14929518919) ;
\node[draw=none](n18972819) at (10.5,-2.5){$S_{786}$};
\draw[densely dashed](n519515195119) -- (n18972819);
\draw[-](n14929518919) to (n18972819) ;
\node[draw=none](n149472819) at (9,-3){$x_{75}$};
\draw[-](n149478472819) to (n149472819) ;
 \node[draw=none](n141414) at (9.5,-3.5) {$S_{785}$};
 \draw[-](n149472819) to (n141414);
 \draw[densely dashed](n18972819) -- (n141414);
\node[draw=none](n149472819) at (9.5,-2.5){$S_{756}$};
\draw[-](n14929518919) to (n149472819);
\draw[densely dashed](n50626921) -- (n149472819) ;
\node[draw=none](n149472834119) at (8,-4){$x_{74}$};
\draw[-] (n1451614) to (n149472834119);
\node[draw=none](n4919410) at (8.5,-4.5){$S_{784}$};
\draw[-](n149472834119) to (n4919410);
\draw[densely dashed](n141414) -- (n4919410);
 \node[draw=none](n148794) at (8.5,-3.5) {$S_{745}$};
 \draw[-](n149472834119)  to (n148794);
 \draw[densely dashed](n14947281932) -- (n148794);
\node[draw=none](n14944172819) at (7,-5){$x_{73}$};
\draw[-](n515144) to (n14944172819) ;
\node[draw=none](n159519) at (7.5,-5.5){$S_{783}$};

\draw[-](n14944172819) to (n159519);
\draw[densely dashed](n4919410) -- (n159519);
 \node[draw=none](n1414114) at (7.5,-4.5) {$S_{734}$};
 \draw[-](n149472834119) to (n1414114);
 \draw[-](n14944172819) to (n1414114);
 \draw[densely dashed](n149131472819) -- (n1414114);
\node[draw=none](n149472819112) at (6,-6){$x_{72}$};
\node[draw=none](n95191916) at (5,-7){$x_{71}$};
\node[draw=none](n59191) at (3,-7){$\cdots$};
\draw[-](n19519619) to (n95191916);
\node[draw=none](n1991691) at (5.5,-7.5){$S_{781}$};
\draw[-](n95191916) to (n1991691);
\node[draw=none](n5106116010) at (5.5,-6.5){$S_{712}$};
\draw[-](n95191916) to (n5106116010) ;
\draw[-](n149472819112) to (n5106116010) ;
\draw[densely dashed](n141314) -- (n5106116010);
\node[draw=none](n510601061) at (6.5,-6.5){$S_{782}$};

\draw[-](n149472819112) to (n510601061);
\draw[densely dashed](n1991691) --(n510601061) ;
\draw[densely dashed](n159519) -- (n510601061) ;
\draw[-](n56117) to (n149472819112);
\node[draw=none](n159191) at (6.5,-5.5){$S_{723}$};
\draw[-](n14944172819) to (n159191);
\draw[-](n149472819112) to (n159191);
\draw[densely dashed](n149472810989) -- (n159191);
\draw[densely dashed] (n200)--(n69196919619);
\node[draw=none](n1595159961) at (9.5,-1.5){$S_{676}$};
\draw[-](n14929518919) to (n1595159961) ;
\draw[-](n70) to (n1595159961);
\draw[densely dashed](n149478472819) --(n1595159961) ;
\node[draw=none](n202) at (12,0) {$x_{78}$};
\node[draw=none](n201) at (11.5,0.5) {$S_{678}$};
\node[draw=none](n203) at (11.5,-0.5) {$S_{778}$};
\node[draw=none](n9519519) at (12.5,-1.5){$S_{878}$};
\draw[densely dashed] (n203)--(n9519519);
\node[draw=none](n519515195119) at (11.5,-1.5){$S_{787}$};
\node[draw=none](n204) at (11,-1) {$x_{77}$};
\draw[-](n519515195119) to (n204);
\node[draw=none](n205) at (10.5,-0.5) {$S_{677}$};
\draw[densely dashed] (n1595159961)--(n205);
\node[draw=none](n210) at (13,-1) {$x_{88}$};
\draw[-](n210) to (n9519519) ;
\node[red,very thick](n52952219) at (12,-2) {$x_{87}$};
\draw[-](n9519519) to (n52952219);
\draw[-](n519515195119) to (n52952219) ;
\node[draw=none](n529521129) at (11,-3) {$x_{86}$};
\draw[-](n18972819) to (n529521129) ;
\node[draw=none](n90472819) at (11.5,-2.5){$S_{867}$};
\draw[-](n529521129) to (n90472819);
\draw[-](n52952219) to (n90472819);
\draw[densely dashed](n69196919619) -- (n90472819);
\node[red,very thick](n123529529) at (10,-4) {$x_{85}$};
\draw[-](n141414) to (n123529529) ;
 \node[draw=none](n144561) at (10.5,-3.5) {$S_{856}$};
 \draw[-](n123529529) to (n144561);
 \draw[-](n529521129) to (n144561);
 \draw[densely dashed](n149472819) -- (n144561);
\node[draw=none](n52951229) at (9,-5) {$x_{84}$};
\draw[-](n4919410) to (n52951229);
\node[draw=none](n2519619) at (9.5,-4.5){$S_{845}$};
\node[draw=none](n18618) at (11.5,-4.5){$\cdots$};
\draw[-](n52951229) to (n2519619);
\draw[-](n123529529) to (n2519619);
\draw[densely dashed](n148794) -- (n2519619);
\node[red,very thick](n512029529) at (8,-6) {$x_{83}$};
\draw[-](n512029529) to (n159519);
\node[draw=none](n51095196) at (8.5,-5.5) {$S_{834}$};
\draw[-](n512029529) to (n51095196);
\draw[-](n52951229) to (n51095196);
\draw[densely dashed](n1414114) -- (n51095196) ;
\node[draw=none](n5295291122) at (7,-7) {$x_{82}$};
\draw[-](n5295291122) to (n510601061);
\node[draw=none](n1951961) at (7.5,-6.5) {$S_{823}$};
\node[draw=none](n1969160) at (9.5,-6.5){$\cdots$};
\draw[-](n5295291122) to (n1951961);
\draw[-](n512029529) to (n1951961);
\draw[densely dashed](n159191)--(n1951961);
\node[red,very thick](n5292441529) at (6,-8) {$x_{81}$};
\draw[-](n1991691) to (n5292441529) ;
\node[draw=none](n95619691) at (6.5,-7.5){$S_{812}$};
\draw[-](n5292441529) to (n95619691) ;
\draw[-](n5295291122) to (n95619691);
\draw[densely dashed](n5106116010) -- (n95619691) ;
\node[draw=none](n211) at (12.5,-0.5) {$S_{788}$};
\draw[densely dashed] (n211)--(n519515195119);
\draw[densely dashed] (n75) -- (n200);
\draw[-] (n202) to (n211);
\draw[-] (n2682827)--(n202);
\draw[densely dashed] (n203) -- (n110);
\draw[-] (n111) to (n200);
\draw[-] (n200) to (n70);
\draw[-] (n69) to (n70);
\draw[-] (n62) to (n69);
\draw[-] (n109) to (n201);
\draw[-] (n201) to (n202);
\draw[-] (n202) to (n203);
\draw[-] (n203) to (n204);
\draw[-] (n204) to (n205);
\draw[-] (n205) to (n111);
\draw[-] (n210) to (n211);
 \draw[-] (n106)--(n68484);
\draw[densely dashed] (n69) -- (n112);
\draw[densely dashed] (n108) -- (n112);
\draw[densely dashed] (n201) -- (n205);
 \draw[-] (n1387445000)--(n103);
 \draw[densely dashed] (n19196199)--(n201);
\end{tikzpicture}
\end{center}
\caption{}
\label{fig12}
\end{figure}
\end{example}
\newpage
\section{Main theorem}
In this section we will state and prove a result that gives us an algebraic relation between entries of a polygonal Heronian frieze that arises from a cylic $n$-gon, where $n>4$.\\ Before stating the result, recall that if $i \leq 0$, $j \leq 0$, or $k \leq 0$ we have, by convention, that $S_{ijk} = 1$ and $x_{ij} = 1$. Additionally, we adopt a convention that $x_{i,j}^p = 1$ for $p < 0$ and $i,j \in \{1,2,...,n\}$.
\begin{theorem}
 Let $P = (A_1, A_2, ..., A_n)$ be a cyclic $n$-gon with vertices ordered anticlockwise, and $n>4$ is an even number. Let $F$ be the corresponding polygonal Heronian frieze, as in Definition \ref{def16}. Then the following holds for the entries of $F$: 

\begin{equation}
\sum_{m=1}^{n}(-1)^{m+1}x(m)S(m) = 0,
\end{equation}
where
\begin{equation*}
    x(m) = \begin{cases}
    x_{12}x_{34}^c \prod_{\substack{l=m \\ l \text{ even }}}^{n-2} x_{l,l+1}^{\frac{l-2}{2}} \prod_{\substack{l=5 \\ l \text{ odd}}}^{m-1} x_{l,l+1}^{\frac{l-3}{2}} , &\text{ for $m$ even;} \\ \\
    x_{12}^{a}\prod_{\substack{l=m-1 \\ l \text{ even}}}^{n-2}x_{l,l+1}^{\frac{l-2}{2}}\prod_{\substack{l=5 \\ l \text{ odd}}}^m x_{l,l+1}^{\frac{l-3}{2}}, &\text{for $m$ odd, $n=6$; } \\ \\
    x_{12}^{a}x_{34}^bx_{m-2,m-1}^{\frac{m-5}{2}}x_{m,m+1}^{\frac{m-3}{2}}  \\ \qquad \qquad \prod_{\substack{l=m-1 \\ l \text{ even}}}^{n-2}x_{l,l+1}^{\frac{l-2}{2}}\prod_{\substack{l=5 \\ l \text{ odd}}}^{m-4}x_{l,l+1}^{\frac{l-5}{2}}, &\text{for $m$ odd, $n>6$ ;}
      \end{cases}
\end{equation*}
\begin{equation*}
S(m) = \begin{cases}
\left(\prod_{k=1}^{m-1}\left(\prod_{\substack{l = k+1 \\ l \text{ even}}}^{m-2} S_{k,l,l+1} \prod_{\substack{l=m+1 \\ l \text{ odd}}}^{n-1}S_{k,l,l+1}\right)\right) \left(\prod_{k=m}^{n-3}\prod_{\substack{l=k+2 \\ l \text{ odd}}}^{n-1}S_{k+1,l,l+1}\right), &\text{for $m$ even;} \\ \\
  S_{m-2,m-1,m+1}\left(\prod_{\substack{l=m+2 \\ l \text{ odd}}}^{n-1}S_{m-1,l,l+1}S_{m-2,l,l+1}S_{m-3,l,l+1}\right)\left(\prod_{k=m}^{n-3}\prod_{\substack{l=k+2 \\ l \text{ odd}}}^{n-1}S_{k+1,l,l+1}\right) \\ \qquad \qquad  
\left[\prod_{k=1}^{m-4}\left((S_{k,k+1,k+2}S_{k,k+1,m-1}S_{k,k+1,m+1})^{(k \text{ mod 2})} \right. \right.  \\ \qquad \qquad  \left. \left. \prod_{\substack{l=k+2 \\ l \text{ even}}}^{m-2}S_{k,l,l+1} \prod_{\substack{l=m+2 \\ l \text{ odd}}}^{n-1}S_{k,l,l+1}\right)\right], &\text{ for $m$ odd},\end{cases}\end{equation*}
where
\\ \\
(i) $a = 1$ for $m \in \{1,3\}$, and $a = 0$ otherwise;\\
(ii) $b = 1$ for $m \in \{1,3,5\}$, and $b = 0$ otherwise; \\
(iii) $c = 0$ for $n = 6$, and $c = 1$ for $n \geq 8$.
\label{impthm}
\end{theorem}
Before giving the proof of the theorem, we will prove a few lemmas and propositions.
\begin{lemma}
    Let $A, B_{l,l+1} \in \mathbb{R}$, where $l \in \{1,2,...,n-1\}$. If $a, b \in \mathbb{Z}$, where $a$ is even, $b$ is odd, and $a \leq b$, then the following hold:
    \begin{equation}
        \prod_{k=a}^{b}\prod_{\substack{l=k+1 \\ l \text{ even}}}^{b} \frac{A}{B_{l,l+1}} = \prod_{\substack{k=a \\k \text{ even}}}^{b}\prod_{\substack{l=k+1 \\ l \text{ even}}}^{b} \frac{A^2}{B_{l,l+1}^2},
        \label{eq111}
    \end{equation} 
    \begin{equation}
     \prod_{k=a}^{b}\prod_{\substack{l=k+2 \\ l \text{ odd }}}^{b+2}\frac{A}{B_{l,l+1}} = \prod_{\substack{k=a \\k \text{ even}}}^{b}\prod_{\substack{l=k+2 \\ l \text{ odd}}}^{b+2} \frac{A^2}{B_{l,l+1}^2}
     \label{eq112}
     \end{equation}
     \label{lem32}
\end{lemma}
\begin{proof}
   Note that 
   \begin{equation*}
   \prod_{\substack{l=k+1 \\ l even}}^{b} \frac{A}{B_{l,l+1}} = \frac{A}{B_{a+2,a+3}} \frac{A}{B_{a+4,a+5}} \cdots \frac{A}{B_{b-1,b}},
   \end{equation*}
    \\for $k = a$ and $k = a+1$. \\ \\
   Similarly,
   \begin{equation*}
   \prod_{\substack{l=k+1 \\ l even}}^{b} \frac{A} {B_{l,l+1}} = \frac{A}{B_{a+4,a+5}}  \frac{A}{B_{a+6,a+7}} \cdots  \frac{A}{B_{b-1,b}},
   \end{equation*}
   \\ for $k = a+2$ and $k = a+3$. \\ \\
   The analogous conclusion can be made for any pair $k$ and $k+1$, where $a+4 \leq k \leq b-3 $ is an even number, i.e. \\ \\$\prod_{\substack{l=k+1 \\ l even}}^{b} \frac{A}{B_{l,l+1}}$ is the same for $k = a+4$   and $k = a+5$, $k = a+6$  and $k = a+7$, ..., $k = b-3$ and $k = b-2$. \\ \\ Since $\prod_{\substack{l=k+1 \\ l even}}^{b} \frac{A}{B_{l,l+1}}$ is an empty product for $k = b-1$ and $k = b$, and equals $1$, the equality \eqref{eq111} follows. The proof of the equality \eqref{eq112} is analogous.
\end{proof}

\begin{lemma} Let $B_{l,l+1} \in \mathbb{R}$,  where $l \in \{1,2,...,n-1\}$. If $a, b \in \mathbb{Z}$, where $a$ is even, $b$ is odd, and $a \leq b$, then
\begin{equation*}
    \prod_{\substack{k=a \\ k \text  { even}}}^{b}\prod_{\substack{l=k+2 \\ l \text{ even}}}^{b+2}\frac{1}{B_{l,l+1}} = \prod_{\substack{k=a+1 \\ k \text{ odd}}}^{b}\prod_{\substack{l=k+1 \\ l \text{ even}}}^{b+2}\frac{1}{B_{l,l+1}}.
\end{equation*}
\label{lem4}
\end{lemma}
\begin{proof}
    It is an easy check that the left-hand side equals the right-hand side.
\end{proof}
\begin{lemma} Let $B_{l,l+1} \in \mathbb{R}$, where $l \in \{1,2,...,n-1\}$. If $a,b \in \mathbb{Z}$, where $a$ and $b$ are odd, and $a \leq b$, then 
\begin{equation*}
    \prod_{\substack{k=a \\ k \text{ odd}}}^{b}\prod_{\substack{l=k+3 \\ l \text{ odd}}}^{b+2}\frac{1}{B_{l,l+1}} =  \prod_{\substack{k=a \\ k \text{ even}}}^{b}\prod_{\substack{l=k+3 \\ l \text{ odd}}}^{b+2}\frac{1}{B_{l,l+1}}.
\end{equation*}
\label{lemtpr}
\end{lemma}
\begin{proof}
Note that 
\begin{equation*}
\begin{split}
     \prod_{\substack{k=a \\ k \text{ odd}}}^{b}\prod_{\substack{l=k+3 \\ l \text{ odd}}}^{b+2}\frac{1}{B_{l,l+1}} &= \frac{1}{B_{a+4,a+5}B_{a+6,a+7}\cdots B_{b+2,b+3}}  \frac{1}{B_{a+6,a+7}B_{a+8,a+9}\cdots B_{b+2,b+3}}\\ &
      \qquad \qquad \frac{1}{B_{a+8,a+9}B_{a+10,a+11}\cdots B_{b+2,b+3}} \cdots \frac{1}{B_{b+2,b+3}}  \\ &
     =  \prod_{\substack{k=a \\ k \text{ even}}}^{b}\prod_{\substack{l=k+3 \\ l \text{ odd}}}^{b+2}\frac{1}{B_{l,l+1}}.
     \end{split}
     \end{equation*}
\end{proof}
The following is easy to verify. \\
\begin{lemma} Let $B_{l,l+1} \in \mathbb{R}$, where $l \in \{1,2,...,n-1\}$. If $a, b \in \mathbb{Z}$, where $a$ and $b$ are even, and $a \leq b$, then
\begin{equation}
    \prod_{\substack{k=a \\ k \text{ even}}}^{b}\prod_{\substack{l=k+3 \\ l \text{ odd}}}^{b+3}B_{l,l+1} = \prod_{\substack{l=a+3 \\ l \text{ odd}}}^{b+3}B_{l,l+1}^{\frac{l-(a+1)}{2}}
  \label{eq1}  
,\end{equation}
and 
\begin{equation}
   \prod_{\substack{k=a \\ k \text{ even}}}^{b}\prod_{\substack{l=k+2 \\ l \text{ even}}}^{b+2}B_{l,l+1} = \prod_{\substack{l=a+2 \\ l \text{ even}}}^{b+2}B_{l,l+1}^{\frac{l-a}{2}}
   \label{eq2}.
\end{equation}
\label{lemspr}
\end{lemma}
  \begin{lemma}
  Let $A\in \mathbb{R}$. If $a, b \in \mathbb{Z}$, where $a$ and $b$ are odd, and $a\leq b$, then
  \begin{equation}
      \prod_{k=a}^{b}\prod_{\substack{l=k+2 \\ l \text{ odd}}}^{b+2}A = A^{(\frac{b-a+2}{2})^2},
      \label{eqa}
  \end{equation}
  and
  \begin{equation}
    \prod_{k=a}^{b}\prod_{\substack{l=k+3 \\ l \text{ odd}}}^{b+2}A = A^{\frac{(b-a)(b-a+2)}{4}}. 
    \label{eqb}
  \end{equation}
  \label{lemgrp}
  \end{lemma}
  \begin{proof}
  For equation \eqref{eqb}, note that $\prod_{\substack{l=k+3 \\ l \text{ odd}}}^{b+2}A = A^{\frac{b-a}{2}}$, for $k = a$ and $k = a+1$. Note as well that
  \begin{equation*}
     \prod_{\substack{l=k+3 \\ l \text{ odd}}}^{b+2}A =
     \begin{cases}
         \begin{alignedat}{2}
    & A^{\frac{b-a-2}{2}}, &\text{for $k = a+2$ and $k = a+3$;}
      \\
      & A^\frac{b-a-4}{2} , &\text{for $k = a+4$ and $k = a+5$;}
      \\
      &\text{ \quad \quad \quad \quad \quad... .... .... } 
      \\
     & A^2 , &\text{for $k = b-4$ and $k = b-3$;} \\
     &A , &\text{for $k = b-2$ and $k = b-1$.}
    \end{alignedat}
  \end{cases}
\end{equation*}
Then  
\begin{equation*}
\prod_{k=a}^{b}\prod_{\substack{l=k+3 \\ l \text{ odd}}}^{b+2}A = A^2 A^4 \cdots A^{b-a-4} A^{b-a-2} A^{b-a} = A^{\frac{(b-a)(b-a+2)}{4}}.
\end{equation*}
For equation \eqref{eqa}, we have that 
\begin{equation*}
  \prod_{k=a}^{b}\prod_{\substack{l=k+2 \\ l \text{ odd}}}^{b+2}A = A^{\frac{b-a+2}{2}} \prod_{k=a}^b\prod_{\substack{l=k+3 \\ l \text{ odd}}}^{b+2}A 
  = A^{\frac{b-a+2}{2}} A^{\frac{(b-a)(b-a+2)}{4}} 
= A^{(\frac{b-a+2}{2})^2}.
\end{equation*}
  \end{proof}
  \begin{lemma}
  Let $A \in \mathbb{R}$ and $a, b \in \mathbb{Z}$, where $a \leq b$. Then
  \begin{equation}
  \prod_{\substack{k=a \\ k\text{ odd}}}^{b}\prod_{\substack{l=k+1 \\ l \text{ even}}}^{b+2}A = A^ {\frac{(b-a+2)(b-a+4)}{8}},
  \label{eq10}
  \end{equation}
  if both $a$ and $b$ are odd, \\  \\ and 
  \begin{equation}
      \prod_{\substack{k=a \\ k \text{ even}}}^{b}\prod_{\substack{l=k+2 \\ l \text{ even}}}^{b+2}A = A^{\frac{(b-a+1)(b-a+3)}{8}},
      \label{eq11}
  \end{equation}
  if $a$ is even, and $b$ is odd.
  \label{lem8}
  \end{lemma}
  \begin{proof}
    First, note that, for $a$ odd, $b$ odd, we have that:
    \begin{equation*}
    \begin{split}
      \prod_{\substack{k=a \\ k \text{ odd}}}^{b}\prod_{\substack{l=k+1 \\ l \text{ even}}}^{b+2}A &=     \prod_{\substack{k=a \\ k \text{ odd}}}^{b}\prod_{\substack{l=k+1 \\ l \text{ even}}}^{b+1}A = \prod_{\substack{l=a+1 \\ l \text{ even}}}^{b+1}A\prod_{\substack{l=a+3 \\ l \text{ even}}}^{b+1}A  \cdots\prod_{\substack{l=b-1 \\ l \text{ even}}}^{b+1}A\prod_{\substack{l=b+1 \\ l \text{ even}}}^{b+1}A \\ &= A^{\frac{b-a+2}{2}} A^{\frac{b-a}{2}}  \cdots A^2 A = A^{\frac{(b-a+2)(b-a+4)}{8}},
      \end{split}
    \end{equation*}
    so equality \eqref{eq10} holds. \\  \\
    Let now $a$ be even, and $b$ odd. Then 
    \begin{equation*}
    \begin{split}
        \prod_{\substack{k=a \\ k \text { even}}}^{b}\prod_{\substack{l=k+2 \\ l \text{ even}}}^{b+2}A &= \prod_{\substack{k=a \\ k \text{ even}}}^{b-1}\prod_{\substack{l=k+2 \\ l \text{ even}}}^{b+1}A = \prod_{\substack{l=a+2 \\ l \text{ even}}}^{b+1}A \prod_{\substack{l=a+4 \\ l \text{ even}}}^{b+1}A \cdots \prod_{\substack{l=b-1 \\ l \text{ even}}}^{b+1}A \prod_{\substack{l=b+1 \\ l \text{ even}}}^{b+1}A\\
        &= A^{\frac{b-a+1}{2}} A^{\frac{b-a-1}{2}} \cdots A^2 A = A^{\frac{(b-a+1)(b-a+3)}{8}},
        \end{split}
    \end{equation*}
     so we have that equality \eqref{eq11} holds as well.
  \end{proof}

Now, note that, since $n$ is even and $P$ is cyclic in the statement of Theorem \eqref{impthm}, from Theorem \eqref{thm009} it follows that \begin{equation}
\sum_{m=1}^{n}(-1)^{m+1} \delta(m) = 0,
\label{eq31}
\end{equation}
where \begin{equation}\delta(m) = \prod_{\substack {i=1 \\ i \neq m}}^{n-1}(\prod_{\substack {j=i+1 \\ j \neq m }}^{n} \delta_{ij}),
\label{eq34}
\end{equation}
recalling that $\delta_{pq}= |A_pA_q|$, for $p, q \in \{1,2,..., n\}.$
 \begin{prop} For $\delta(m)$ as in \eqref{eq34}, the following holds:
 \label{prop9}
 \begin{multline}
\delta(m) = R^{\frac{n-2}{2}} \left(\prod_{k=1}^{m-1}\left(\prod_{\substack {l=k+1 \\ l \text{ even}}}^{m-1}S_{k,l,l+1} \prod_{\substack{l=m+1 \\ l \text{ odd }}}^{n-1}S_{k,l,l+1}\right)\right)\left(\prod_{k=m}^{n-3}\prod_{\substack {l=k+2 \\ l \text{ odd }}}^{n-1}S_{k+1,l,l+1}\right)\\  \left(\prod_{\substack {k=2 \\ k \text{ even}}}^{m-1}\prod_{\substack {l=k+1 \\ l \text{ even}}}^{m-1} \frac{R^2}{x_{l,l+1}}\right)\left(\prod_{\substack{k=m \\ k\text{ even}}}^{n-3}\prod_{\substack {l=k+2 \\ l odd }}^{n-1}\frac{R^2}{x_{l,l+1}}\right)\left(\prod_{\substack {l=m+1 \\ l\text{ odd}}}^{n-1} \frac{R^{m-2}}{x_{l,l+1}^{\frac{m-2}{2}}}\right),
\label{deltam}
 \end{multline}
 for even $m$, $2 \leq m \leq n$;
 \begin{equation}
   \delta(1) = \left(\prod_{k=1}^{n-3}\prod_{\substack {l=k+2 \\ l \text{ odd}}}^{n-1}S_{k+1,l,l+1}R\right) \left(\prod_{\substack{k=2 \\ k \text{ even}}}^{n-3}\prod_{\substack {l=k+2 \\ l \text{ odd}}}^{n-1}\frac{1}{x_{l,l+1}}\right) ;
   \label{dm}
 \end{equation}
 \begin{equation}
     \delta(3) = S_{124}R \left(\prod_{\substack{l=5 \\ l \text{ odd }}}^{n-1}S_{1,l,l+1}S_{2,l,l+1}R^2\right)\left(\prod_{k=3}^{n-3}\prod_{\substack{l=k+2 \\ l\text{ odd}}}^{n-1}S_{k+1,l,l+1}R\right)\left(\prod_{\substack{k=3 \\k \text{ odd}}}^{n-3}\prod_{\substack{l=k+2 \\ l \text{ odd}}}^{n-1}\frac{1}{x_{l,l+1}}\right);
     \label{deltamm}
 \end{equation}
 and
 \begin{equation}
 \begin{split}
    \delta(m) &= \frac{S_{1,2,m-1}S_{1,2,m+1}R^2}{x_{12}}\left(\prod_{\substack{l=2 \\ l \text{ even}}}^{m-2}S_{1,l,l+1}R\right)\left(\prod_{\substack{l=m+2 \\ l\text{ odd}}}^{n-1}S_{1,l,l+1}R\right) 
        \left(\prod_{\substack{l=m+2 \\ l \text{ odd}}}^{n-1} \frac{S_{m-1,l,l+1}S_{m-3,l,l+1}R^2}{x_{l,l+1}}\right)  \\ & \qquad \qquad  \left(\prod_{\substack{k=3 \\ k \text{ odd}}}^{m-4}\left(\frac{S_{k,k+1,m-1}S_{k,k+1,m+1}R^2}{x_{k,k+1}}\prod_{\substack{l=k+1 \\ l \text{ even}}}^{m-2}\frac{S_{k,l,l+1}R}{x_{l,l+1}}\prod_{\substack{l=m+2 \\ l \text{ odd}}}^{n-1}\frac{1}{x_{l,l+1}}\right)\right)\left(\prod_{\substack{k=2 \\ k \text{ even}}}^{m-5}\prod_{\substack{l=k+2 \\ l \text{ even}}}^{m-2} S_{k,l,l+1}R\right)  \\& \qquad \qquad \left(\prod_{k=2}^{m-4}\prod_{\substack{l=m+2 \\ l \text{ odd}}}^{n-1}S_{k,l,l+1}R\right) S_{m-2,m-1,m+1}R \left(\prod_{\substack{l=m+2 \\ l \text{ odd}}}^{n-1}\frac{S_{m-2,l,l+1}R}{x_{l,l+1}}\right)\left(\prod_{\substack{k=m \\ k \text{ odd}}}^{n-3} S_{k+1,k+2,k+3}R\right)  \\ & \qquad \qquad \left(\prod_{k=m}^{n-3}\prod_{\substack{l=k+3 \\ l \text{ odd}}}^{n-1}S_{k+1,l,l+1}R\right) \left(\prod_{\substack{k=m \\ k \text{ odd}}}^{n-3}\prod_{\substack{l=k+3 \\ l \text{ odd}}}^{n-1} \frac{1}{x_{l,l+1}}\right),
    \label{deltammm}
\end{split}   
\end{equation}
for odd $m$, $3 <m \leq n-1$, where $R$ is the radius of the circumscribed circle of $P$.
\label{prop38}
 \end{prop}
 \begin{proof}
For even $m$, $2 \leq m \leq n$, equality \eqref{eq34} can be rewritten as follows: 
\begin{equation*}
\begin{split}
    \delta(m) &= \left(\prod_{\substack {l=2 \\ l \text{ even}}}^{m-2}\delta_{1l}\delta_{1,l+1}\delta_{l,l+1}\right)\left(\prod_{\substack {l=m+1 \\ l \text{ odd}}}^{n-1}\delta_{1l}\delta_{1,l+1}\delta_{l,l+1}\right)  \\ & \qquad \qquad \left(\prod_{k=2}^{m-1}\left(\prod_{\substack {l=k+1 \\ l \text{ even}}}^{m-1}\delta_{kl}\delta_{k,l+1}\prod_{\substack{l=m+1 \\ l \text{ odd }}}^{n-1}\delta_{kl}\delta_{k,l+1}\right)\right)
    \left(\prod_{k=m}^{n-3}\prod_{\substack {l=k+2 \\ l \text{ odd }}}^{n-1}\delta_{k+1,l}\delta_{k+1,l+1}\right)  \\ &=
    \left(\prod_{\substack {l=2 \\ l \text{ even}}}^{m-2}S_{1,l,l+1}R\right)\left(\prod_{\substack {l=m+1 \\ l \text{ odd}}}^{n-1}S_{1,l,l+1}R\right)  \\ & \qquad \qquad \left(\prod_{k=2}^{m-1}\left(\prod_{\substack {l=k+1 \\ l \text{ even}}}^{m-1}\frac{S_{k,l,l+1}R}{\delta_{l,l+1}} \prod_{\substack{l=m+1 \\ l \text{ odd }}}^{n-1}\frac{S_{k,l,l+1}R}{\delta_{l,l+1}}\right)\right) \left(\prod_{k=m}^{n-3}\prod_{\substack {l=k+2 \\ l \text{ odd }}}^{n-1}\frac{S_{k+1,l,l+1}R}{\delta_{l,l+1}}\right)  \\ & 
     = R^{\frac{n-2}{2}} \prod_{k=1}^{m-1}\left(\prod_{\substack {l=k+1 \\ l \text{ even}}}^{m-1}S_{k,l,l+1} \prod_{\substack{l=m+1 \\ l \text{ odd }}}^{n-1}S_{k,l,l+1}\right)\left(\prod_{k=m}^{n-3}\prod_{\substack {l=k+2 \\ l \text{ odd }}}^{n-1}S_{k+1,l,l+1}\right) \\ & \qquad \qquad \left(\prod_{k=2}^{m-1}\left(\prod_{\substack {l=k+1 \\ l \text{ even}}}^{m-1}\frac{R}{\delta_{l,l+1}}\prod_{\substack{l=m+1 \\ l\text{ odd }}}^{n-1}\frac{R}{\delta_{l,l+1}}\right)\right)\left(\prod_{k=m}^{n-3}\prod_{\substack {l=k+2 \\ l\text{ odd}}}^{n-1}\frac{R}{\delta_{l,l+1}}\right).
\end{split}
\end{equation*}
Let now $Q: = \left(\prod_{k=2}^{m-1}\left(\prod_{\substack {l=k+1 \\ l \text{ even}}}^{m-1}\frac{R}{\delta_{l,l+1}} \prod_{\substack{l=m+1 \\ l \text{ odd} }}^{n-1}\frac{R}{\delta_{l,l+1}}\right)\right)\left(\prod_{k=m}^{n-3}\prod_{\substack {l=k+2 \\ l \text{ odd }}}^{n-1}\frac{R}{\delta_{l,l+1}}\right).$
\\ Then 
\begin{equation*}
Q = \left(\prod_{k=2}^{m-1}\prod_{\substack {l=k+1 \\ l \text{ even}}}^{m-1} \frac{R}{\delta_{l,l+1}}\right) \left(\prod_{\substack {l=m+1 \\ l \text{ odd}}}^{n-1} \frac{R^{m-2}}{\delta_{l,l+1}^{m-2}}\right)\left(\prod_{k=m}^{n-3}\prod_{\substack {l=k+2 \\ l \text{ odd }}}^{n-1}\frac{R}{\delta_{l,l+1}}\right).
\end{equation*}
By Lemma \ref{lem32} \eqref{eq111}, we have that
\begin{equation*}
   \prod_{k=2}^{m-1}\prod_{\substack {l=k+1 \\ l \text{ even}}}^{m-1} \frac{R}{\delta_{l,l+1}} = \prod_{\substack {k=2 \\ k \text{ even}}}^{m-1}\prod_{\substack {l=k+1 \\ l \text{ even}}}^{m-1} \frac{R^2}{x_{l,l+1}}.
\end{equation*}
From Lemma \ref{lem32} \eqref{eq112}, it follows that 
\begin{equation*}
    \prod_{k=m}^{n-3}\prod_{\substack {l=k+2 \\ l \text{ odd} }}^{n-1}\frac{R}{\delta_{l,l+1}} = \prod_{\substack{k=m \\ k \text{ even}}}^{n-3}\prod_{\substack {l=k+2 \\ l \text{ odd }}}^{n-1}\frac{R^2}{x_{l,l+1}}.
\end{equation*}
Finally, we have that 
\begin{equation*}
    \prod_{\substack {l=m+1 \\ l \text{ odd}}}^{n-1} \frac{R^{m-2}}{\delta_{l,l+1}^{m-2}} = \prod_{\substack {l=m+1 \\ l \text{ odd}}}^{n-1} \frac{R^{m-2}}{x_{l,l+1}^{\frac{m-2}{2}}},
\end{equation*} so it follows that 
\begin{equation*}
Q = \left(\prod_{\substack {k=2 \\ k \text{ even}}}^{m-1}\prod_{\substack {l=k+1 \\ l \text{ even}}}^{m-1} \frac{R^2}{x_{l,l+1}}\right)\left(\prod_{\substack{k=m \\ k \text { even}}}^{n-3}\prod_{\substack {l=k+2 \\ l \text{ odd }}}^{n-1}\frac{R^2}{x_{l,l+1}}\right)\left(\prod_{\substack {l=m+1 \\ l \text{ odd}}}^{n-1} \frac{R^{m-2}}{x_{l,l+1}^{\frac{m-2}{2}}}\right).
\end{equation*}
Since 
\begin{equation*}
 \delta(m) = R^{\frac{n-2}{2}}\left( \prod_{k=1}^{m-1}\left(\prod_{\substack {l=k+1 \\ l \text{ even}}}^{m-1}S_{k,l,l+1} \prod_{\substack{l=m+1 \\ l \text{ odd }}}^{n-1}S_{k,l,l+1}\right)\right)\left(\prod_{k=m}^{n-3}\prod_{\substack {l=k+2 \\ l \text{ odd }}}^{n-1}S_{k+1,l,l+1}\right) \cdot Q,
\end{equation*}
we have that 
\begin{multline}
\delta(m) = R^{\frac{n-2}{2}} \left(\prod_{k=1}^{m-1}\left(\prod_{\substack {l=k+1 \\ l \text{ even}}}^{m-1}S_{k,l,l+1} \prod_{\substack{l=m+1 \\ l \text{ odd }}}^{n-1}S_{k,l,l+1}\right)\right)\left(\prod_{k=m}^{n-3}\prod_{\substack {l=k+2 \\ l \text{ odd }}}^{n-1}S_{k+1,l,l+1}\right) \\ \left(\prod_{\substack {k=2 \\ k \text{ even}}}^{m-1}\prod_{\substack {l=k+1 \\ l \text{ even}}}^{m-1} \frac{R^2}{x_{l,l+1}}\right)\left(\prod_{\substack{k=m \\ k \text { even}}}^{n-3}\prod_{\substack {l=k+2 \\ l \text{ odd }}}^{n-1}\frac{R^2}{x_{l,l+1}}\right)\left(\prod_{\substack {l=m+1 \\ l \text{ odd}}}^{n-1} \frac{R^{m-2}}{x_{l,l+1}^{\frac{m-2}{2}}}\right),
\end{multline}   
for $m$ even, $2 \leq m \leq n$. \\ \\
Let us now see how equality \eqref{eq34} looks like when $m$ is odd. We will look separately into the cases $m = 1$, $m = 3$ and $m > 3$. \\ \\ Firstly, let $m = 1$. Then 
\begin{equation*}
\begin{split}
\delta(m) = \delta(1)& = \delta_{23}\delta_{24}\delta_{34}\left(\prod_{\substack{l=5 \\ l \text{ odd}}}^{n-1}\delta_{2l}\delta_{2,l+1}\delta_{l,l+1}\right) \\ & \qquad \qquad  \left(\prod_{k=2}^{n-3}\left(\left(\delta_{k+1,k+2}\delta_{k+1,k+3}\right)^{(k \text{ mod 2})}\prod_{\substack {l=k+3 \\ l \text{ odd}}}^{n-1}\delta_{k+1,l}\delta_{k+1,l+1}\right)\right) \\ &= \delta_{23}\delta_{24}\delta_{34}\left(\prod_{\substack{l=5 \\ l \text{ odd}}}^{n-1}\delta_{2l}\delta_{2,l+1}\delta_{l,l+1}\right) \left(\prod_{k=2}^{n-3}\prod_{\substack {l=k+2 \\ l \text{ odd}}}^{n-1}\delta_{k+1,l}\delta_{k+1,l+1}\right) \\ &= 
S_{234}R\left(\prod_{\substack{l=5 \\ l \text{ odd}}}^{n-1}S_{2,l,l+1}R\right)\left(\prod_{k=2}^{n-3}\prod_{\substack {l=k+2 \\ l \text{ odd}}}^{n-1}\frac{S_{k+1,l,l+1}R}{\delta_{l,l+1}}\right).
\end{split}
\end{equation*}
Using Lemma \ref{lem32} \eqref{eq112}, it follows that 
\begin{equation*}
    \prod_{k=2}^{n-3}\prod_{\substack {l=k+2 \\ l \text{ odd}}}^{n-1}\frac{S_{k+1,l,l+1}R}{\delta_{l,l+1}} = \left(\prod_{k=2}^{n-3}\prod_{\substack {l=k+2 \\ l \text{ odd}}}^{n-1}S_{k+1,l,l+1}R\right) \left(\prod_{\substack{k=2 \\ k \text{ even}}}^{n-3}\prod_{\substack {l=k+2 \\ l \text{ odd}}}^{n-1}\frac{1}{x_{l,l+1}}\right),
\end{equation*}
so
\begin{equation*}
\begin{split}
    \delta(1) &= S_{234}R\left(\prod_{\substack{l=5 \\ l\text{ odd}}}^{n-1}S_{2,l,l+1}R\right)\left(\prod_{k=2}^{n-3}\prod_{\substack {l=k+2 \\ l \text{ odd}}}^{n-1}S_{k+1,l,l+1}R\right) \left(\prod_{\substack{k=2 \\ k \text{ even}}}^{n-3}\prod_{\substack {l=k+2 \\ l \text{ odd}}}^{n-1}\frac{1}{x_{l,l+1}}\right) \\ 
    &= \left(\prod_{\substack{l=3 \\ l \text{ odd}}}^{n-1}S_{2,l,l+1}R\right)\left(\prod_{k=2}^{n-3}\prod_{\substack {l=k+2 \\ l \text{ odd}}}^{n-1}S_{k+1,l,l+1}R\right) \left(\prod_{\substack{k=2 \\ k \text{ even}}}^{n-3}\prod_{\substack {l=k+2 \\ l \text{ odd}}}^{n-1}\frac{1}{x_{l,l+1}}\right) \\ 
    & = \left(\prod_{k=1}^{n-3}\prod_{\substack {l=k+2 \\ l \text{ odd}}}^{n-1}S_{k+1,l,l+1}R\right) \left(\prod_{\substack{k=2 \\ k \text{ even}}}^{n-3}\prod_{\substack {l=k+2 \\ l \text{ odd}}}^{n-1}\frac{1}{x_{l,l+1}}\right).
    \end{split}
\end{equation*}
Now let $m = 3$. We have: 
\begin{equation}
 \begin{split}   
 \delta(m) = \delta(3) &= \delta_{12}\delta_{14}\delta_{24}\left(\prod_{\substack{l=5 \\ l \text{ odd}}}^{n-1}\delta_{1l}\delta_{1,l+1}\delta_{l,l+1}\delta_{2,l}\delta_{2,l+1}\right)  \left(\prod_{k=3}^{n-3}\prod_{\substack{l=k+2 \\ l \text{ odd}}}^{n-1}\delta_{k+1,l}\delta_{k+1,l+1}\right) \\ &= S_{124}R\left(\prod_{\substack{l=5 \\ l \text{ odd}}}^{n-1}\left(S_{1,l,l+1}R\frac{S_{2,l,l+1}R}{\delta_{l,l+1}}\right)\right)\left(\prod_{k=3}^{n-3}\prod_{\substack{l=k+2 \\ l \text{ odd}}}^{n-1}\frac{S_{k+1,l,l+1}R}{\delta_{l,l+1}}\right)\\ 
 &= S_{124}R\left(\prod_{\substack{l=5 \\ l \text{ odd}}}^{n-1}S_{1,l,l+1}S_{2,l,l+1}R^2\right)\left(\prod_{\substack{l=5 \\ l \text{ odd}}}^{n-1}\frac{1}{x_{l,l+1}}\right)\left(\prod_{k=3}^{n-3}\prod_{\substack{l=k+2 \\ l \text{ odd}}}^{n-1}S_{k+1,l,l+1}R\right) \left(\prod_{k=4}^{n-3}\prod_{\substack{l=k+2 \\ l \text{ odd}}}^{n-1}\frac{1}{\delta_{l,l+1}}\right).
 \end{split}
\end{equation}
Using Lemma \ref{lem32} \eqref{eq112}, we have that
\begin{equation*}
    \prod_{k=4}^{n-3}\prod_{\substack{l=k+2 \\ l \text{ odd}}}^{n-1}\frac{1}{\delta_{l,l+1}} = \prod_{\substack{{k=4} \\ k \text{ even}}}^{n-3}\prod_{\substack{l=k+2 \\ l \text{ odd}}}^{n-1}\frac{1}{x_{l,l+1}}, 
\end{equation*}
and also using the fact that
\begin{equation*}
\prod_{\substack{{k=4} \\ k \text{ even}}}^{n-3}\prod_{\substack{l=k+2 \\ l \text{ odd}}}^{n-1}\frac{1}{x_{l,l+1}} = \prod_{\substack{{k=5} \\ k \text{ odd}}}^{n-3}\prod_{\substack{l=k+2 \\ l \text{ odd}}}^{n-1}\frac{1}{x_{l,l+1}},
\end{equation*}
we get
\begin{equation*}
    \delta(3) = S_{124}R\left(\prod_{\substack{l=5 \\ l \text{ odd }}}^{n-1}S_{1,l,l+1}S_{2,l,l+1}R^2\right)\left(\prod_{k=3}^{n-3}\prod_{\substack{l=k+2 \\ l \text{ odd}}}^{n-1}S_{k+1,l,l+1}R\right)\left(\prod_{\substack{k=3 \\k \text{ odd}}}^{n-3}\prod_{\substack{l=k+2 \\ l \text{ odd}}}^{n-1}\frac{1}{x_{l,l+1}}\right).
\end{equation*}
\\
Let us now consider $\delta(m)$, for odd $m$, $3 < m \leq n-1$. For such an $m$, we have that 
\begin{equation*}
\begin{split}
    \delta(m) &= \left(\prod_{\substack{l=2 \\ l \text{ even}}}^{m-2}\delta_{1,l}\delta_{1,l+1}\delta_{l,l+1}\right)\left(\delta_{1,m-1}\delta_{2,m-1}\right)(\delta_{1,m+1}\delta_{2,m+1}) \\ & \qquad \qquad \left(\prod_{\substack{l=m+2 \\ l \text{ odd}}}^{n-1}\delta_{1,l}\delta_{1,l+1}\delta_{l,l+1}\right)\left(\prod_{\substack{k=2 \\ k \text{ even}}}^{m-1}\left(\prod_{\substack{l=k+2 \\ l \text{even}}}^{m-2}\delta_{k,l}\delta_{k,l+1}\prod_{\substack{l=m+2 \\ l \text{ odd}}}^{n-1}\delta_{k,l}\delta_{k,l+1}\right)\right) \\ & \qquad \qquad 
    \left(\prod_{\substack{k=3 \\ k \text{ odd}}}^{m-4}\left(\delta_{k,m-1}\delta_{k+1,m-1}\delta_{k,m+1}\delta_{k+1,m+1}\prod_{\substack{l=k+1 \\ l \text{ even}}}^{m-2}\delta_{k,l}\delta_{k,l+1}\prod_{\substack{l=m+2 \\ l \text{ odd}}}^{n-1}\delta_{k,l}\delta_{k,l+1}\right)\right) \\ & \qquad \qquad  \left(\delta_{m-2,m-1}\delta_{m-2,m+1}\delta_{m-1,m+1}\right)\left(\prod_{\substack{l=m+2 \\ l \text{ odd}}}^{n-1}\delta_{m-2,l}\delta_{m-2,l+1}\right)\\ & \qquad \qquad \left(\prod_{\substack{k=m \\ k \text{ odd}}}^{n-3}\prod_{\substack{l=k+2 \\ l \text{ odd}}}^{n-1}\delta_{k+1,l}\delta_{k+1,l+1}\right) \left(\prod_{\substack{k=m \\ k \text{ even}}}^{n-3}\prod_{\substack{l=k+3 \\ l \text{ odd}}}^{n-1}\delta_{k+1,l}\delta_{k+1,l+1}\right).
    \end{split}
\end{equation*}
Let now 
\begin{equation*}
\begin{split}
    Q_1 :&=\left(\prod_{\substack{k=2 \\ k \text{ even}}}^{m-1}\left(\prod_{\substack{l=k+2 \\ l \text{ even}}}^{m-2}\delta_{k,l}\delta_{k,l+1}\prod_{\substack{l=m+2 \\ l \text{ odd}}}^{n-1}\delta_{k,l}\delta_{k,l+1}\right)\right) \\ & \qquad \qquad
    \left(\prod_{\substack{k=3 \\ k \text{ odd}}}^{m-4}\left(\delta_{k,m-1}\delta_{k+1,m-1}\delta_{k,m+1}\delta_{k+1,m+1}\prod_{\substack{l=k+1 \\ l \text{ even}}}^{m-2}\delta_{k,l}\delta_{k,l+1}\prod_{\substack{l=m+2 \\ l \text{ odd}}}^{n-1}\delta_{k,l}\delta_{k,l+1}\right)\right),
    \end{split}
   \end{equation*}
   and 
   \begin{equation}
   Q_2:= \left(\prod_{\substack{l=m+2 \\ l\text{ odd}}}^{n-1}\delta_{m-2,l}\delta_{m-2,l+1}\right)\left(\prod_{\substack{k=m \\ k \text{ odd}}}^{n-3}\prod_{\substack{l=k+2 \\ l \text{ odd}}}^{n-1}\delta_{k+1,l}\delta_{k+1,l+1}\right)\left(\prod_{\substack{k=m \\ k \text{ even}}}^{n-3}\prod_{\substack{l=k+3 \\ l \text{ odd}}}^{n-1}\delta_{k+1,l}\delta_{k+1,l+1}\right).
   \label{eq3000}
   \end{equation}
Then 
\begin{equation}
\begin{split}
    \delta(m) &= \frac{S_{1,2,m-1}S_{1,2,m+1}R^2}{x_{12}}\left(\prod_{\substack{l=2 \\ l \text{ even}}}^{m-2}S_{1,l,l+1}R\right)\left(\prod_{\substack{l=m+2 \\ l \text{ odd}}}^{n-1}S_{1,l,l+1}R\right) \\ & \qquad \qquad
     Q_1  \left(S_{m-2,m-1,m+1}R\right)  Q_2.
    \label{eq42}
    \end{split}
\end{equation}
We will now write $Q_1$ and $Q_2$ in terms of the entries of $F$.  
\begin{equation*}
\begin{split}
    Q_1 &= \left(\prod_{\substack{k=2 \\ k \text{ even}}}^{m-1}\left(\prod_{\substack{l=k+2 \\ l \text{ even}}}^{m-2} \frac{S_{k,l,l+1}R}{\delta_{l,l+1}}\prod_{\substack{l=m+2 \\ l \text{ odd}}}^{n-1}\frac{S_{k,l,l+1}R}{\delta_{l,l+1}}\right)\right)\\ & \qquad \qquad \left(\prod_{\substack{k=3 \\ k \text{ odd}}}^{m-4}\left(\frac{S_{k,k+1,m-1}S_{k,k+1,m+1}R^2}{x_{k,k+1}}\prod_{\substack{l=k+1 \\ l \text{ even}}}^{m-2}\frac{S_{k,l,l+1}R}{\delta_{l,l+1}}\prod_{\substack{l=m+2 \\ l \text{ odd}}}^{n-1}\frac{S_{k,l,l+1}R}{\delta_{l,l+1}}\right)\right) \\  &= \left(\prod_{\substack{l=m+2 \\ l \text{ odd}}}^{n-1}\frac{S_{m-1,l,l+1}R}{\delta_{l,l+1}}\right)\left(\prod_{\substack{l=m+2 \\ l \text{ odd}}}^{n-1}\frac{S_{m-3,l,l+1}R}{\delta_{l,l+1}}\right)\left(\prod_{\substack{k=2 \\ k \text{ even}}}^{m-5}\left(\prod_{\substack{l=k+2 \\ l \text{ even}}}^{m-2} \frac{S_{k,l,l+1}R}{\delta_{l,l+1}}\prod_{\substack{l=m+2 \\ l \text{ odd}}}^{n-1}\frac{S_{k,l,l+1}R}{\delta_{l,l+1}}\right)\right) \\ & \qquad \qquad \left(\prod_{\substack{k=3 \\ k \text{ odd}}}^{m-4}\left(\frac{S_{k,k+1,m-1}S_{k,k+1,m+1}R^2}{x_{k,k+1}}\prod_{\substack{l=k+1 \\ l \text{ even}}}^{m-2}\frac{S_{k,l,l+1}R}{\delta_{l,l+1}}\prod_{\substack{l=m+2 \\ l \text{ odd}}}^{n-1}\frac{S_{k,l,l+1}R}{\delta_{l,l+1}}\right)\right),
    \end{split}
\end{equation*}
where we use the fact that 
\begin{equation*}
\begin{split}
  \left( \prod_{\substack{k=2 \\ k \text{ even}}}^{m-1}\left(\prod_{\substack{l=k+2 \\ l \text{ even}}}^{m-2} \frac{S_{k,l,l+1}R}{\delta_{l,l+1}}\prod_{\substack{l=m+2 \\ l \text{ odd}}}^{n-1}\frac{S_{k,l,l+1}R}{\delta_{l,l+1}}\right)\right) &= \left(\prod_{\substack{l=m+2 \\ l \text{ odd}}}^{n-1}\frac{S_{m-1,l,l+1}R}{\delta_{l,l+1}}\right)\left(\prod_{\substack{l=m+2 \\ l \text{ odd}}}^{n-1}\frac{S_{m-3,l,l+1}R}{\delta_{l,l+1}}\right) \\ & \qquad \qquad \left(\prod_{\substack{k=2 \\ k \text{ even}}}^{m-5}\left(\prod_{\substack{l=k+2 \\ l \text{ even}}}^{m-2} \frac{S_{k,l,l+1}R}{\delta_{l,l+1}}\prod_{\substack{l=m+2 \\ l \text{ odd}}}^{n-1}\frac{S_{k,l,l+1}R}{\delta_{l,l+1}}\right) \right).
   \end{split}
\end{equation*}
Reorganizing terms we get that
\begin{equation*}
\begin{split}
Q_1 &= \left(\prod_{\substack{l=m+2 \\ l \text{ odd}}}^{n-1}\frac{S_{m-1,l,l+1}S_{m-3,l,l+1}R^2}{x_{l,l+1}}\right)\left(\prod_{\substack{k=2 \\ k \text{ even}}}^{m-5}\left(\prod_{\substack{l=k+2 \\ l \text{ even}}}^{m-2} \frac{1}{\delta_{l,l+1}}\prod_{\substack{l=m+2 \\ l \text{ odd}}}^{n-1}\frac{1}{\delta_{l,l+1}}\right)\right)\\ & \qquad \qquad \left(\prod_{\substack{k=3 \\ k \text{ odd}}}^{m-4}\left(\frac{S_{k,k+1,m-1}S_{k,k+1,m+1}R^2}{x_{k,k+1}}\prod_{\substack{l=k+1 \\ l \text{ even}}}^{m-2}\frac{S_{k,l,l+1}R}{\delta_{l,l+1}}\prod_{\substack{l=m+2 \\ l \text{ odd}}}^{n-1}\frac{1}{\delta_{l,l+1}}\right)\right) \\ 
& \qquad \qquad \left(\prod_{\substack{k=2 \\ k \text{ even}}}^{m-5}\prod_{\substack{l=k+2 \\ l \text{ even}}}^{m-2}S_{k,l,l+1}R\right) \left(\prod_{\substack{k=2 \\ k \text{ even}}}^{m-5}\prod_{\substack{l=m+2 \\ l \text{ odd}}}^{n-1}S_{k,l,l+1}R\right) \left(\prod_{\substack{k=3 \\ k \text{ odd}}}^{m-4}\prod_{\substack{l=m+2 \\ l \text{ odd}}}^{n-1}S_{k,l,l+1}R\right).
\end{split}
\end{equation*}
Let us now rewrite some terms of $Q_1$. Namely,
\begin{equation*}
    \prod_{\substack{k=2 \\ k \text{ even}}}^{m-5}\prod_{\substack{l=m+2 \\ l \text{ odd}}}^{n-1}\frac{1}{\delta_{l,l+1}} =  \prod_{\substack{k=3 \\ k \text{ odd}}}^{m-4}\prod_{\substack{l=m+2 \\ l \text{ odd}}}^{n-1}\frac{1}{\delta_{l,l+1}},
\end{equation*}
so 
\begin{equation*}
 \left(\prod_{\substack{k=2 \\ k \text{ even}}}^{m-5}\prod_{\substack{l=m+2 \\ l \text{ odd}}}^{n-1}\frac{1}{\delta_{l,l+1}}\right)  \left(\prod_{\substack{k=3 \\ k \text{ odd}}}^{m-4}\prod_{\substack{l=m+2 \\ l \text{ odd}}}^{n-1}\frac{1}{\delta_{l,l+1}}\right) = \prod_{\substack{k=3 \\ k \text{ odd}}}^{m-4}\prod_{\substack{l=m+2 \\ l \text{ odd}}}^{n-1}\frac{1}{x_{l,l+1}}.
\end{equation*}
We also have that 
\begin{equation}
\prod_{\substack{k=2 \\ k \text{ even}}}^{m-5}\prod_{\substack{l=k+2 \\ l \text{ even}}}^{m-2} \frac{1}{\delta_{l,l+1}} = \prod_{\substack{k=2 \\ k \text{ even}}}^{m-5}\prod_{\substack{l=k+2 \\ l \text{ even}}}^{m-3} \frac{1}{\delta_{l,l+1}} =\prod_{\substack{k=3 \\ k \text{ odd}}}^{m-4}\prod_{\substack{l=k+1 \\ l \text{ even}}}^{m-2} \frac{1}{\delta_{l,l+1}},
\label{eq1112}
\end{equation}
where the last equality is due to Lemma \ref{lemspr} \eqref{eq2}. \\ \\ Using \eqref{eq1112} it follows that 
\begin{equation*}
\left(\prod_{\substack{k=2 \\ k \text{ even}}}^{m-5}\prod_{\substack{l=k+2 \\ l \text{ even}}}^{m-2} \frac{1}{\delta_{l,l+1}}\right)\left(\prod_{\substack{k=3 \\ k \text{ odd}}}^{m-4}\prod_{\substack{l=k+1 \\ l \text{ even}}}^{m-2} \frac{S_{k,l,l+1}R}{\delta_{l,l+1}}\right) =  \prod_{\substack{k=3 \\ k \text{ odd}}}^{m-4}\prod_{\substack{l=k+1 \\ l \text{ even}}}^{m-2} \frac{S_{k,l,l+1}R}{x_{l,l+1}}.
\end{equation*}
Finally, using that 
\begin{equation*}
    \left(\prod_{\substack{k=2 \\ k \text{ even}}}^{m-5}\prod_{\substack{l=m+2 \\ l \text{ odd}}}^{n-1}S_{k,l,l+1}R\right) \left(\prod_{\substack{k=3 \\ k \text{ odd}}}^{m-4}\prod_{\substack{l=m+2 \\ l \text{ odd}}}^{n-1}S_{k,l,l+1}R\right) = \prod_{k=2}^{m-4}\left(\prod_{\substack{l=m+2 \\ l \text{ odd}}}^{n-1}S_{k,l,l+1}R\right),
\end{equation*}
we get
\begin{equation}
\begin{split}
Q_1 &= \left(\prod_{\substack{l=m+2 \\ l \text{ odd}}}^{n-1} \frac{S_{m-1,l,l+1}S_{m-3,l,l+1}R^2}{x_{l,l+1}}\right) \\ & \qquad \qquad \left(\prod_{\substack{k=3 \\ k \text{ odd}}}^{m-4}\left(\frac{S_{k,k+1,m-1}S_{k,k+1,m+1}R^2}{x_{k,k+1}}\prod_{\substack{l=k+1 \\ l \text{ even}}}^{m-2}\frac{S_{k,l,l+1}R}{x_{l,l+1}}\prod_{\substack{l=m+2 \\ l \text{ odd}}}^{n-1}\frac{1}{x_{l,l+1}}\right)\right)  \\ 
& \qquad \qquad \left(\prod_{\substack{k=2 \\ k \text{ even}}}^{m-5}\prod_{\substack{l=k+2 \\ l \text{ even}}}^{m-2} S_{k,l,l+1}R\right) \left(\prod_{k=2}^{m-4}\prod_{\substack{l=m+2 \\ l \text{ odd}}}^{n-1}S_{k,l,l+1}R\right).
\end{split}
\label{eqpone}
\end{equation}
Also, we can rewrite \eqref{eq3000} as 
\begin{equation}
 \begin{split}
 Q_2 &= \left(\prod_{\substack{l=m+2 \\ l \text{ odd}}}^{n-1} \frac {S_{m-2,l,l+1}R}{\delta_{l,l+1}}\right)\left(\prod_{\substack{k=m \\ k \text{ odd}}}^{n-3}\frac{S_{k+1,k+2,k+3}R}{\delta_{k+2,k+3}}\right) \\ & \qquad \qquad \left(\prod_{\substack{k=m \\ k \text{ odd}}}^{n-3}\prod_{\substack{l=k+3 \\ l \text{ odd}}}^{n-1}\frac{S_{k+1,l,l+1}R}{\delta_{l,l+1}}\right)\left(\prod_{\substack{k=m \\ k \text{ even}}}^{n-3}\prod_{\substack{l=k+3 \\ l \text{ odd}}}^{n-1}\frac{S_{k+1,l,l+1}R}{\delta_{l,l+1}}\right).
 \end{split}   
\end{equation}
Now, note that
\begin{equation*}
    \prod_{\substack{l=m+2 \\ l \text{ odd}}}^{n-1}\frac{1}{\delta_{l,l+1}} = \prod_{\substack{k=m \\ k \text{ odd }}}^{n-3}\frac{1}{\delta_{k+2,k+3}}, \end{equation*}so   
    \begin{equation*}
    \left(\prod_{\substack{l=m+2 \\ l \text{ odd}}}^{n-1}\frac{1}{\delta_{l,l+1}}\right) \left(\prod_{\substack{k=m \\ k \text{ odd} }}^{n-3}\frac{1}{\delta_{k+2,k+3}}\right) = \prod_{\substack{l=m+2 \\ l \text{ odd}}}^{n-1}\frac{1}{x_{l,l+1}}.
\end{equation*}
Additionally, using Lemma \ref{lemtpr} to conclude that 
\begin{equation*}
  \prod_{\substack{k=m \\ k \text{ odd}}}^{n-3}\prod_{\substack{l=k+3 \\ l \text{ odd}}}^{n-1}\frac{1}{\delta_{l,l+1}}  = \prod_{\substack{k=m \\ k \text{ even}}}^{n-3}\prod_{\substack{l=k+3 \\ l \text{ odd}}}^{n-1}\frac{1}{\delta_{l,l+1}},\end{equation*} 
  we obtain that
  \begin{equation*}
  \left(\prod_{\substack{k=m \\ k \text{ odd}}}^{n-3}\prod_{\substack{l=k+3 \\ l odd}}^{n-1}\frac{1}{\delta_{l,l+1}}\right) \left( \prod_{\substack{k=m \\ k \text{ even}}}^{n-3}\prod_{\substack{l=k+3 \\ l \text{ odd}}}^{n-1}\frac{1}{\delta_{l,l+1}}\right) = \prod_{\substack{k=m \\ k \text{ odd}}}^{n-3}\prod_{\substack{l=k+3 \\ l \text{ odd}}}^{n-1}\frac{1}{x_{l,l+1}}, 
\end{equation*}
and so we get
\begin{equation}
    \begin{split}
        Q_2 &= \left(\prod_{\substack{l=m+2 \\ l \text{ odd}}}^{n-1}\frac{S_{m-2,l,l+1}R}{x_{l,l+1}}\right)\left(\prod_{\substack{k=m \\ k \text{ odd}}}^{n-3} S_{k+1,k+2,k+3}R\right) \\ 
        & \qquad \qquad \left(\prod_{k=m}^{n-3}\prod_{\substack{l=k+3 \\ l \text{ odd}}}^{n-1}S_{k+1,l,l+1}R\right)\left(\prod_{\substack{k=m \\ k \text{ odd}}}^{n-3}\prod_{\substack{l=k+3 \\ l \text{ odd}}}^{n-1} \frac{1}{x_{l,l+1}}\right).
    \end{split}
    \label{eqptwo}
\end{equation}
Substituting \eqref{eqpone} and \eqref{eqptwo} into \eqref{eq42}, the statement of the theorem for $m$ odd, $3 < m \leq n-1$, follows.

\end{proof}
\begin{remark}Note that $\delta(m)$, as in Proposition \ref{prop9}, can be expressed in the form $\delta(m) = \frac{S(m)}{X(m)}$, where $S(m)$ is monomial in the entries of $F$ of the form $S_{***}$, and $X(m)$ is monomial in $R$ and the entries of $F$ of the form $x_{**}$.
\label{nap}
\end{remark}
\begin{prop} If $X(m)$ is as in Remark \ref{nap}, it holds that
\begin{equation*}
X(m) = \begin{cases}
         \begin{alignedat}{2}
&\frac{1}{R^{(\frac{n-2}{2})^2}}\prod_{\substack{l=m+1 \\ l \text{ odd}}}^{n-1} x_{l,l+1}^{\frac{l-3}{2}}\prod_{\substack{l=4 \\ l \text{ even}}}^{m-2} x_{l,l+1}^{\frac{l-2}{2}}; &\text{for even $m$, $2\leq m \leq n$} \\ 
& \frac{1}{R^{(\frac{n-2}{2})^2}} x_{12}^a \prod_{\substack{l=m+2 \\ l \text{ odd}}}^{n-1}x_{l,l+1}^{\frac{l-3}{2}}\prod_{\substack{l=3 \\ l \text{ odd}}}^{m-4}x_{l,l+1}\prod_{\substack{l=4 \\ l \text{ even}}}^{m-3}x_{l,l+1}^{\frac{l-2}{2}}; &\text{\quad for odd $m$, $1 \leq m \leq n-1$,}
 \end{alignedat}
 \end{cases}
    \end{equation*}
    where $a = 0$ for $m = 1$ and $m = 3$, and $a = 1$ for $m \geq 5$.
    \label{prop001}
\end{prop}
\begin{proof}
Let $m$ be even, and $2 \leq m \leq n$. From Proposition \ref{prop38} \eqref{deltam} it follows that
\begin{equation*}
 X(m) = \frac{1}{R^{\frac{n-2}{2}}}\left(\prod_{\substack{k=2 \\ k \text{ even}}}^{m-1}\prod_{\substack{l=k+1 \\ l \text{ even}}}^{m-1} \frac{1}{R^2}x_{l,l+1}\right) \left(\prod_{\substack{k=m \\ k \text{ even}}}^{n-3}\prod_{\substack{l=k+2 \\ l \text{ odd}}}^{n-1} \frac{1}{R^2}x_{l,l+1}\right)\left(\prod_{\substack{l=m+1 \\ l \text{ odd}}}^{n-1} \frac{1}{R^{m-2}}x_{l,l+1}^{\frac{m-2}{2}}\right).
\end{equation*}
\noindent Using the fact that $m$ is even, we have that
\begin{equation*}
    \prod_{\substack{k=2 \\ k \text{ even}}}^{m-1}\prod_{\substack{l=k+1 \\ l \text{ even}}}^{m-1} \frac{1}{R^2}x_{l,l+1} = \prod_{\substack{k=2 \\ k \text{ even}}}^{m-4}\prod_{\substack{l=k+2 \\ l \text{ even}}}^{m-2} \frac{1}{R^2}x_{l,l+1}.
    \end{equation*} 
    We also have that 
    \begin{equation*}
    \begin{split}
     \prod_{\substack{k=2 \\ k \text{ even}}}^{m-4}\prod_{\substack{l=k+2 \\ l \text{ even}}}^{m-2}  \frac{1}{R^2}x_{l,l+1} &= \frac{1}{R^{\frac{(m-4)(m-2)}{4}}}\prod_{\substack{k=2 \\ k \text{ even}}}^{m-4}\prod_{\substack{l=k+2 \\ l \text{ even}}}^{m-2}x_{l,l+1} \\
     &= \frac{1}{R^{\frac{(m-4)(m-2)}{4}}}\prod_{\substack{l=4 \\ l \text{ even}}}^{m-2}x_{l,l+1}^{\frac{l-2}{2}},
     \end{split}
    \end{equation*}
    where the last equality is due to Lemma \ref{lemspr} \eqref{eq2}. \\ \\ Similarly, one also gets \begin{equation*}
    \begin{split}
        \prod_{\substack{k=m \\ k \text{ even}}}^{n-3}\prod_{\substack{l=k+2 \\ l \text{ odd}}}^{n-1} \frac{1}{R^2}x_{l,l+1} &= \prod_{\substack{k=m \\ k \text{ even}}}^{n-4}\prod_{\substack{l=k+3 \\ l \text{ odd}}}^{n-1} \frac{1}{R^2}x_{l,l+1}\\ &= \frac{1}{R^{\frac{(n-m)(n-m-2)}{4}}}\prod_{\substack{k=m \\k \text{ even}}}^{n-4}\prod_{\substack{l=k+3 \\ l \text{ odd}}}^{n-1}x_{l,l+1} \\
        &= \frac{1}{R^{\frac{(n-m)(n-m-2)}{4}}}\prod_{\substack{l=m+3 \\ l \text{ odd}}}^{n-1}x_{l,l+1}^{\frac{l-(m+1)}{2}},
        \end{split}
\end{equation*}
where the last equality is due to Lemma \ref{lemspr} \eqref{eq1}. \\ \\
It follows that 
\begin{equation}
\begin{split}
    X(m) &= \frac{1}{R^{\frac{n-2}{2}}}\frac{1}{R^{\frac{(m-4)(m-2)}{4}}}\left(\prod_{\substack{l=4 \\ l \text{ even}}}^{m-2}x_{l,l+1}^{\frac{l-2}{2}}\right)\frac{1}{R^{\frac{(n-m)(n-m-2)}{4}}}\left(\prod_{\substack{l=m+3 \\ l \text{ odd}}}^{n-1}x_{l,l+1}^{\frac{l-(m+1)}{2}}\right)\left(\prod_{\substack{l=m+1 \\ l \text{ odd}}}^{n-1}\frac{1}{R^{m-2}}x_{l,l+1}^{\frac{m-2}{2}}\right)\\ 
    &= \frac{1}{R^{\frac{n-2}{2}}}\frac{1}{R^{\frac{(m-4)(m-2)}{4}}}\frac{1}{R^{\frac{(n-m)(n-m-2)}{4}}}x_{m+1,m+2}^{\frac{m-2}{2}}\prod_{\substack{l=4 \\ l \text{ even}}}^{m-2}x_{l,l+1}^{\frac{l-2}{2}}\prod_{\substack{l=m+3 \\ l \text{ odd}}}^{n-1}x_{l,l+1}^{\frac{l-(m+1)}{2}}\prod_{\substack{l=m+1 \\ l \text{ odd}}}^{n-1}\frac{1}{R^{m-2}}\prod_{\substack{l=m+3 \\ l \text{ odd}}}^{n-1}x_{l,l+1}^{\frac{m-2}{2}}\\
    &= \frac{1}{R^{\frac{n-2}{2}}}\frac{1}{R^{\frac{(m-4)(m-2)}{4}}}\frac{1}{R^{\frac{(n-m)(n-m-2)}{4}}}\frac{1}{R^{\frac{(n-m)(m-2)}{2}}}x_{m+1,m+2}^{\frac{m-2}{2}}\prod_{\substack{l=m+3 \\ l \text{ odd}}}^{n-1}x_{l,l+1}^{\frac{l-3}{2}}\prod_{\substack{l=4 \\ l \text{ even}}}^{m-2}x_{l,l+1}^{\frac{l-2}{2}} \\
    &= \frac{1}{R^{(\frac{n-2}{2})^2}}\prod_{\substack{l=m+1 \\ l \text{ odd}}}^{n-1}x_{l,l+1}^{\frac{l-3}{2}}\prod_{\substack{l=4 \\ l \text{ even}}}^{m-2}x_{l,l+1}^{\frac{l-2}{2}},
    \end{split}
\end{equation}
for even $m$, $2 \leq m \leq n$. \\ \\ 
Let now $m = 1$, and recall that 
\begin{equation*}
    \delta(1) = \left(\prod_{k=1}^{n-3}\prod_{\substack {l=k+2 \\ l \text{ odd}}}^{n-1}S_{k+1,l,l+1}R\right) \left(\prod_{\substack{k=2 \\ k \text{ even}}}^{n-3}\prod_{\substack {l=k+2 \\ l \text{ odd}}}^{n-1}\frac{1}{x_{l,l+1}}\right).
\end{equation*}
It follows that 
\begin{equation*}
\begin{split}
X(1) &= \left(\prod_{k=1}^{n-3}\prod_{\substack{l=k+2 \\l \text{ odd}}}^{n-1}\frac{1}{R}\right)\left(\prod_{\substack {k=2 \\ k \text{ even}}}^{n-3}\prod_{\substack{l=k+2 \\ l \text{ odd}}}^{n-1}x_{l,l+1}\right)\\
&= \left(\prod_{k=1}^{n-3}\prod_{\substack{l=k+2 \\l \text{ odd}}}^{n-1}\frac{1}{R}\right)\left(\prod_{\substack {k=2 \\ k \text{ even}}}^{n-4}\prod_{\substack{l=k+3 \\ l odd}}^{n-1}x_{l,l+1} \right),
\end{split}
\end{equation*}
where the last equality is due to the fact that $n$ is even, and for $k$ even, $l = k+3$ is the smallest odd number such that $l\geq k+2$. \\ \\
\noindent Using Lemma \ref{lemgrp} \eqref{eqa} and Lemma \ref{lemspr} \eqref{eq1}, respectively, we get that 
\begin{equation*}
\prod_{k=1}^{n-3}\prod_{\substack{l=k+2 \\l \text{ odd}}}^{n-1}\frac{1}{R} = \frac{1}{R^{(\frac{n-2}{2})^{2}}}, \prod_{\substack {k=2 \\ k \text{ even}}}^{n-4}\prod_{\substack{l=k+3 \\ l \text{ odd}}}^{n-1}x_{l,l+1} = \prod_{\substack{l=5 \\ l \text{ odd}}}^{n-1}x_{l,l+1}^{\frac{l-3}{2}},
\end{equation*}
so 
\begin{equation*}
    X(1) =  \frac{1}{R^{(\frac{n-2}{2})^{2}}}  \prod_{\substack{l=5 \\ l \text{ odd}}}^{n-1}x_{l,l+1}^{\frac{l-3}{2}}.
    \end{equation*}
    Recall as well that 
    \begin{equation*}
    \delta(3) = S_{124}R\left(\prod_{\substack{l=5 \\ l \text{ odd }}}^{n-1}S_{1,l,l+1}S_{2,l,l+1}R^2\right)\left(\prod_{k=3}^{n-3}\prod_{\substack{l=k+2 \\ \text{ l odd}}}^{n-1}S_{k+1,l,l+1}R\right)\left(\prod_{\substack{k=3 \\k \text{ odd}}}^{n-3}\prod_{\substack{l=k+2 \\ l \text{ odd}}}^{n-1}\frac{1}{x_{l,l+1}}\right),
\end{equation*}
which we can rewrite as 
\begin{equation*}
\begin{split}
\delta(3) &= S_{124}\left(\prod_{\substack{l=5 \\ l \text{ odd}}}^{n-1}S_{1,l,l+1}S_{2,l,l+1}\right)\left(\prod_{k=3}^{n-3}\prod_{\substack{l=k+2 \\ l \text{ odd}}}^{n-1}S_{k+1,l,l+1}\right) \\ & \qquad \qquad \left(\prod_{k=1}^{n-3}\prod_{\substack{l=k+2 \\ l \text{ odd}}}^{n-1}R\right)\left(\prod_{\substack{k=3 \\k \text{ odd}}}^{n-3}\prod_{\substack{l=k+2 \\ l \text{ odd}}}^{n-1}\frac{1}{x_{l,l+1}}\right).
\end{split}
\end{equation*}
One easily checks that 
\begin{equation*}
\prod_{\substack{k=3 \\ k \text{ odd}}}^{n-3}\prod_{\substack{l=k+2 \\ l \text{ odd}}}^{n-1}x_{l,l+1} = \prod_{\substack{k=2 \\ k \text{ even}}}^{n-4}\prod_{\substack{l=k+3 \\ l \text{ odd}}}^{n-1}x_{l,l+1},
\end{equation*}
so we conclude that  $X(1) = X(3)$. \\  \\ 
Let now $m$ be odd, and $5 \leq m \leq n-1$. Using Proposition \ref{prop38} \eqref{deltammm}, Remark \ref{nap}, and combining and reorganizing some of the terms of $\delta(m)$, we get that
\begin{equation*}
\begin{split}
        X(m) &= \frac{1}{R^3} x_{12}\left(\prod_{\substack{l=m+2 \\ l \text{ odd}}}^{n-1}\frac{1}{R^4}x_{l,l+1}^2\right)\left(\prod_{\substack{k=3 \\ k \text{ odd}}}^{m-4}\left(\frac{1}{R^2}x_{k,k+1}\prod_{\substack{l=k+1 \\ l \text{ even}}}^{m-2}\frac{1}{R}x_{l,l+1}\prod_{\substack{l=m+2 \\ l \text{ odd}}}^{n-1}x_{l,l+1}\right)\right) 
        \left(\prod_{\substack{k=m \\ k \text{ odd }}}^{n-3}\prod_{\substack{l=k+3 \\ l \text{ odd}}}^{n-1}x_{l,l+1}\right) \\
        & \qquad \qquad \left(\prod_{\substack{l=2 \\ l \text{ even}}}^{m-2}\frac{1}{R}\right)  \left(\prod_{\substack{k=2 \\k \text{ even}}}^{m-5}\prod_{\substack{l=k+2 \\ l \text{ even}}}^{m-2}\frac{1}{R}\right)\left(\prod_{k=2}^{m-4}\prod_{\substack{l=m+2 \\ l \text{ odd}}}^{n-1}\frac{1}{R}\right) \left(\prod_{\substack{k=m \\ k \text{ odd}}}^{n-3}\frac{1}{R}\right)\left(\prod_{k=m}^{n-3}\prod_{\substack{l=k+3 \\ l \text{ odd}}}^{n-1}\frac{1}{R}\right).
 \end{split}       
\end{equation*}
Let us now combine and rewrite some terms of $X(m)$. Namely,
\begin{equation}
    \frac{1}{R^3}x_{12}\left(\prod_{\substack{l=m+2 \\ l \text{ odd}}}^{n-1}\frac{1}{R^4}x_{l,l+1}^2\right)\left(\prod_{\substack{k=m \\ k \text{ odd }}}^{n-3}\prod_{\substack{l=k+3 \\ l \text{ odd}}}^{n-1}x_{l,l+1}\right) = \frac{1}{R^{2n-2m+1}}x_{12}\left(\prod_{\substack{l=m+2 \\ l \text{ odd}}}^{n-1}x_{l,l+1}^{2}\right)\left(\prod_{\substack{k=m \\ k \text{ odd }}}^{n-3}\prod_{\substack{l=k+3 \\ l \text{ odd}}}^{n-1}x_{l,l+1}\right).
    \label{eq3}
\end{equation}
Using Lemma \ref{lemtpr} one gets that 
\begin{equation*}
    \prod_{\substack{k=m \\ k \text{ odd}}}^{n-3}\prod_{\substack{l=k+3 \\ l \text{ odd}}}^{n-1}x_{l,l+1} =   \prod_{\substack{k=m \\ k \text{ even}}}^{n-3}\prod_{\substack{l=k+3 \\ l \text{ odd}}}^{n-1}x_{l,l+1},
\end{equation*}
and using the facts that $m$ is odd, $n$ is even,
\begin{equation*}
     \prod_{\substack{k=m \\ k \text{ even}}}^{n-3}\prod_{\substack{l=k+3 \\ l \text{ odd}}}^{n-1}x_{l,l+1} =  \prod_{\substack{k=m+1 \\ k \text{ even}}}^{n-4}\prod_{\substack{l=k+3 \\ l \text{ odd}}}^{n-1}x_{l,l+1}.
\end{equation*}
Now, using Lemma \ref{lemspr} \eqref{eq1}, we have that \begin{equation*}
\prod_{\substack{k=m+1 \\ k \text{ even}}}^{n-4}\prod_{\substack{l=k+3 \\ l \text{ odd}}}^{n-1}x_{l,l+1} = \prod_{\substack{l=m+4 \\ l \text{ odd}}}^{n-1}x_{l,l+1}^{\frac{l-(m+2)}{2}} = \prod_{\substack{l=m+2 \\ l \text{ odd}}}^{n-1}x_{l,l+1}^{\frac{l-(m+2)}{2}},
\end{equation*} 
so we can  rewrite the equality \eqref{eq3} as 
\begin{equation}
    \frac{1}{R^3}x_{12}\left(\prod_{\substack{l=m+2 \\ l \text{ odd}}}^{n-1}\frac{1}{R^4}x_{l,l+1}^2\right)\left(\prod_{\substack{k=m \\ k \text{ odd }}}^{n-3}\prod_{\substack{l=k+3 \\ l \text{ odd}}}^{n-1}x_{l,l+1}\right) = \frac{1}{R^{2n-2m+1}}x_{12}\prod_{\substack{l=m+2 \\ l \text{ odd}}}^{n-1}x_{l,l+1}^{\frac{l-m+2}{2}}.
    \label{eq200}
    \end{equation}
Also, note that:
\begin{equation}
    \prod_{\substack{k=3 \\ k \text{ odd}}}^{m-4}\prod_{\substack{l=k+1 \\ l \text{ even}}}^{m-2}x_{l,l+1} = \prod_{\substack{k=2 \\ k \text{ even}}}^{m-4}\prod_{\substack{l=k+2 \\ l \text{ even}}}^{m-2}x_{l,l+1},
    \label{eq80}
\end{equation}
by Lemma \ref{lem4}, and 
\begin{equation}
    \prod_{\substack{k=2 \\ k \text{ even}}}^{m-4}\prod_{\substack{l=k+2 \\ l \text{ even}}}^{m-2}x_{l,l+1} = \prod_{\substack{k=2 \\ k \text{ even}}}^{m-5}\prod_{\substack{l=k+2 \\ l \text{ even}}}^{m-3}x_{l,l+1} = \prod_{\substack{l=4 \\ l \text{ even}}}^{m-3}x_{l,l+1}^{\frac{l-2}{2}}, 
    \label{eq90}
\end{equation}
 by Lemma \ref{lemspr} \eqref{eq2}, using the assumption that $m$ is odd. \\ \\
From \eqref{eq80} and \eqref{eq90}, it follows that 
\begin{equation}
    \prod_{\substack{k=3 \\ k \text{ odd}}}^{m-4}\prod_{\substack{l=k+1 \\ l \text{ even}}}^{m-2}x_{l,l+1} = \prod_{\substack{l=4 \\l \text{ even}}}^{m-3}x_{l,l+1}^{\frac{l-2}{2}}.
    \label{eq5}
\end{equation}
Using Lemma \ref{lem8} \eqref{eq10}, equality \eqref{eq5} and the fact that 
\begin{equation*}
 \prod_{\substack{k=3 \\ k \text{ odd}}}^{m-4} \frac{1}{R^2} = \frac{1}{R^{m-5}},\prod_{\substack{k=3 \\ k \text{ odd}}}^{m-4}\prod_{\substack{l=m+2 \\ l \text{ odd}}}^{n-1}x_{l,l+1} = \prod_{\substack{l=m+2 \\ l \text{ odd}}}^{n-1}x_{l,l+1}^{\frac{m-5}{2}},
    \end{equation*}gives us the following: 
\begin{equation}
\begin{split}
    \prod_{\substack{k=3 \\ k \text{ odd}}}^{m-4}\left(\frac{1}{R^2}x_{k,k+1}\prod_{\substack{l=k+1 \\ l \text{ even}}}^{m-2}\frac{1}{R}x_{l,l+1}\prod_{\substack{l=m+2 \\ l \text{ odd}}}^{n-1}x_{l,l+1}\right) &= \left(\prod_{\substack{k=3 \\ k \text{ odd}}}^{m-4}\frac{1}{R^2}\right)\left(\prod_{\substack{k=3 \\ k \text{ odd}}}^{m-4}\prod_{\substack{l=k+1 \\ l \text{ even}}}^{m-2}\frac{1}{R}\right)  \left(\prod_{\substack{k=3 \\ k \text{ odd}}}^{m-4}x_{k,k+1}\right) \\ & \qquad \qquad \left(\prod_{\substack{k=3 \\ k \text{ odd}}}^{m-4} \prod_{\substack{l=k+1 \\ l \text{ even}}}^{m-2}x_{l,l+1}\right) \left(\prod_{\substack{k=3 \\ k \text{ odd}}}^{m-4}\prod_{\substack{l=m+2 \\ l \text{ odd}}}^{n-1}x_{l,l+1}\right) \\ &= \frac{1}{R^{m-5}}  \frac{1}{R^{\frac{(m-3)(m-5)}{8}}}\prod_{\substack{l=3 \\ l \text{ odd}}}^{m-4}x_{l,l+1}\prod_{\substack{l=4 \\ l \text{ even}}}^{m-3}x_{l,l+1}^{\frac{l-2}{2}}\prod_{\substack{l=m+2 \\ l \text{ odd}}}^{n-1}x_{l,l+1}^{\frac{m-5}{2}}\\ &= \frac{1}{R^{\frac{m^2-25}{8}}}\prod_{\substack{l=3 \\ l \text{ odd}}}^{m-4}x_{l,l+1}\prod_{\substack{l=4 \\ l \text{ even}}}^{m-3}x_{l,l+1}^{\frac{l-2}{2}}\prod_{\substack{l=m+2 \\ l \text{ odd}}}^{n-1}x_{l,l+1}^{\frac{m-5}{2}}.
    \end{split}
    \label{eq100}
\end{equation}
Now, using the easily verifiable facts that
\begin{equation*}
 \prod_{\substack{l=2 \\ l \text{ even}}}^{m-2}\frac{1}{R} = \prod_{\substack{l=2 \\ l \text{ even}}}^{m-3}\frac{1}{R} = \frac{1}{R^{\frac{m-3}{2}}}, 
 \end{equation*}
 \begin{equation*}
     \prod_{\substack{k=m \\ k \text{ odd}}}^{n-3}\frac{1}{R} = \frac{1}{R^{\frac{n-m-1}{2}}},
 \end{equation*}
 and
 \begin{equation*}
     \prod_{k=2}^{m-4}\prod_{\substack{l=m+2 \\ l \text{ odd}}}^{n-1}\frac{1}{R} = \prod_{\substack{l=m+2 \\ l \text{ odd}}}^{n-1}\frac{1}{R^{m-5}} = \frac{1}{R^{\frac{(n-m-1)(m-5)}{2}}},
 \end{equation*} \\
 and Lemma \ref{lem8} \eqref{eq11} together with Lemma \ref{lemgrp} \eqref{eqb}, and the assumption that $m$ is odd, we conclude that 
 \begin{equation*}
\prod_{\substack{k=2 \\ k \text{ even}}}^{m-5}\prod_{\substack{l=k+2 \\ l \text{ even}}}^{m-2}  \frac{1}{R} = \prod_{\substack{k=2 \\ k \text{ even}}}^{m-4}\prod_{\substack{l=k+2 \\ l \text{ even}}}^{m-2}\frac{1}{R} = \frac{1}{R^{\frac{(m-3)(m-5)}{8}}},
 \end{equation*}
 and
 \begin{equation*}
     \prod_{k=m}^{n-3}\prod_{\substack{l=k+3 \\ 
 l \text{ odd}}}^{n-1}\frac{1}{R} = \frac{1}{R^{\frac{(n-m-1)(n-m-3)}{4}}}.
 \end{equation*}
 We get 
     \begin{equation}
  \left(\prod_{\substack{l=2 \\ l \text{ even}}}^{m-2}\frac{1}{R}\right) \left(\prod_{\substack{k=2 \\k \text{ even}}}^{m-5}\prod_{\substack{l=k+2 \\ l \text{ even}}}^{m-2}\frac{1}{R}\right)\left(\prod_{k=2}^{m-4}\prod_{\substack{l=m+2 \\ l \text{ odd}}}^{n-1}\frac{1}{R}\right) \left(\prod_{\substack{k=m \\ k \text{ odd}}}^{n-3}\frac{1}{R}\right)\left(\prod_{k=m}^{n-3}\prod_{\substack{l=k+3 \\ l \text{ odd}}}^{n-1}\frac{1}{R}\right) = \frac{1}{R^{\frac{-m^2+16m-24n+25+2n^2}{8}}}.
  \label{eq101}
 \end{equation}
Using \eqref{eq200}, \eqref{eq100} and \eqref{eq101}, we get that 
\begin{equation*}
\begin{split}
    X(m) &= \frac{1}{R^{(\frac{n-2}{2})^2}}x_{12}\prod_{\substack{l=m+2 \\ l \text{ odd}}}^{n-1}x_{l,l+1}^{\frac{l-m+2}{2}}\prod_{\substack{l=3 \\ l \text{ odd}}}^{m-4}x_{l,l+1}\prod_{\substack{l=4 \\ l \text{ even}}}^{m-3}x_{l,l+1}^{\frac{l-2}{2}}\prod_{\substack{l=m+2 \\ l \text{ odd}}}^{n-1}x_{l,l+1}^{\frac{m-5}{2}} \\ &= \frac{1}{R^{(\frac{n-2}{2})^2}}x_{12}\prod_{\substack{l=m+2 \\ l \text{ odd}}}^{n-1}x_{l,l+1}^{\frac{l-3}{2}}\prod_{\substack{l=3 \\ l \text{ odd}}}^{m-4}x_{l,l+1}\prod_{\substack{l=4 \\ l \text{ even}}}^{m-3}x_{l,l+1}^{\frac{l-2}{2}},
    \end{split}
\end{equation*}
for $m$ odd, $5 \leq m \leq n-1 $, completing the proof. 
 
    \end{proof}
    \begin{prop}If $S(m)$ is as in Remark \ref{nap}, it holds that
    \begin{equation}
     S(m) = 
     \left(\prod_{k=1}^{m-1}\left(\prod_{\substack {l=k+1 \\ l \text{ even}}}^{m-1}S_{k,l,l+1} \prod_{\substack{l=m+1 \\ l \text{ odd }}}^{n-1}S_{k,l,l+1}\right)\right)\left(\prod_{k=m}^{n-3}\prod_{\substack {l=k+2 \\ l \text{ odd }}}^{n-1}S_{k+1,l,l+1}\right),
     \label{sssm}
    \end{equation}
for $m$ even, $2 \leq m \leq n$;
    \begin{equation}
    \begin{split}
        S(m) &= S_{m-2,m-1,m+1}\left(\prod_{\substack{l=m+2 \\ l \text{ odd}}}^{n-1}S_{m-1,l,l+1}S_{m-2,l,l+1}S_{m-3,l,l+1}\right) \left(\prod_{k=m}^{n-3}\prod_{\substack{l=k+2 \\ l \text{ odd}}}^{n-1}S_{k+1,l,l+1}\right) \\ & \qquad \qquad  \left(\prod_{k=1}^{m-4}\left((S_{k,k+1,k+2}S_{k,k+1,m-1}S_{k,k+1,m+1})^{(k \text{ mod  
 2)}}\prod_{\substack{l=k+2 \\ l \text{ even}}}^{m-2}S_{k,l,l+1}\prod_{\substack{l=m+2 \\ l \text{ odd}}}^{n-1}S_{k,l,l+1}\right)\right), 
        \end{split}
    \end{equation}
    for odd $m$, $5 \leq m \leq n-1$;
    \begin{equation}
    S(1) = \prod_{k=1}^{n-3}\prod_{\substack{l=k+2 \\ l \text{ odd}}}^{n-1}S_{k+1,l,l+1};
    \label{nznm}
    \end{equation}
    and
    \begin{equation}
    S(3) = S_{124}\left(\prod_{\substack{l=5 \\ l \text{ odd}}}^{n-1}S_{1,l,l+1}S_{2,l,l+1}\right)\left(\prod_{k=3}^{n-3}\prod_{\substack{l=k+2 \\ l \text{ odd}}}^{n-1}S_{k+1,l,l+1}\right).
    \label{nznmm}
    \end{equation}
    \label{prop002}
    \end{prop}
    \begin{proof}
    Equality \eqref{sssm} follows from Proposition \ref{prop9} \eqref{deltam}, and equalities \eqref{nznm} and \eqref{nznmm}
     follow from Proposition \ref{prop9} \eqref{dm} and Proposition \ref{prop9} \eqref{deltamm}, respectively.\\ \\
     From Proposition \ref{prop9}\eqref{deltammm} it follows that 
     \begin{equation*}
     \begin{split}
         S(m) &= S_{1,2,m-1}S_{1,2,m+1}S_{m-2,m-1,m+1}\left(\prod_{\substack{l=2 \\ l \text{ even}}}^{m-2}S_{1,l,l+1}\right)\left(\prod_{\substack{k=m \\ k \text{ odd}}}^{n-3}S_{k+1,k+2,k+3}\right)\left(\prod_{k=m}^{n-3}\prod_{\substack{l=k+3 \\ l \text{ odd}}}^{n-1}S_{k+1,l,l+1}\right) \\ & \qquad \qquad \left(\prod_{\substack{l=m+2 \\ l \text{ odd}}}^{n-1}S_{1,l,l+1}S_{m-1,l,l+1}S_{m-2,l,l+1}S_{m-3,l,l+1}\right) \left(\prod_{\substack{k=3\\ k \text{ odd}}}^{m-4}\left(S_{k,k+1,m-1}S_{k,k+1,m+1}\prod_{\substack{l=k+1 \\ l \text{ even}}}^{m-2}S_{k,l,l+1}\right)\right) \\ & \qquad \qquad \left(\prod_{\substack{k=2 \\ k \text{ even}}}^{m-4}\prod_{\substack{l=k+2 \\ l \text{ even}}}^{m-2}S_{k,l,l+1}\right) \left(\prod_{k=2}^{m-4}\prod_{\substack{l=m+2 \\ l odd}}^{n-1}S_{k,l,l+1}\right),
         \end{split}
         \end{equation*}
         for $m$ odd, $5 \leq m \leq n-1$. \\ \\
Since 
\begin{equation*}
    S_{1,2,m-1}S_{1,2,m+1}\left(\prod_{\substack{k=3 \\ k \text{ odd}}}^{m-4}S_{k,k+1,m-1}S_{k,k+1,m+1}\right) = \prod_{\substack{k=1 \\ k \text{ odd}}}^{m-4}S_{k,k+1,m-1}S_{k,k+1,m+1},
\end{equation*}
\begin{equation*}
    \left(\prod_{\substack{l=m+2 \\ l \text{ odd}}}^{n-1}S_{1,l,l+1}\right)\left(\prod_{k=2}^{m-4}\prod_{\substack{l=m+2 \\ l \text{ odd}}}^{n-1}S_{k,l,l+1}\right) = \prod_{k=1}^{m-4}\prod_{\substack{l=m+2\\ l \text{ odd}}}^{n-1}S_{k,l,l+1},
\end{equation*}
and
\begin{equation*}
    \left(\prod_{k=m}^{n-3}\prod_{\substack{l=k+3 \\ l \text{ odd}}}^{n-1}S_{k+1,l,l+1}\right)\left(\prod_{\substack{k=m \\ k \text{ odd}}}^{n-3}S_{k+1,k+2,k+3}\right) = \prod_{k=m}^{n-3}\prod_{\substack{l=k+2 \\ l \text{ odd}}}^{n-1}S_{k+1,l,l+1},
\end{equation*}
it follows that, for $m$ odd, $5 \leq m \leq n-1$,
\begin{equation*}
\begin{split}
S(m) &= 
S_{m-2,m-1,m+1}\left(\prod_{\substack{l=2 \\ l \text{ even}}}^{m-2}S_{1,l,l+1}\right) \left(\prod_{k=m}^{n-3}\prod_{\substack{l=k+2 \\ l \text{ odd}}}^{n-1}S_{k+1,l,l+1}\right)\left(\prod_{\substack{l=m+2 \\ l \text{ odd}}}^{n-1}S_{m-1,l,l+1}S_{m-2,l,l,l+1}S_{m-3,l,l+1}\right) \\ & \qquad \qquad \left(\prod_{\substack{k=1 \\ k \text{ odd}}}^{m-4}S_{k,k+1,m-1}S_{k,k+1,m+1}\right)  \left(\prod_{\substack{k=3 \\ k \text{ odd}}}^{m-4}\prod_{\substack{l=k+1 \\ l \text{ even}}}^{m-2}S_{k,l,l+1}\right)\left(\prod_{\substack{k=2 \\ k \text{ even}}}^{m-4}\prod_{\substack{l=k+2 \\ l \text{ even}}}^{m-2}S_{k,l,l+1}\right)\left(\prod_{k=1}^{m-4}\prod_{\substack{l=m+2 \\ l \text{ odd}}}^{n-1}S_{k,l,l+1}\right).
\end{split}
\end{equation*}
Now, using that 
\begin{equation*}
    \left(\prod_{\substack{k=2 \\ k \text{ even}}}^{m-4}\prod_{\substack{l=k+2 \\ l \text{ even}}}^{m-2}S_{k,l,l+1}\right)\left(\prod_{\substack{k=3 \\ k \text{ odd}}}^{m-4}\prod_{\substack{l=k+1 \\ l \text{ even}}}^{m-2}S_{k,l,l+1}\right) = \left(\prod_{k=2}^{m-4}\prod_{\substack{l=k+2 \\ l \text{ even}}}^{m-2}S_{k,l,l+1}\right)\left(\prod_{\substack{k=3 \\ k \text{ odd}}}^{m-4}S_{k,k+1,k+2}\right),
\end{equation*}
we have 
\begin{equation}
\begin{split}
S(m) &= S_{m-2,m-1,m+1}  \left(\prod_{k=m}^{n-3}\prod_{\substack{l=k+2 \\ l \text{ odd}}}^{n-1}S_{k+1,l,l+1}\right) \left(\prod_{\substack{l=m+2 \\ l \text{ odd}}}^{n-1}S_{m-1,l,l+1}S_{m-2,l,l+1}S_{m-3,l,l+1}\right)\\ & \qquad \qquad \left(\prod_{\substack{k=1 \\ k \text{ odd}}}^{m-4}S_{k,k+1,m-1}S_{k,k+1,m+1}\right)\left(\prod_{k=1}^{m-4}\prod_{\substack{l=m+2 \\ l \text{ odd}}}^{n-1}S_{k,l,l+1}\right) \left(\prod_{\substack{l=2 \\ l \text{ even}}}^{m-2}S_{1,l,l+1}\right) \\ &\qquad \qquad \left(\prod_{\substack{k=3 \\ k \text{ odd}}}^{m-4}S_{k,k+1,k+2}\right)\left(\prod_{k=2}^{m-4}\prod_{\substack{l=k+2 \\ l \text{ even}}}^{m-2}S_{k,l,l+1}\right).
\end{split}
\label{nnn}
\end{equation}
Finally, noting that
\begin{equation*}
    \prod_{\substack{l=2 \\ l \text{ even}}}^{m-2}S_{1,l,l+1} = S_{123}\prod_{\substack{l=4 \\ l \text{ even}}}^{m-2}S_{1,l,l+1},
\end{equation*}
and then applying equalities 
\begin{equation*}
    S_{123}\prod_{\substack{k=3 \\ k \text{ odd}}}^{m-4}S_{k,k+1,k+2} = \prod_{\substack{k=1 \\ k \text{ odd}}}^{m-4}S_{k,k+1,k+2},
\end{equation*}
and
\begin{equation*}
    \left(\prod_{\substack{l=4 \\ l \text{ even}}}^{m-2}S_{1,l,l+1}\right)\left(\prod_{k=2}^{m-4}\prod_{\substack{l=k+2 \\ l \text{ even}}}^{m-2}S_{k,l,l+1}\right) = \prod_{k=1}^{m-4}\prod_{\substack{l=k+2 \\ l \text{ even}}}^{m-2}S_{k,l,l+1}
\end{equation*}
to \eqref{nnn}, gives
\begin{equation*}
\begin{split}
    S(m) & = S_{m-2,m-1,m+1} \left(\prod_{k=m}^{n-3}\prod_{\substack{l=k+2 \\ l \text{ odd}}}^{n-1}S_{k+1,l,l+1}\right)\left(\prod_{\substack{l=m+2 \\ l \text{ odd}}}^{n-1}S_{m-1,l,l+1}S_{m-2,l,l+1}S_{m-3,l,l+1}\right) \\ & \qquad \qquad \left(\prod_{\substack{k=1 \\ k \text{ odd}}}^{m-4}S_{k,k+1,m-1}S_{k,k+1,m+1}\right) \left(\prod_{k=1}^{m-4}\prod_{\substack{l=m+2 \\ l \text{ odd}}}^{n-1}S_{k,l,l+1}\right) \left(\prod_{\substack{k=1 \\ k \text{ odd}}}^{m-4}S_{k,k+1,k+2}\right)\left(\prod_{k=1}^{m-4}\prod_{\substack{l=k+2 \\ l \text{ even}}}^{m-2}S_{k,l,l+1}\right) \\ & = S_{m-2,m-1,m+1}\left(\prod_{\substack{l=m+2 \\ l \text{ odd}}}^{n-1}S_{m-1,l,l+1}S_{m-2,l,l+1}S_{m-3,l,l+1}\right) \left(\prod_{k=m}^{n-3}\prod_{\substack{l=k+2 \\ l \text{ odd}}}^{n-1}S_{k+1,l,l+1}\right) \\ & \qquad \qquad \left(\prod_{k=1}^{m-4}\left(\left(S_{k,k+1,k+2}S_{k,k+1,m-1}S_{k,k+1,m+1}\right)^{(k \text{ mod 
 2)}}\prod_{\substack{l=k+2 \\ l \text{ even}}}^{m-2}S_{k,l,l+1}\prod_{\substack{l=m+2 \\ l \text{ odd}}}^{n-1}S_{k,l,l+1}\right)\right),
\end{split}
\end{equation*}
for $m$ odd, $5 \leq m \leq n-1$.
         \end{proof}
 \begin{prop} Let $X(m)$, for $1\leq m \leq n$, be as in Proposition \ref{prop001}. Then 
     \begin{equation*}
    L = \lcm (X(m))_{1 \leq m \leq n} = \frac{1}{R^{(\frac{n-2}{2})^2}}x_{12}x_{34}^{b}\prod_{\substack{l=4 \\ l \text{ even}}}^{n-2}x_{l,l+1}^{\frac{l-2}{2}}\prod_{\substack{l=5 \\ l \text{ odd}}}^{n-1}x_{l,l+1}^{\frac{l-3}{2}},
\end{equation*}
where $b = 0$ for $n = 6$, and $b = 1$ for $n \geq 8$.
\label{lcm}
     \end{prop} 
     \begin{proof}Let $l \in \{1,2,3,...,n-1\}$, and let us find the maximum exponent of $x_{l,l+1}$ amongst the $X(m)$ for $m \in \{1,2,...,n\}$. \\
Note that exponent of $x_{23}$ is equal to zero in each $X(m)$, for $m \in \{1,2,...,n\}$.\\
Tables \ref{tab:table 1}, \ref{tab:table 2}, \ref{tab:table 3} and \ref{tab: table 4}, given below show the exponents of $x_{l,l+1}$ in the $X(m)$.

\begin{table}[h!]
 \centering
 \begin{tabular}{|c|c|}
 \hline
 $\mathbf{range \quad of \quad  values \quad  of \quad m}$ & \textbf{exponent of $\mathbf{x_{l,l+1}}$ in $\mathbf{X(m)}$} \\
 \hline
 $m$ even, such that $l+2\leq m \leq n$ & $\frac{l-2}{2}$
 \\ \hline
 $m$ even, such that $m<l+2$ & $0$
 \\ \hline
 $m$ odd, such that $m<l+3$ & $0$
 \\ \hline
 $m$ odd, such that $l+3 \leq m \leq n-1$ & $\frac{l-2}{2}$\\ 
 \hline
 \end{tabular}
 \caption{Exponents of $x_{l,l+1}$ for $l$ even, $4\leq l \leq n-2$}
 \label{tab:table 1}
\end{table}
\begin{table}[h!]
 \centering
 \begin{tabular}{|c|c|}
 \hline
 $\mathbf{range \quad of \quad values \quad of \quad m}$ & \textbf{exponent of $\mathbf{x_{l,l+1}}$ in $\mathbf{X(m)}$} \\
 \hline
 $m$ even, such that $2\leq m \leq l-1$ & $\frac{l-3}{2}$
 \\ \hline
 $m$ even, such that $l-1< m\leq n$ & $0$
 \\ \hline
 $m$ odd, such that $1 \leq m \leq l-2$ & $\frac{l-3}{2}$
 \\ \hline
 $m$ odd, such that $l-2 < m <l+4$ & $0$\\ 
 \hline
 $m$ odd, such that $l+4 \leq m \leq n-1$ & $1$ \\
 \hline
 \end{tabular}
 \caption{Exponents of $x_{l,l+1}$ for $l$ odd, $5\leq l \leq n-1$}
 \label{tab:table 2}
\end{table}
\begin{table}[h!]
 \centering
 \begin{tabular}{|c|c|}
 \hline
 $\mathbf{range \quad of \quad values \quad of \quad m}$ & \textbf{exponent of $\mathbf{x_{12}}$ in $\mathbf{X(m)}$} \\
 \hline
 $m$ odd, such that $5 \leq m \leq n-1$ & $1$
 \\ \hline
 $m \in \{1,3\}$ & $0$
 \\ \hline
 $m$ even, such that $2 \leq m \leq n$ & $0$
 \end{tabular}
 \caption{Exponents of $x_{12}$}
  \label{tab:table 3}
\end{table}
\begin{table}[h!]
 \centering
 \begin{tabular}{|c|c|}
 \hline
 $\mathbf{range \quad of \quad values \quad of \quad m}$ & \textbf{exponent of $\mathbf{x_{34}}$ in $\mathbf{X(m)}$} \\
 \hline
 $m$ odd, such that $7\leq m \leq n-1$ & $1$
 \\ \hline
 $m \in \{1,3,5\}$ & $0$
 \\ \hline
 $m$ even, such that $2 \leq m \leq n$ & $0$
 \end{tabular}
 \caption{Exponents of $x_{34}$}
  \label{tab: table 4}
\end{table}
\noindent From Table \ref{tab:table 1}, we conclude that,  for $l$ even, $4 \leq l \leq n-2$, the maximum of all exponents of $x_{l,l+1}$, in $X(m)$, where $m \in \{1,2,...,n\}$, is $\frac{l-2}{2}$.
 \\ Since for $l \geq 5$, $\frac{l-3}{2} \geq 1$, we conclude  from Table \ref{tab:table 2} that, for $l$ odd, $5 \leq l \leq n-1$, the maximum exponent of $x_{l,l+1}$ is $\frac{l-3}{2}$. From Table \ref{tab:table 3}, we have that the maximum exponent of $x_{12}$ is $1$, and from Table \ref{tab: table 4} that the maximum exponent of $x_{34}$ is $0$, for $n = 6$, and $1$ for $n\geq 8$. \\ \\
Hence,
\begin{equation*}
    \lcm(X(m))_{1 \leq m \leq n} = \frac{1}{R^{(\frac{n-2}{2})^2}}x_{12}x_{34}^{c}\prod_{\substack{l=4 \\ l \text{ even}}}^{n-2}x_{l,l+1}^{\frac{l-2}{2}}\prod_{\substack{l=5 \\ l \text{ odd}}}^{n-1}x_{l,l+1}^{\frac{l-3}{2}},
\end{equation*}
where $c = 0$ for $n = 6$, and $c = 1$ for $n \geq 8$.
\end{proof}
\begin{proof}[Proof of Theorem 3.1] Multiplying the equality \eqref{eq31} by $L$, as given in Proposition \ref{lcm}, we get that
\begin{equation}
    \sum_{m=1}^{n}\left[\frac{(-1)^{m+1}}{R^{(\frac{n-2}{2})^2}}x_{12}x_{34}^c\prod_{\substack{l=4 \\ l \text{ even}}}^{n-2}x_{l,l+1}^{\frac{l-2}{2}}\prod_{\substack{l=5 \\ l \text{ odd}}}^{n-1}x_{l,l+1}^{\frac{l-3}{2}}\delta(m)\right] = 0,
    \label{eq1001}
\end{equation}
where $c = 0$ for $n = 6$, and $c = 1$ for $n \geq 8$. \\ \\
Using Remark \ref{nap}, we can rewrite \eqref{eq1001} as 
\begin{equation}
   \sum_{m=1}^{n}\left[\frac{(-1)^{m+1}}{R^{(\frac{n-2}{2})^2}}x_{12}x_{34}^c\prod_{\substack{l=4 \\ l \text{ even}}}^{n-2}x_{l,l+1}^{\frac{l-2}{2}}\prod_{\substack{l=5 \\ l \text{ odd}}}^{n-1}x_{l,l+1}^{\frac{l-3}{2}}\frac{S(m)}{X(m)}\right] = 0.
   \label{eq2002}
\end{equation}
Now, in order to prove the theorem, we need to show that $\frac{L}{X(m)} = x(m)$. \\To do that, let us look at the quotient $\frac{L}{X(m)}.$\\ \\ Namely, from Proposition \ref{prop001}, we have that 
\begin{equation*}
    X(m) = \frac{1}{R^{(\frac{n-2}{2})^2}}\prod_{\substack{l=m+1 \\ l \text{ odd}}}^{n-1} x_{l,l+1}^{\frac{l-3}{2}}\prod_{\substack{l=4 \\ l \text{ even}}}^{m-2} x_{l,l+1}^{\frac{l-2}{2}},
\end{equation*}
for $m$ even, $2 \leq m \leq n$, so it follows that 
\begin{equation}
    \frac{L}{X(m)} = \frac{x_{12}x_{34}^c\prod_{\substack{l=4 \\ l \text{ even}}}^{n-2}x_{l,l+1}^{\frac{l-2}{2}}\prod_{\substack{l=5 \\ l \text{ odd}}}^{n-1}x_{l,l+1}^{\frac{l-3}{2}}}{\prod_{\substack{l=m+1 \\ l \text{ odd}}}^{n-1}x_{l,l+1}^{\frac{l-3}{2}}\prod_{\substack{l=4 \\ l \text{ even}}}^{m-2}x_{l,l+1}^{\frac{l-2}{2}}},
    \label{kol}
\end{equation}
for $m$ even, $2 \leq m \leq n$. \\ \\ 
Equality \eqref{kol} can be rewritten as:
\begin{equation*}
    \frac{L}{X(m)} = x_{12}x_{34}^c \frac{\prod_{\substack{l=4 \\ l \text{ even}}}^{n-2}x_{l,l+1}^{\frac{l-2}{2}}}{\prod_{\substack{l=4 \\ l \text{ even}}}^{m-2}x_{l,l+1}^{\frac{l-2}{2}}} \cdot \frac{\prod_{\substack{l=5 \\ l \text{ odd}}}^{n-1}x_{l,l+1}^{\frac{l-3}{2}}}{\prod_{\substack{l=m+1 \\ l \text{ odd}}}^{n-1}x_{l,l+1}^{\frac{l-3}{2}}}.
\end{equation*}
If $m \neq n$ and $m \neq 2,4$, we have that
\begin{equation}
    \frac{\prod_{\substack{l=4 \\ l \text{ even}}}^{n-2}x_{l,l+1}^{\frac{l-2}{2}}}{\prod_{\substack{l=4 \\ l \text{ even}}}^{m-2}x_{l,l+1}^{\frac{l-2}{2}}} = \prod_{\substack{l=m \\ l \text{ even}}}^{n-2}x_{l,l+1}^{\frac{l-2}{2}},
    \label{prva}
\end{equation}
and
\begin{equation}
\frac{\prod_{\substack{l=5 \\ l \text{ odd}}}^{n-1}x_{l,l+1}^{\frac{l-3}{2}}}{\prod_{\substack{l=m+1 \\ l \text{ odd}}}^{n-1}x_{l,l+1}^{\frac{l-3}{2}}} = \prod_{\substack{l=5 \\ l \text{ odd}}}^{m-1}x_{l,l+1}^{\frac{l-3}{2}},
\label{druga}
\end{equation}
so it follows that
\begin{equation*}
    \frac{L}{X(m)} = x_{12}x_{34}^c\prod_{\substack{l=m \\ l \text{ even}}}^{n-2}x_{l,l+1}^{\frac{l-2}{2}}\prod_{\substack{l=5 \\ l \text{ odd}}}^{m-1}x_{l,l+1}^{\frac{l-3}{2}},
\end{equation*}
where $c = 0$ for $n = 6$, and $c = 1$ for $n \geq 8$. \\ \\
Using the convention that 
\begin{equation}
    \prod_{k=a}^{b}A = 1,
    \label{conv}
    \end{equation}
    if $a < b$, one easily checks that \eqref{prva} and \eqref{druga} hold for $m \in \{n,2,4\}$ as well, which proves the theorem for even $m$.
\\ \\
Now, recall that, for odd $m$, $1 \leq m \leq n-1$, 
\begin{equation*}
    X(m) = \frac{1}{R^{(\frac{n-2}{2})^2}} x_{12}^a \prod_{\substack{l=m+2 \\ l \text{ odd}}}^{n-1}x_{l,l+1}^{\frac{l-3}{2}}\prod_{\substack{l=3 \\ l \text{ odd}}}^{m-4}x_{l,l+1}\prod_{\substack{l=4 \\ l \text{ even}}}^{m-3}x_{l,l+1}^{\frac{l-2}{2}},
    \end{equation*}
  where $a = 0$ for $m = 1$ and $m = 3$, and $a = 1$ for $m \geq 5$. \\
  \\ It follows that, for $m$ odd, $ 1 \leq m \leq n-1$, and $n > 6$,
  \begin{equation}
  \begin{split}
\frac{L}{X(m)} &= \frac{x_{12}x_{34}\prod_{\substack{l=4 \\ l \text{ even}}}^{n-2}x_{l,l+1}^{\frac{l-2}{2}}\prod_{\substack{l=5 \\ l \text{ odd}}}^{n-1}x_{l,l+1}^{\frac{l-3}{2}}}{x_{12}^a\prod_{\substack{l=m+2 \\ l \text{ odd}}}^{n-1}x_{l,l+1}^{\frac{l-3}{2}}\prod_{\substack{l=3 \\ l \text{ odd}}}^{m-4}x_{l,l+1}\prod_{\substack{l=4 \\ l \text{ even}}}^{m-3}x_{l,l+1}^{\frac{l-2}{2}}},
\end{split}
\label{kol1}
  \end{equation}
  where $a = 0$ for $m = 1$ and $m = 3$, and $a = 1$ for $m \geq 5$. \\ \\
  If $7 \leq m < n-1$, we have that
  \begin{equation*}
      \prod_{\substack{l=4 \\ l \text{ even}}}^{n-2}x_{l,l+1}^{\frac{l-2}{2}} = \prod_{\substack{l=4 \\ l \text{ even}}}^{m-3}x_{l,l+1}^{\frac{l-2}{2}}\prod_{\substack{l=m-1 \\ l \text{ even}}}^{n-2}x_{l,l+1}^{\frac{l-2}{2}},
  \end{equation*}
  and
  \begin{equation*}
      \prod_{\substack{l=5 \\ l \text{ odd}}}^{n-1}x_{l,l+1}^{\frac{l-3}{2}} = \prod_{\substack{l=5 \\ l \text{ odd}}}^{m}x_{l,l+1}^{\frac{l-3}{2}}\prod_{\substack{l=m+2 \\ l \text{ odd}}}^{n-1}x_{l,l+1}^{\frac{l-3}{2}},
  \end{equation*}
  so 
  \begin{equation}
  \begin{split}
      \frac{L}{X(m)} &= \frac{x_{34}\prod_{\substack{l=m-1 \\ l \text{ even}}}^{n-2}x_{l,l+1}^{\frac{l-2}{2}}\prod_{\substack{l=5 \\ l \text{ odd}}}^{m}x_{l,l+1}^{\frac{l-3}{2}}}{\prod_{\substack{l=3 \\ l \text{ odd}}}^{m-4}x_{l,l+1}} = \frac{\prod_{\substack{l=m-1 \\ l \text{ even}}}^{n-2}x_{l,l+1}^{\frac{l-2}{2}}\prod_{\substack{l=5 \\ l \text{ odd}}}^{m}x_{l,l+1}^{\frac{l-3}{2}}}{\prod_{\substack{l=5 \\ l \text{ odd}}}^{m-4}x_{l,l+1}} \\ &=  \frac{x_{m,m+1}^{\frac{m-3}{2}}x_{m-2,m-1}^{\frac{m-5}{2}}\prod_{\substack{l=m-1 \\ l \text{ even}}}^{n-2}x_{l,l+1}^{\frac{l-2}{2}}\prod_{\substack{l=5 \\ l \text{ odd}}}^{m-4}x_{l,l+1}^{\frac{l-3}{2}}}{\prod_{\substack{l=5 \\ l \text{ odd}}}^{m-4}x_{l,l+1}} \\ &= x_{m,m+1}^{\frac{m-3}{2}}x_{m-2,m-1}^{\frac{m-5}{2}}\prod_{\substack{l=m-1\\ l \text{ even}}}^{n-2}x_{l,l+1}^{\frac{l-2}{2}}\prod_{\substack{l=5 \\ l \text{ odd}}}^{m-4}x_{l,l+1}^{\frac{l-5}{2}},
      \end{split}
      \label{eq67}
  \end{equation}
  where, in case $m = 7$, for establishing the second and the third equality sign, we use the convention \eqref{conv}. \\ \\
For $n>6$, the quotient $\frac{L}{X(m)}$ is the same for $m=1$ and $m = 3$, and is given by
\begin{equation}
    \frac{L}{X(m)} = \frac{x_{12}x_{34}\prod_{\substack{l=4 \\ l \text{ even}}}^{n-2}x_{l,l+1}^{\frac{l-2}{2}}\prod_{\substack{l=5 \\ l \text{ odd}}}^{n-1}x_{l,l+1}^{\frac{l-3}{2}}}{\prod_{\substack{l=5 \\ l \text{ odd}}}^{n-1}x_{l,l+1}^{\frac{l-3}{2}}} = x_{12}x_{34}\prod_{\substack{l=4 \\ l \text{ even}}}^{n-2}x_{l,l+1}^{\frac{l-2}{2}}.
    \label{kol2}
\end{equation}
For $n>6$ and $m = 5$, we have that
\begin{equation}
   \frac{L} {X(m)} = \frac{x_{12}x_{34}\prod_{\substack{l=4 \\ l \text{ even}}}^{n-2}x_{l,l+1}^{\frac{l-2}{2}}\prod_{\substack{l=5 \\ l \text{ odd}}}^{n-1}x_{l,l+1}^{\frac{l-3}{2}}}{x_{12}\prod_{\substack{l=7 \\ l \text{ odd}}}^{n-1}x_{l,l+1}^{\frac{l-3}{2}}} = x_{34}x_{56}\prod_{\substack{l=4 \\ l \text{ even}}}^{n-2}x_{l,l+1}^{\frac{l-2}{2}}.
    \label{kol3}
\end{equation}
It is an easy check that 
\begin{equation}
    \frac{L}{X(m)} = x_{n-3,n-2}^{\frac{n-6}{2}}x_{n-2,n-1}^{\frac{n-4}{2}}x_{n-1,n}^{\frac{n-4}{2}}\prod_{\substack{l=5 \\ l \text{ odd}}}^{n-5}x_{l,l+1}^{\frac{l-5}{2}},
    \label{eq70}
\end{equation}
for $n>6$ and $m = n-1$. \\ \\
Now, the statement of the theorem for $m$ odd, $n > 6$, follows from \eqref{eq67}, \eqref{kol2}, \eqref{kol3} and \eqref{eq70}, and the proof for the case $m$ odd, $n = 6$, can be conducted in a similar way.
\end{proof}
\begin{example}
Let $P =(A_{1}, A_{2}, A_{3}, A_{4}, A_{5}, A_{6}, A_7, A_8)$ be a cyclic $8$-gon, whose vertices are ordered anticlockwise. Then, by Theorem \ref{impthm}, we have that the entries of the corresponding polygonal Heronian frieze satisfy the following polynomial relation:
\begin{equation*}
\begin{split}
    & x_{12}x_{34}x_{45}x_{67}^2S_{234}S_{256}S_{278}S_{356}S_{378}S_{456}S_{478}S_{578}S_{678}  \\  & \qquad \qquad -x_{12}x_{34}x_{45}x_{67}^2S_{134}S_{156}S_{178}S_{356}S_{378}S_{456}S_{478}S_{578}S_{678}  \\ & \qquad \qquad + x_{12}x_{34}x_{45}x_{67}^2S_{124}S_{256}S_{156}S_{278}S_{178}S_{456}S_{478}S_{578}S_{678} \\ & \qquad \qquad -x_{12}x_{34}x_{45}x_{67}^2S_{123}S_{156}S_{178}S_{256}S_{278}S_{356}S_{378}S_{578}S_{678}  \\ & \qquad \qquad +x_{34}x_{56}x_{45}x_{67}^2S_{346}S_{478}S_{378}S_{278}S_{678}S_{123}S_{124}S_{126}S_{178}  \\ & \qquad \qquad -x_{12}x_{34}x_{56}x_{67}^2S_{123}S_{145}S_{178}S_{245}S_{278}S_{345}S_{378}S_{478}S_{578}  \\ & \qquad \qquad +x_{56}x_{78}^2x_{67}^2S_{568}S_{123}S_{126}S_{128}S_{145}S_{245}S_{345}S_{346}S_{348}  \\ & \qquad \qquad -x_{12}x_{34}x_{56}x_{78}^2S_{123}S_{145}S_{167}S_{245}S_{267}S_{345}S_{367}S_{467}S_{567} = 0.
    \end{split}
\end{equation*}
\end{example}
\noindent
\textbf{Acknowledgement}:
This paper was written while the author was a PhD student at the University of Leeds, funded by the University of Leeds and supervised by Bethany Rose Marsh. The author is also thankful to Francesca Fedele for many helpful discussions.

\Addresses
\end{document}